\documentclass[12pt]{amsart}

\usepackage[pdfauthor   = {Lukas\ Kuehne\ and\ Geva\ Yashfe},
            pdftitle    = {On\ entropic\ and\ almost\ multilinear\ representability\ of\ matroids},
            pdfsubject  = {},
            pdfkeywords = {matroids;\ entropic matroids;\ undecidability;},
            bookmarks=true,
            bookmarksopen=true,
            pagebackref=true,
            hyperindex=true,
            colorlinks=true,
            linkcolor=blue,
            citecolor=blue,
            filecolor=blue,
            urlcolor=blue,
            ]{hyperref}

\usepackage[utf8]{inputenc}
\usepackage[T1]{fontenc}
\usepackage{comment}
\usepackage[a4paper, includeheadfoot, margin=2.81cm, bottom=1cm]{geometry} 

\usepackage{times}
\usepackage{mathrsfs}
\usepackage{mathtools}
\usepackage{latexsym}
\usepackage{amssymb}
\usepackage{amsthm}
\usepackage{epsfig}
\usepackage{colortbl}
\usepackage[all]{xy}
\usepackage{fancyvrb}
\usepackage{graphicx}
\usepackage{subcaption}
\usepackage[font=small,labelfont=bf]{caption}
\usepackage[dvipsnames]{xcolor}
\usepackage[shortlabels]{enumitem}
\usepackage{blkarray, bigstrut}
\usepackage{xparse}
\usepackage[bottom]{footmisc}

\usepackage{wrapfig}
\usepackage{tikz}
\usetikzlibrary{shapes,arrows,matrix,backgrounds,positioning,plotmarks,calc,patterns,
decorations.shapes,
decorations.fractals,
decorations.markings,
decorations.pathreplacing,
decorations.pathmorphing,
decorations.text,
arrows.meta
}
\usepackage{tikz-cd}
\usepackage[colorinlistoftodos,shadow]{todonotes}
\usepackage{bm}

\usepackage{multirow}
\usepackage{mdwlist}
\usepackage{stmaryrd}
\usepackage{mathdots}

\usepackage[toc,page]{appendix}
\usepackage{float}

\usepackage{bbold}
\usepackage{bigfoot}

\usepackage[sort&compress,capitalise]{cleveref}

\crefname{exmp}{Example}{Examples}
\crefname{prop}{Proposition}{Propositions}

\usepackage[linesnumbered,commentsnumbered,ruled,vlined]{algorithm2e}

\SetCommentSty{mycommfont}

\newtheoremstyle{mytheoremstyle} 
    {5pt}                    
    {5pt}                    
    {\itshape}                   
    {\parindent}                           
    {\bf}                   
    {.}                          
    {.5em}                       
    {}  

\theoremstyle{mytheoremstyle}

\newtheorem{theorem}{Theorem}[section]

\newtheorem{lemm}[theorem]{Lemma}
\newtheorem{prop}[theorem]{Proposition}
\newtheorem{coro}[theorem]{Corollary}
\newtheorem{prob}[theorem]{Problem}
\newtheorem{construction}[theorem]{Construction}
\newtheorem{notation}[theorem]{Notation}

\makeatletter
\@addtoreset{claims}{section}
\makeatother

\newtheoremstyle{mytdefintionstyle} 
    {5pt}                    
    {5pt}                    
    {\rm}                   
    {\parindent}                           
    {\bf}                   
    {.}                          
    {.5em}                       
    {}  

\theoremstyle{remark}
\newtheorem{rmrk}[theorem]{Remark}

\theoremstyle{mytdefintionstyle}
\newtheorem{defn}[theorem]{Definition}

\newtheoremstyle{exmp_contd}
    {5pt}                    
    {5pt}                    
    {\rm}                   
    {\parindent}                           
    {\bf}                   
    {.}                          
    {.5em}                       
    {\thmname{#1}\ \thmnumber{ #2}\thmnote{#3}\ (continued)}  
\theoremstyle{exmp_contd}

\DeclareMathOperator{\img}{im}
\DeclareMathOperator{\spa}{span}

\newcommand\A{\mathcal{A}}
\newcommand\C{\mathbb{C}}

\newcommand\F{\mathbb{F}}

\newcommand{\Q}{\mathbb{Q}}

\newcommand\R{\mathbb{R}}

\newcommand{\Z}{\mathbb{Z}}
\newcommand\N{\mathbb{N}}
\renewcommand\phi{\varphi}
\DeclareMathOperator\rk{rk}

\DeclareMathOperator\id{id}

\definecolor{darkgray}{rgb}{0.3,0.3,0.3}
\definecolor{LightGray}{gray}{0.9}

\newcommand{\topstrut}[1][1.2ex]{\setlength\bigstrutjot{#1}{\bigstrut[t]}}
\newcommand{\botstrut}[1][0.9ex]{\setlength\bigstrutjot{#1}{\bigstrut[b]}}

\DeclareDocumentEnvironment{bheadmatrix}{O{}O{*{10}{c}}}{%
	\begin{blockarray}{#2}
		#1 \\[-0.8ex]
		\begin{block}{[#2]}
			\topstrut}%
		{%
			\botstrut\\
		\end{block}
	\end{blockarray}
}%

\definecolor{darkgreen}{rgb}{0.008,0.617,0.067}
\definecolor{brown}{rgb}{0.6,0.4,0.2}

\newif\ifjournalversion

\author{Lukas K\"uhne}
\address{Fakult\"at f\"ur Mathematik, Universit\"at Bielefeld, Bielefeld, Germany}
\email{\href{mailto:Lukas Kuehne<lukas.kuehne@math.uni-bielefeld.de>}{lukas.kuehne@math.uni-bielefeld.de}}

\author{Geva Yashfe}
\address{Einstein Institute of Mathematics, The Hebrew University of Jerusalem, Giv’at Ram, Jerusalem, 91904, Israel}
 \email{\href{mailto:Geva Yashfe <geva.yashfe@mail.huji.ac.il>}{geva.yashfe@mail.huji.ac.il}}

\begin{document}

\title{On entropic and almost multilinear representability of matroids}
\begin{abstract}
This article studies two notions of generalized matroid representations motivated by algorithmic information theory and cryptographic secret sharing. The first (entropic representability) involves discrete random variables, while the second (almost-multilinear representability) deals with approximate subspace arrangements. In both cases, we prove that determining whether an input matroid has such a representation is undecidable. Consequently, the conditional independence implication problem is also undecidable, providing an independent answer to a question posed by Geiger and Pearl, recently resolved by Cheuk Ting Li.
These problems are also closely related to characterizing achievable rates in network coding and constructing secret sharing schemes. For example, another corollary of our work is that deciding whether an access structure admits an ideal secret sharing scheme is undecidable. Our approach reduces undecidable problems from group theory to matroid representation problems. Specifically, we reduce the uniform word problem for finite groups to entropic representability and the word problem for sofic groups to almost-multilinear representability. A key part of this reduction involves modifying group presentations into forms where linear representations are generic in an appropriate sense when restricted to the generating set.
\end{abstract}

\thanks{L.K. was supported by the Studienstiftung des deutschen Volkes, the ERC StG 716424 - CASe and the Deutsche Forschungsgemeinschaft (DFG, German Research Foundation) -- SFB-TRR 358/1 2023 -- 491392403 and SPP 2458 -- 539866293.
	G.Y. was supported by ERC StG 716424 - CASe, by ERC COG 101045750 HodgeGeoComb and by ISF grant 1050/16.}


\maketitle

\tableofcontents

\section{Introduction}\label{sec:intro}
\subsection{Main results}
A \emph{matroid} is a combinatorial abstraction of linear independence in vector spaces and forests in graphs.
Classically, a matroid is said to be representable over a field if there exists a set of vectors in some vector space over that field such that the subsets of linearly independent vectors are exactly the independent subsets of the matroid.
This article investigates entropic and multilinear representations.
\subsubsection{Entropic matroids}

\begin{prob}
	The \emph{entropic matroid representation problem} asks the following:
	\begin{description}
		\item[Instance] A matroid $M$ on a finite ground set $E$ with rank function $r$.
		\item[Question] Does there exist a family of discrete random variables $\{X_e\}_{e\in E}$ and a positive scalar $\lambda$ such that for all $A \subseteq E$ the joint entropy $H(X_A)$ of the variables $\{X_a\}_{a \in A}$ equals $\lambda \cdot r(A)$?
	\end{description}
\end{prob}

Matroids for which the answer is positive
are called \emph{entropic}.
The class of entropic matroids contains the ones that are representable over a field (also called linear matroids) and multilinear matroids.
Entropic matroids possibly go back to Fujishige~\cite{Fuj78} and these representations are equivalent to matroid representations by partitions~\cite{Mat99} and almost affine codes~\cite{SA98}.

The first main result of this article is the following:
\begin{theorem}\label{thm:entropic}
	The entropic matroid representation problem is algorithmically undecidable.	
\end{theorem}
(This is restated and proved as \cref{thm:undecidable_entropic})
In contrast, representability over some field can be decided using Gr\"obner bases~\cite[Thm. 6.8.9]{Oxl11}.
Generalized matrix representability over a division ring and multilinear representability are also undecidable~\cite{KPY20,KY19}.

Entropic matroids are related to ideal secret sharing schemes:
In the theory of secret sharing schemes one wants to distribute shares of a secret amongst a number of participants.
The goal is that only certain authorized subsets of the participants can recover the secret by combining their shares, while other subsets of the participants can recover no information about the secret. See~\cite{Stin92,secret_sharing_padro,secret_sharing_beimel} for detailed exposition.
The family of subsets of the participants that can jointly recover the secret is called the \emph{access structure}.

A secret sharing scheme is \emph{ideal} if the size of the share given to each participant equals the size of the secret. Brickell and Davenport observed that the access structure of an ideal secret sharing scheme determines a matroid, and called matroids arising in such a way \emph{secret sharing matroids}~\cite{BD91}.
These are the same as the entropic matroids~\cite{Mat99}.
Martin extended this bijection to connected monotone access structures that potentially don't admit an ideal secret sharing scheme~\cite{Mar91}, and Seymour proved that the Vam\'os matroid is not a secret sharing matroid~\cite{Sey92}.

Martin asked which connected monotone access structures admit an ideal secret sharing scheme~\cite{Mar91}.
\Cref{thm:entropic} show that this question is undecidable.

\subsubsection{Almost multilinear matroids} Almost-multilinear matroids are matroids approximately representable by subspace arrangements. See below for a precise definition.

\begin{prob}
	The \emph{almost multilinear matroid representation problem} asks the following:
	\begin{description}
		\item[Instance] A matroid $M$ on a finite ground set $E$ with rank function $r$ and a field~$\F$.
		\item[Question] Is it true that for every $\varepsilon > 0$ there exists a vector space $V$ over $\F$ together with a collection of subspaces $\{W_e\}_{e\in E}$ and a $c\in \N$ such that
		\[
			\max_{S\subseteq E}\left|r\left(S\right)-\frac{1}{c} \dim\left(\sum_{e\in S}W_{e}\right)\right|<\varepsilon?
		\]
	\end{description}
\end{prob}

Matroids for which this problem has a positive answer are called \emph{almost multilinear}.
This class generalizes the class of linear and multilinear matroids and is defined analogously to the class of almost entropic matroids studied by Mat\'{u}\v{s}~\cite{Mat07,Mat19}.
Almost multilinear matroids are elements of the closure of the cone of realizable polymatroids defined by Kinser~\cite{Kin09}.
Our second main result of the article is the following.

\begin{theorem}\label{thm:almost}
	The almost multilinear matroid representation problem is algorithmically undecidable.
\end{theorem}

Multilinear matroids found applications to network coding capacity:
In~\cite{ElR10}, El Rouayheb et al.\ constructed linear network capacity problems equivalent to multilinear matroid representability. 
Our previous result in~\cite{KY19} implies that the question whether an instance of the network coding problem has a  linear vector coding solution is undecidable.
\Cref{thm:almost} implies that it is also undecidable whether an instance of the network coding problem has an approximate linear vector coding solution.

A natural extension of both theorems is the question whether almost entropic representability is also undecidable.
This will be shown to be the case in the upcoming paper~\cite{almost_entropic}, which crucially relies on our work here for the almost-multilinear case.

\subsection{Conditional independence implications}
Given a finite ground set $E$, a \emph{conditional independence (CI) statement} is a triple $(A\perp B\mid C)$ of subsets $A,B,C\subseteq E$ which encodes the statement ``$A$ is independent from $B$ given $C$''.
We say that a family of discrete random variables $\{X_e\}_{e\in E}$ \emph{realizes} a CI statement $(A\perp B\mid C)$ if $X_{A}$ and $X_{B}$ are probabilistically independent given $X_{C}$. Here $X_A$ is the random variable given by the tuple $(X_a)_{a \in A}$, so that its distribution is the joint distribution of variables with indices in $A$.

\begin{prob}
	The \emph{conditional independence implication problem (CII)} is:
	\begin{description}
		\item[Instance] A set $\mathcal A$ of CI statements on a finite ground set $E$ and a CI statement~$c$.
		\item[Question] Does
		\[
		\bigwedge_{A\in \mathcal A} A \Rightarrow c
		\]
		hold for every family $\{X_e\}_{e\in E}$ of discrete random variables?
		In other words, is it true that whenever a family $\{X_e\}_{e\in E}$ of discrete random variables realizes all CI statements in $\A$ it also realizes the CI statement $c$.
	\end{description}
\end{prob}

In the literature, the sets appearing in CI statements are sometimes defined to be pairwise disjoint.
In this paper, we do not make this assumption but note that both formulations are equivalent as shown by Cheuk Ting Li~\cite{Li21}.

In the 1980s, Pearl and Paz conjectured that there exists a finite set of axioms characterizing all valid CI implication statements~\cite{PP86}.
This conjecture was later refuted by Studen\'y~\cite{Stu90}.
Subsequently, Geiger and Pearl proved that the CII problem is decidable under certain conditions on the CI statements and asked whether it is undecidable in general~\cite{GP93}.
Partial results concerning the CII problem were obtained in~\cite{NGSG13,Li21} and it was shown in~\cite{KKNS20} that the CII problem is co-recursively enumerable.
Recently Cheuk Ting Li showed that the CII problem is undecidable~\cite{Li22}.

An oracle to decide the CII problem can also decide the EMR problem.
Therefore we obtain a second independent solution of the long-standing CII problem.
\begin{coro}
	The conditional independence implication (CII) problem is algorithmically undecidable.
\end{coro}

\subsection{Related work}
We attempt to give a concise summary of that part of the literature that is most relevant to this paper, and apologize for any omissions.

Very recently Cheuk Ting Li proved that the conditional independence implication problem is undecidable, as well as that the networking coding problem is undecidable \cite{Li22}. His work became available very late in our writing. The methods used in both papers are related to each other, and also to the methods of \cite{KY19}: all three papers reduce a representability problem to the uniform word problem for finite groups. The similarity in methods seems to end there: the proof in \cite{Li22} uses different (though related) combinatorial configurations of random variables, and is significantly shorter than ours. We do not know whether it can be used to prove that entropic representability of matroids is undecidable (and thus be applied to show that it is undecidable whether an access structure admits an ideal secret scheme). It also does not cover approximate results like almost-multilinear representability.

Multilinear representations of Dowling geometries were studied by Beimel, Ben-Efraim, Padr\'{o}, and Tyomkin in~\cite{BBP14}.
This work was extended by Ben-Efraim and Mat\'{u}\v{s}  to entropic matroids~\cite{MB20} building on Mat\'{u}\v{s}' earlier work on these matroids in~\cite{Mat99}.
We previously used partial Dowling geometries to prove that the representability problem of multilinear matroids is undecidable~\cite{KY19}, where these matroids were called ``generalized Dowling geometries''.
With Rudi Pendavingh, we used more general von Staudt constructions to compare the multilinear matroid representations with representations over division rings~\cite{KPY20}.
Almost entropic matroids featured prominently in Mat\'{u}\v{s}' recent article where he proved that algebraic matroids are almost entropic~\cite{Mat19}.

\subsection{Methods and structure of the article}

We first sketch the structure of the main argument and then describe the paper section by section. Undefined terms can be found in the preliminaries (\cref{sec:preliminaries}).

The basic idea is that given a finitely presented group $G = \langle S \mid R \rangle$ and one of the generators $s \in S$, we construct a finite family of matroids $\mathcal{M}$ (in an explicit, computable way). The construction is different for the entropic and for the almost-multilinear case. It satisfies:
\begin{itemize}
	\item At least one of the matroids $M \in \mathcal{M}$ is entropic if and only if $G$ has a finite quotient in which $s$ maps to a nontrivial element.
	\item If $G$ is a sofic group, then at least one of the matroids $M \in \mathcal{M}$ is almost-multilinear if and only if $s$ is nontrivial in $G$.
\end{itemize}
See \cref{fig:iff_schematic}.

\begin{center}
	\begin{figure}[hbt]
		\begin{tikzpicture}[
			box/.style = {rectangle, draw=blue, fill=blue!10, rounded corners, minimum width=4cm, minimum height=1cm, align=center},
			arrow/.style = {thick, ->, >=Stealth},
			imply/.style={double distance=2pt, -Implies, thick}
			]			
			\node[box] (nontrivial) { 
				$s \neq e$ in some finite quotient of $G$};
			\node[box, right=2.5cm of nontrivial] (representable) {	Some $M \in \mathcal{M}$ is entropic };
			
			\draw[imply, bend right=20] (representable) to node[below] {(A)} (nontrivial);
			\draw[imply, bend right=20] (nontrivial) to node[above] {(B)} (representable);
			
		\node[box] (nontrivial) { 
			$s \neq e$ in some finite quotient of $G$};
		\node[box, right=2.5cm of nontrivial] (representable) {	Some $M \in \mathcal{M}$ is entropic };
				
			\node[box, below=1.5cm of nontrivial] (nontrivial2) { 
				$s \neq e$ in the sofic group $G$};
			\node[box, right=2.5cm of nontrivial2] (representable2) {Some $M \in \mathcal{M}$ is almost-multilinear};
			
			\draw[imply, bend right=20] (representable2) to node[below] {(A)} (nontrivial2);
			\draw[imply, bend right=20] (nontrivial2) to node[above] {(B)} (representable2);
			
			\node[box, below=1.5cm of nontrivial] (nontrivial2) { 
				$s \neq e$ in the sofic group $G$};
			\node[box, right=2.5cm of nontrivial2] (representable2) {Some $M \in \mathcal{M}$ is almost-multilinear};
			
		\end{tikzpicture}
		\caption{The four implications described above.
			The first diagram shows the two implications used in the proof that the recognition of entropic matroids is undecidable.
		  The second diagram is used for the analogous statement for almost multilinear matroids.}
		\label{fig:iff_schematic}
	\end{figure}
\end{center}

Hence, the so-called uniform word problem for finite groups can be reduced to the entropic matroid representation problem. In the same way, the word problem for torsion-free sofic groups can be reduced to the almost-multilinear representation problem. Both of these word problems are known to be undecidable (see \cref{sec:uwpfg,sec:sofic}).

Schematically, the construction of $\mathcal{M}$ in the entropic case  is shown in \Cref{fig:construction_of_M}.
\begin{center}
\begin{figure}[htb]
\begin{tikzpicture}[scale=0.8, transform shape,
	box/.style = {rectangle, draw=blue, fill=blue!10, rounded corners, minimum width=4cm, minimum height=1cm, align=center},
	arrow/.style = {thick, ->, >=Stealth},
	imply/.style={double distance=2pt, -Implies, thick}
	]
	
	\node[box] (input) {Input: $G = \langle S \mid R \rangle$ and $s \in S$};
	\node[box, below=1.5cm of input] (scrambling) {Construct the augmented scrambling \\ $G'' = \langle S'' \mid R'' \rangle$ of $G$};
	\node[box, below=1.5cm of scrambling, 
	text width=6cm] (matroids) {Construct the set $\mathcal{M}=\mathcal{M}_{S'',R''}$ of matroids subordinate to the presentation};
	
	\draw[arrow] (input) -- node[text width=1.6cm, midway, left]{\Cref{sec:scrambling}} (scrambling);
	\draw[arrow] (scrambling) -- node[text width=1.6cm, midway, left]{\Cref{sec:dowling}} (matroids); 
\end{tikzpicture}
\caption{Construction of the finite matroid family $\mathcal{M}$.}
\label[figure]{fig:construction_of_M}
\end{figure}
\end{center}

The implications labelled (A) in \Cref{fig:iff_schematic} are relatively straightforward, and do not require the scrambling construction. For entropic matroids and the uniform word problem for finite groups, this is proved in \Cref{sec:entropic_reps}. For almost-multilinear matroids and the word problem for sofic groups, this is proved in \cref{thm:almost_multilinear_to_group}.

The implication (B) in the entropic case is somewhat more delicate and (together with the construction of augmented scramblings, which was designed for this purpose) takes up much of the paper. See \cref{sec:scrambling}. In the proof it is useful to have some linear algebraic tools, so even for the statement on entropic matroids we work specifically with multilinear representations (multilinear matroids are entropic, so finding a multilinear representation suffices). In the almost-multilinear case, implication (B) is relatively short: the results of \cite{optimal_linear_sofic_approx} are available to replace scrambling in the approximate setting. (See \cref{lem:linear_sofic_amplification}, and note that this requires the additional assumption that our group is torsion-free.)

The paper is organized as follows.
\begin{enumerate}
	\item We start by recalling definitions and setting up basic notions and notation in~\Cref{sec:preliminaries}.
	\item Given a finite group, one can define an associated matroid, the so-called \emph{Dowling geometry}, whose representations are closely related to the representation theoretic properties of the group~\cite{Dow73}.
	We work with a extension of this construction to finitely presented groups which we present in~\Cref{sec:dowling}.
	We call the resulting matroids \emph{partial Dowling geometries}.
	We first used them in~\cite{KY19,KPY20}.
	They are special cases of frame matroids as studied by Zaslavsky in~\cite{Zas03}.
	The idea is to encode group presentations via the von Staudt constructions.
	\item After defining entropic matroids in~\Cref{sec:probability_space_reps} as well as the essentially equivalent (but more convenient) notion of probability space representations, we prove in ~\Cref{sec:entropic_reps} that the existence of an entropic representation of the partial Dowling geometry of a symmetric triangular presentation $\langle S\mid R\rangle$ implies the existence of a group homomorphism from $\langle S\mid R\rangle$ to a finite group such that images of some elements are nontrivial.
	\item In \cref{sec:multilinear} we discuss multilinear matroid representations and introduce an equivalent (but more convenient) notion we call \emph{vector space representations}, as part of our preparation for proving implication (B) of \Cref{fig:iff_schematic}.
	\item In \cref{sec:scrambling} we introduce the scrambling and augmentation constructions and prove that the resulting groups have linear representations with desirable properties.
	\item In \cref{sec:entropic_undecidability} we put together our tools to show that the entropic representation problem is undecidable.
	\item In \cref{sec:cii} we briefly discuss the conditional independence implication problem.
	\item In \cref{sec:almost_mutilinear} we discuss almost-multilinear matroids. The discussion parallels the earlier sections: first we introduce approximate vector space representations in \cref{sec:approx_vec_reps}. Then we discuss almost-multilinear representations of partial Dowling geometries in \cref{sec:approx_dowling}. In \cref{sec:almost_undecidable} we prove the almost-multilinear representation problem is undecidable.
\end{enumerate}

\section*{Acknowledgments}

We would like to thank Tobias Boege, Michael Dobbins, Zlil Sela, and Thomas Zaslavsky for inspiring discussions.
We are grateful to Karim Adiprasito for introducing us to the topic of almost multilinear matroids.

Last but not least we are indebted to the anonymous referees for carefully reading an earlier version of the paper and their suggestion which helped to significantly improve it.

\section{Preliminaries}\label{sec:preliminaries}

\subsection{Notation for probability spaces and random variables}\label[section]{sec:rand_variable_notation}

An indexed collection of random variables on a probability space $\left(\Omega,\mathcal{F},P\right)$
consists of a set $E$, a collection of measurable spaces $\left\{ \left(\Omega_{e},\mathcal{F}_{e}\right)\right\} _{e\in E}$, and a collection of measurable functions $\left\{ X_{e}:\Omega\rightarrow\Omega_{e}\right\} _{e\in E}$. 

For convenience, we often write ``let $\left\{ X_{e}\right\} _{e\in E}$
be a collection of random variables on $\left(\Omega,\mathcal{F},P\right)$,''
and use the notation $\left\{ \left(\Omega_{e},\mathcal{F}_{e}\right)\right\} _{e\in E}$
for the codomains of the random variables without explicitly naming
them. We also denote by $\left\{ P_{e}\right\} _{e\in E}$ the probability
measures defined by $P_{e}=\left(X_{e}\right)_{*}P$. By definition
this implies that each of the transformations 
\[
X_{e}:\left(\Omega,\mathcal{F},P\right)\rightarrow\left(\Omega_{e},\mathcal{F}_{e},P_{e}\right)
\]
is measure-preserving.

Given a collection of random variables $\left\{ X_{e}\right\} _{e\in E}$
on $(\Omega,\mathcal{F},P)$ as above and a tuple $S=\left(s_{1},\ldots,s_{n}\right)$
with elements in $E$, we define a measurable space $\left(\Omega_{S},\mathcal{F}_{S}\right)$
by $\Omega_{S}=\prod_{i=1}^{n}\Omega_{s_{i}}$ and $\mathcal{F}_{S}=\bigotimes_{i=1}^{n}\mathcal{F}_{s_{i}}$
the $\sigma$-algebra generated by measurable boxes (which are the
sets $\prod_{i=1}^{n}A_{i}$ with $A_{i}\in\mathcal{F}_{s_{i}}$ for
each $i$). We then define a random variable $X_{S}:\Omega\rightarrow\Omega_{S}$
by 
\[
X_{S}\left(\omega\right)=\left(X_{s_{i}}\left(\omega\right)\right)_{i=1}^{n}.
\]
If the order is inessential, the same notation can be used if $S$ is a set.
On $\left(\Omega_{S},\mathcal{F}_{S}\right)$ we define the probability
measure $P_{S}=\left(X_{S}\right)_{*}P$, the pushforward
of $P$.

\subsection{Entropy functions of discrete random variables}

Let $\{X_e\}_{e\in E}$ be a collection of discrete random variables on $(\Omega,\mathcal{F},P)$.
For each $S\subseteq E$ we denote by $H(X_S)$ the \emph{(Shannon) entropy} of the random variable $X_S$:
\begin{equation*}\label{eq:entropy}
H(X_S)\coloneqq -\sum_{\omega\in \Omega_S} P_S( \omega ) \log P_S(\omega).
\end{equation*}
We set $H(X_S)\coloneqq \infty$ if the sum does not converge.
The base of the logarithm is irrelevant for this article; for consistency we choose to work with the base $2$ throughout.

\subsection{Matroids}

We frequently use standard terminology from matroid theory, as for instance explained in Oxley's textbook~\cite{Oxl11}.
For the reader's convenience we just briefly recall the definition of a matroid.

\begin{defn}\label{def:matroid}
	A \emph{matroid} $M=(E,r)$ is a pair consisting of a finite \emph{ground set} $E$ together with a \emph{rank function} $r:\mathcal{P}(E)\to \Z_{\ge 0}$ such that
	\begin{enumerate}
		\item\label{it:mat_a} $r(A)\le |A|$ for all $A\subseteq E$,
		\item\label{it:mat_b} $r(A)\le r(B)$ for all $A\subseteq B\subseteq E$ ($r$ is monotone), and
		\item\label{it:mat_c} $r(A\cup B)+r(A\cap B)\le r(A)+r(B) $ for all $A,B\subseteq E$ ($r$ is submodular).
	\end{enumerate}
	A pair $(E,r)$ with  $r:\mathcal{P}(E)\to \R_{\ge 0}$ satisfying~\eqref{it:mat_b} and~\eqref{it:mat_c} is called a \emph{polymatroid}.
\end{defn}

We will freely use standard matroid terminology such as independent sets, bases or flats and refer to Oxley's book for their definitions\cite{Oxl11}.
In particular, we will use that a matroid can be defined by specifying its flats.

\subsection{Matroid representations}
Matroid theory is the combinatorial study of various notions of dependence and independence, analogous to those in linear algebra.
A well-studied subclass of matroids is the class of linearly-representable matroids, in which the rank function is actually given by linear-algebraic rank:
A matroid $M=(E,r)$ is \emph{representable} over a field $\F$ if there exists a family of vectors $\{v_e\}_{e\in E}$ in a vector space over $\F$ such that $r(S)=\dim(\spa(\{v_e\}_{e\in S}))$ for all $S\subseteq E$.
The matroid $M$ is also called \emph{linear} over $\F$ in this case.

When studying any notion of matroid representability, it is desirable to be able to decide whether a given matroid is representable. 
For example, using Gr\"obner bases one can decide whether a matroid is linear over an algebraically closed field~\cite[Theorem 6.8.9]{Oxl11}.
The question of whether one can decide representability over $\Q$ is equivalent to the solvability of Diophantine equations in $\Q$~\cite{Stu87}.
This is a variant of Hilbert's tenth problem and still open.
In this paper we study generalized notions of matroid representability and the associated decision problems.

In this section we define the notions of representability that we will investigate throughout the article.

\begin{defn}[\cite{SA98}]
	A matroid $M=\left(E,r\right)$ is \emph{multilinear} over a field
	$\mathbb{F}$ if there exist an integer $c$ and a vector space $V$
	over $\mathbb{F}$ with subspaces $\left\{ W_{e}\right\} _{e\in E}$
	such that for each $S\subseteq E$
	\[
		r(S) = \frac{1}{c}\dim\left(\sum_{e\in S}W_{e}\right).
	\]
	In this case the vector space $V$ and the indexed family of subspaces
	$\left\{ W_{e}\right\} _{e\in E}$ are called a multilinear representation
	of $M$, or a representation of $M$ \emph{as a $c$-arrangement}.
	(We learned the term ``$c$-arrangement'' from~\cite{GM88}.)
	Observe that if we add the constraint $c=1$ we recover the definition
	of linear representability.
\end{defn}

Given a collection $\{X_e\}_{e\in E}$ of discrete random variables, Fujishige observed that the assignment $H:\mathcal{P}(E)\to \mathbb{R}_{\ge 0}$ given for each $S\subseteq E$ by the entropy $H(X_S)$ is a polymatroid~\cite{Fuj78}. Polymatroids arising this way are called \emph{entropic}. Subsequently entropic polymatroids were studied by various authors, and entropic matroids were defined, for instance by  Mat\'{u}\v{s}, who called them ``strongly probabilistically representable matroids'' in \cite{Mat99}:

\begin{defn}
	A matroid $M=(E,r)$ is \emph{entropic} if there exists a family $\{X_e\}_{e\in E}$ of random variables on a discrete probability space $(\Omega,\mathcal{F},P)$ and a real $\lambda>0$ such that for all subsets $S\subseteq E$
	\[
		r(S)=\lambda H(X_S).
	\]
\end{defn}

Note that in contrast to this definition but following the discussion above, a polymatroid $(E,r)$ is entropic if there exists random variables $\{X_e\}_{e\in E}$ such that $r(S)=H(X_S)$ for all subsets $S\subseteq E$ (there is no scaling factor).

We now introduce approximate notions of multilinear and entropic matroid representations.
\begin{defn}
	A  polymatroid $\left(E,r\right)$ is \emph{linear} over
	a field $\mathbb{F}$ if there exists a vector space $V$ and a collection of subspaces
	$\left\{ W_{e}\right\} _{e\in E}$ of $V$ satisfying that for all $S\subseteq E$:
	\[
	r\left(S\right)=\dim\left(\sum_{e\in S}W_{e}\right). 
	\]

	A matroid $M=(E,r)$ is \emph{almost multilinear}
	if for every $\varepsilon>0$ there exists a linear polymatroid $\left(\widetilde{E},\widetilde{r}\right)$
	and a $c\in\mathbb{N}$ such that 
	\[
	\left\Vert r-\frac{1}{c}\widetilde{r}\right\Vert _{\infty}=\max_{S\subseteq E}\left|r\left(S\right)-\frac{1}{c}\widetilde{r}\left(S\right)\right|<\varepsilon.
	\]
\end{defn}
	
	Note that we may always assume the ambient vector space $V$ of a
	linear polymatroid $\left(E,r\right)$ is finite dimensional: if the
	representation is given by the subspaces $\left\{ W_{e}\right\} _{e\in E}$
	of $V$, we may replace $V$ by $\sum_{e\in E}W_{e}$, which has dimension
	$r\left(E\right)$.

\begin{defn}[\cite{Mat07}]\label{def:almost_entropic}
	A matroid $(E,r)$ is \emph{almost entropic} if for every $\varepsilon>0$ there exists a collection of discrete random variables $\{X_e\}_{e\in E}$ on a probability space $(\Omega,\mathcal{F},P)$ and a real $\lambda>0$ such that
	\[
	\max_{S\subseteq E}\left|r\left(S\right)-\lambda H(X_S)\right|<\varepsilon.
	\]
\end{defn}

\subsection{Hamming and rank distance}
\begin{defn}
	Let $n\in\mathbb{N}$. The \emph{normalized Hamming distance} $d_\mathrm{hamm}$ is the metric on the symmetric group $S_n$ defined by
	\[
	d_\mathrm{hamm} (\sigma,\tau) \coloneqq \frac{1}{n} |\{ i\in \left[n\right]\mid\sigma(i)\neq \tau(i) \}|
	\]
	for all $\sigma,\tau\in S_n$.
\end{defn}
The normalized Hamming distance satisfies that if $\sigma,\sigma',\tau \in S_n$ then
\begin{align*} 
	&d_\mathrm{hamm}(\sigma,\sigma')=d_\mathrm{hamm}(\sigma\circ\tau,\sigma'\circ\tau)=d_\mathrm{hamm}(\tau\circ\sigma,\tau\circ\sigma').
\end{align*}
\begin{defn}
	Let $A,B\in M_{n}\left(\mathbb{F}\right)$ be matrices. Their \emph{normalized
		rank distance} is
	\[
	d_{\text{rk}}\left(A,B\right)\coloneqq\frac{1}{n}\rk\left(A-B\right).
	\]
	
	More generally, if $T_1, T_2: V \to W$ are linear maps between finite dimensional vector spaces $V,W$ over a field, define
	\[d_{\rk}(T_1,T_2) \coloneqq \frac{1}{\dim(W)}\rk(T_1-T_2), \]
	where the rank of a linear map is the dimension of its image.
\end{defn}

By abuse of notation, we denote all these functions $d_\mathrm{rk}: \mathrm{Hom}(V,W) \times \mathrm{Hom}(V,W) \to \R$ (or $M_n(\F) \times M_n(\F)$) by the same name. It will always be clear from the arguments which function we mean.

By representing maps with respect to a fixed basis, it is clear that any result on the metric $d_{\rk}$ defined on $M_n(\F)$ extends to $\mathrm{End}(W)=\mathrm{Hom}(W,W)$ for any finite dimensional vector space $W$ over $\F$ and vice versa.

It is well-known that the function  $d_{\mathrm{rk}}: M_n(\F) \times M_n(\F) \to \R$ is a metric, see e.g.,~\cite[Remark 1.3]{rankmetriccodes}. 
In particular, the function $d_{\mathrm{rk}}:\mathrm{Hom}(V,W) \times \mathrm{Hom}(V,W) \to \R$ is a metric on $\mathrm{Hom}(V,W)$.

\begin{rmrk}\label[remark]{rem:normalized_rank_metric}
	Note that $d_{\rk}$ is left- and right-invariant under composition with invertible transformations, in the sense that if $T_1,T_2 \in \mathrm{Hom}(V,W)$ and $S,Q$ are invertible linear transformations such that $S$ has domain $W$ and $Q$ has codomain $V$, then
	\[d_\mathrm{rk}(T_1,T_2) = d_\mathrm{rk}(S\circ T_1 \circ Q, S\circ T_2 \circ Q).\]
	
	If the requirement that $S,Q$ are invertible is dropped and $S: W \to U$, we obtain instead
	\[d_\mathrm{rk}(S\circ T_1 \circ Q, S\circ T_2 \circ Q) \le \frac{\dim W}{\dim U}d_\mathrm{rk}(T_1,T_2).\]
	To see this, observe that
	\[\rk(S \circ T_1 \circ Q - S\circ T_2 \circ Q) =\rk(S\circ (T_1 - T_2) \circ Q) \le \rk(T_1 - T_2).\]
	In particular, if $A,B,C \in M_n(\F)$ and $d_{\rk}(A,B) < \varepsilon$ then also $d_{\rk}(CA,CB) < \varepsilon$.
\end{rmrk}

\subsection{The uniform word problem for finite groups}
\label[section]{sec:uwpfg}
The uniform word problem for finite groups (UWPFG) is the following decision problem.
\begin{description}
	\item[Instance] A finite presentation $\langle S \mid R \rangle$ of a group $G$ and an element $w\in S$.
	\item[Task] Decide whether there exists a finite group $H$ and a homomorphism $\varphi:G\to H$ such that $w\notin\ker(\varphi)$.
\end{description}
Our undecidability results in the entropic setting rely on the following consequence of Slobodskoi's work \cite{Slo81}.
\begin{theorem}\label{thm:undecidable_grp_theory}
	The uniform word problem for finite groups is undecidable.
\end{theorem}
Slobodskoi's result is stronger: it shows that in fact the word problem for finite groups is undecidable for some specific $\langle S \mid R \rangle$ (in the notation above, it is only the word $w$ that is not fixed).

Note that this problem is semi-decidable: there exists an algorithm which halts and returns the answer whenever it is positive, and otherwise runs forever.

\subsection{Sofic groups}
\label[section]{sec:sofic}
For an introduction to sofic groups see the survey by Pestov~\cite{Pes08}.

The following is one of several equivalent definitions of sofic groups (see for instance \cite{ES06}). To see its equivalence to the characterization in \cite[Theorem 3.5]{Pes08}, one uses the amplification trick described in the proof of the same theorem.
\begin{defn}\label{def:sofic}
	A group $G$ is \emph{sofic} if for every finite $F\subseteq G$ and for every $\varepsilon >0 $ there exist an $n\in\N$ and a mapping $\theta : F \to S_n$ such that
	\begin{enumerate}
		\item If $g,h,gh\in F$ then $d_\mathrm{hamm}(\theta(g)\theta(h),\theta(gh))<\varepsilon$,
		\item If the neutral element $e_G$ is in $F$ then $d_\mathrm{hamm}(\theta(e),\id) < \varepsilon$, and
		\item If $g,h\in F$ are distinct then $d_\mathrm{hamm}(\theta(g),\theta(h))\geq 1-\varepsilon$.
	\end{enumerate}
\end{defn}

Our proof that the existence of almost multilinear matroid representations is undecidable relies on the following theorem.
\begin{theorem}\label{thm:undecidable_sofic}
	There exists a finitely presented, torsion-free sofic group with an undecidable word problem.
\end{theorem}
This follows from the standard result that a solvable group is sofic, together with the construction \cite{undecidable_torsionfree_solvable} of Baumslag, Gildenhuys, and Strebel for a finitely presented, solvable, torsion-free group with undecidable word problem. (The first construction of this general kind appeared in \cite{Kha81}, but the group constructed there has torsion.)

\subsection{Approximate representations of groups}
\label[section]{sec:approx_group_reps}
In order to study almost-multilinear Dowling geometries we need a ``linear version'' of soficity. This has been studied by Arzhantseva and P\u{a}unescu in \cite{AP12}. Our definitions are specialized to the finitely presented case and avoid metric ultraproducts.

\begin{defn}
	Let $G = \langle S \mid R \rangle$ be a finitely presented group and let $\varepsilon>0$. An $\varepsilon$-approximate representation of the presentation $\langle S \mid R \rangle$ of $G$ over a field $\mathbb{F}$ is a function
	\[\rho:S \to \mathrm{GL}_n(\mathbb{F})\] satisfying:
	\begin{enumerate}
		\item If $r = s_{i_1}^{\epsilon_1}\cdot\ldots s_{i_k}^{\epsilon_k}$ is a relator in $R$ then $d_\mathrm{rk}(I, \rho(s_{i_1})^{\epsilon_1}\cdot\ldots\cdot\rho(s_{i_k})^{\epsilon_k})<\varepsilon$ (in this case we say that $\rho$ \emph{$\varepsilon$-satisfies} $r$).
		\item If the neutral element $e_G$ is in $S$ then $d_\mathrm{rk}(\rho(e),I) < \varepsilon$.
	\end{enumerate}
	
	An $\varepsilon$-approximate representation of $\langle S \mid R \rangle$ naturally extends to all words on $S$: if $w=w(S)=s_{i_1}^{\epsilon_1} \cdot\ldots\cdot s_{i_k}^{\epsilon_k}$ is any word in $S$, we denote $\rho(w) = \rho\left(s_{i_1}\right)^{\epsilon_1} \cdot\ldots\cdot \rho\left(s_{i_k}\right)^{\epsilon_k}$.
\end{defn}

The following lemma is a direct implication of \cite[Theorem A]{optimal_linear_sofic_approx}, together with the fact that a sofic group is linear-sofic.
\begin{lemm}\label[lemma]{lem:linear_sofic_amplification}
	Let $G = \langle S \mid R \rangle$ be a finitely presented sofic group. If $G$ is torsion-free and $\mathbb{F}$ has characteristic $0$ then for all $\varepsilon>0$  there exists an $n\ge 1$ and an $\varepsilon$-approximate representation
	\[\rho:S \to \mathrm{GL}_n(\mathbb{F})\]
	satisfying in addition that $d_\mathrm{rk}(\rho(s),\rho(s')) \ge 1-\varepsilon$ whenever $s,s' \in S$ map to distinct elements of $G$.
\end{lemm}

\subsection{Finitely presented categories}
\label{sec:fp_cats}
Finitely presented categories are to finitely presented monoids as groupoids are to groups. We use these in our discussion of almost-multilinear representations.

For the following definitions see also Awodey's book \cite{Awo10}.
All our directed graphs may have multiple edges between the same pair
of vertices.
\begin{defn}[free categories] 
	The \emph{free category} on a directed
	graph $G$ is the category $\mathcal{C}\left(G\right)$ in which objects
	are the vertices of $G$, morphisms are (directed) paths in $G$,
	and composition is given by concatenating paths.
\end{defn}
\begin{defn}	
	A \emph{congruence} on a category $\mathcal{C}$ is an equivalence relation $\sim$ on the morphisms of $\mathcal{C}$ such that:
	\begin{enumerate}
		\item If $f \sim g$ then $f,g$ have the same domain and the same codomain.
		\item If $f \sim g$ then $a \circ f \circ b \sim a \circ g \circ b$ for all morphisms $a$ with codomain the domain of $f,g$ and all morphisms $b$ with domain the codomain of $f,g$.
	\end{enumerate}
\end{defn}

A congruence on $\mathcal{C}$ is precisely an equivalence relation $\sim$ on morphisms such that there is a quotient category $\mathcal{C}/\mathord{\sim}$ with the same objects as $\mathcal{C}$ and such that $\hom_{\mathcal{C}/\mathord{\sim}}(x,y) = \hom_{\mathcal{C}}(x,y)/\mathord{\sim}$ for all $x,y$ objects in $\mathcal{C}$. The composition in $\mathcal{C}/\mathord{\sim}$ is induced from that of $\mathcal{C}$, and the identity morphisms are the images through the quotient map of those in $\mathcal{C}$.

\begin{defn}[finitely presented categories]
	Let $\mathcal{C}$ be a category and let $R=\left\{ f_{i}=g_{i}\right\} _{i\in I}$
	be a set of formal expressions (``relations'') such that for each $i\in I$,
	$f_{i},g_{i}:x_{i}\rightarrow y_{i}$ are two morphisms between the
	same two objects of $\mathcal{C}$. Denote by $\sim_{R}$ the minimal
	congruence satisfying that $f_{i}\sim_{R}g_{i}$ for each $i\in I$.
	We call $\sim_{R}$ the congruence generated by the relations in $R$.
	
	A \emph{finitely presented category} is a category of the form $\mathcal{C}\left(G\right)/\mathord{\sim_R}$,
	where $G$ is a finite directed graph and $R$ is a finite set of
	relations between morphisms of $\mathcal{C}\left(G\right)$.
	
	In this situation we denote $\langle G \mid R \rangle = \mathcal{C}\left(G\right)/\mathord{\sim_R}$. As far as we are aware this notation is nonstandard, but it gives us a convenient way to refer to the finite set $R$, rather than just to the congruence $\sim_R$.
\end{defn}
\begin{rmrk}\label{rem:congruence}
	In the notation above, the congruence $\sim_R$ is precisely the equivalence relation in which $h_1,h_2$ are equivalent if and only if they may be written in the form
	\[h_1 = a_1 \circ f_{i_1} \circ a_2 \circ f_{i_2} \circ \ldots \circ a_n \circ f_{i_n} \circ a_{n+1},\]
	\[h_2 = a_1 \circ g_{i_1} \circ a_2 \circ g_{i_2} \circ \ldots \circ a_n \circ g_{i_n} \circ a_{n+1},\]
	where $n\in\mathbb{N}$ and ``$f_{i_j} = g_{i_j}$'' is a relation in $R$ for each $1 \le j \le n$.
	
	To verify this, it suffices to note that $\sim_R$ is indeed a congruence and that it contains the relations $f_i \sim_R g_i$ for all $i\in I$.
\end{rmrk}

\begin{rmrk}\label{rem:functor}
	To construct a functor $F$ from a finitely presented category $\mathcal{C}=\mathcal{C}\left(G\right)/\mathord{\sim_R}$
	into a category $\mathcal{D}$, it suffices to define $F$ on the
	objects and morphisms of $\mathcal{C}$ corresponding to vertices
	and edges of $G$, and to show that if $\varphi_{1}=\varphi_{2}$
	is a relation in $R$ then $F\left(\varphi_{1}\right)=F\left(\varphi_{2}\right)$
	in $\mathcal{D}$ (see \cite{Awo10}).
\end{rmrk}

\begin{defn}
	A groupoid is a category in which every morphism has a two-sided inverse. A finitely-presented groupoid is a finitely-presented category that happens to be a groupoid.
\end{defn}
\begin{rmrk}
	For a finitely presented category $\mathcal{C}(G)/\mathord{\sim_R}$ to be a groupoid it suffices that each generating morphism (i.e. arising from an edge of $G$) is invertible.
\end{rmrk}

\subsection{Approximate representations of groupoids}
\label{sec:approx_groupoid_reps}
In some situations it is more natural to produce approximate representations of Dowling groupoids (see \cref{sec:dowling} below) than of the corresponding groups. It is useful to have some results applicable in this situation. 

\begin{defn}
	Let $\varepsilon > 0$. An $\varepsilon$-representation $\rho$ of a finitely presented groupoid $\mathcal{C} = \langle G \mid R\rangle$ over a field $\F$ is a functor $\rho$ from $\mathcal{C}(G)$ to the category of finite dimensional $\F$-vector spaces such that
\begin{enumerate}
	\item If ``$f = g$'' is a relation in $R$ then $d_\mathrm{rk}(\rho(f),\rho(g)) < \varepsilon$.
	\item If $x,y$ are vertices in the same connected component of $G$ then $\dim \rho(x) = \dim \rho(y)$.
\end{enumerate}
\end{defn}

\begin{rmrk}
	Note that by \cref{rem:functor} it suffices to specify an approximate representation of $\langle G \mid R \rangle$ on the vertices and edges of the graph $G$.
	
	Condition (b) can be omitted more-or-less harmlessly, in the sense that approximate $\varepsilon$-``representations'' that do not satisfy it can be approximated by ones that do (with slightly larger $\varepsilon$, depending on the particular presentation). But it shortens some proofs.
\end{rmrk}
\begin{lemm}
	Let $\mathcal{C} = \langle G \mid R \rangle$ be a finitely presented groupoid. Denote the congruence generated by $R$ by $\sim_R$. Let $h_1,h_2$ be two morphisms in the free category $\mathcal{C}(G)$ that map to the same morphism of $\mathcal{C}$ (note that in particular they have the same domain and codomain). Then there exists $k\in\mathbb{N}$ such that for any $\varepsilon > 0$ and every $\varepsilon$-approximate representation $\rho$ of $\langle G \mid R \rangle$:
	\[d_\mathrm{rk}(\rho(h_1),\rho(h_2)) < k\varepsilon.\]
\end{lemm}
\begin{proof}
By \cref{rem:congruence}, we may write
\[h_1 = a_1 \circ f_{i_1} \circ a_2 \circ f_{i_2} \circ \ldots \circ a_n \circ f_{i_k} \circ a_{k+1},\]
\[h_2 = a_1 \circ g_{i_1} \circ a_2 \circ g_{i_2} \circ \ldots \circ a_k \circ g_{i_k} \circ a_{k+1},\]
where $k\in\mathbb{N}$ and ``$f_{i_j} = g_{i_j}$'' is a relation in $R$ for each $1 \le j \le k$.

Since
\[d_\mathrm{rk}(\rho(f_{i_j}), \rho(g_{i_j})) < \varepsilon\]
for each $1 \le j \le k$, the result follows by \cref{rem:normalized_rank_metric}.
\end{proof}

\begin{coro}\label[corollary]{cor:approx_different_groupoid_elements}
	Let $\mathcal{C}=\left\langle G\mid R\right\rangle$ be a finitely presented groupoid,
	let $h_1,h_2$ be morphisms in the free category $\mathcal{C}(G)$ with the same domain and codomain, and let $\alpha > 0$. If for each $\varepsilon > 0$ there exists an $\varepsilon$-approximate representation $\rho$ of $\langle G\mid R\rangle$ such that
	\[d_\mathrm{rk}(\rho(h_1),\rho(h_2)) \ge \alpha\]
	then $h_1,h_2$ map to different elements of $\mathcal{C}$.
\end{coro}
\begin{proof}
	By the previous lemma, if $h_1,h_2$ map to the same element of $\mathcal{C}$ then for all small enough $\varepsilon>0$ each $\varepsilon$-approximate representation $\rho$ of $\langle G \mid R \rangle$ satisfies $d_{\rk}(h_1,h_2) < \alpha$.
\end{proof}

\subsection{Some algebraic lemmas}\label{sec:algebraic_lemmas}

We collect some results about field theory and linear algebra for later use.
We use them in order to prove that certain matroids are (almost) multilinear.

Recall that a group $G$ is residually finite if for each $x\in G$
such that $x\neq e_{G}$ there exists a finite group $H$ and a homomorphism
$\varphi:G\rightarrow H$ such that $\varphi\left(x\right)\neq e_{H}$.
\begin{theorem}
	[Mal'cev's theorem]Let $\mathbb{F}$ be a field. A finitely generated
	subgroup of $\mathrm{GL}_{n}\left(\mathbb{F}\right)$ is residually
	finite.
\end{theorem}
For a proof see \cite{LS77} for example.

\begin{lemm}\label{lem:derangement_rep}
	Let $\mathbb{F}$ be a field, $G$ a finitely generated group, $g\in G$
	an element, and let $\rho:G\rightarrow\mathrm{GL}_{n}(\mathbb{F})$
	be a representation such that $\rho\left(g\right)\neq I_{n}$. Then
	there exists $n^{\prime}\in\mathbb{N}$ and a representation $\rho^{\prime}:G\rightarrow\mathrm{GL}_{n^{\prime}}(\mathbb{F})$
	such that for every $x\in G$ the matrix $\rho^{\prime}\left(x\right)$ is either $I_{n'}$ or the permutation matrix
	of a derangement, and $\rho'(g)\neq I_{n'}$.
\end{lemm}
\begin{proof}
	The image $\rho\left(G\right)$ is a finitely generated subgroup of
	$\mathrm{GL}_{n}(\mathbb{F})$, so it is residually finite by Mal'cev's
	theorem. 
	Therefore there exists a finite group $H$ and a homomorphism $\varphi:\rho\left(G\right)\rightarrow H$
	such that $\varphi\left(\rho\left(g\right)\right)\neq e_{H}$.
	The left action of $H$ on itself defines a permutation representation
	of $H$, and thus of $\rho\left(G\right)$ and of $G$, on a set of
	$n^{\prime}=\left|H\right|$ elements, such that any element that
	acts nontrivially acts by a derangement. Choosing a bijection between
	$H$ and the set $\left\{ 1,\ldots,n^{\prime}\right\} $ we produce
	a homomorphism $G\rightarrow S_{n^{\prime}}$ which maps each $x\in G$ to the identity or to a
	derangement (and $g$ to a derangement). There is a homomorphism $S_{n^{\prime}}\rightarrow\mathrm{GL}_{n^{\prime}}(\mathbb{F})$
	which maps each permutation to its permutation matrix. Taking $\rho^{\prime}$
	to be the composition of these homomorphisms $G\rightarrow S_{n^{\prime}}\rightarrow\mathrm{GL}_{n^{\prime}}(\mathbb{F})$
	yields the result.
\end{proof}
\begin{lemm}
	\label[lemma]{lem:derangement_matrix}Let $\mathbb{F}$ be an algebraically
	closed field of characteristic either $0$ or larger than $n$ and
	let $A\in\mathrm{GL}_{n}(\mathbb{F})$ be the permutation matrix of
	a derangement. Then $A$ is conjugate to a block diagonal matrix in
	which every $k\times k$ nonzero block is a diagonal matrix of the
	form
	\[
	\left[\begin{matrix}\omega^{0}\\
	& \omega^{1}\\
	&  & \ddots\\
	&  &  & \omega^{k-1}
	\end{matrix}\right]\in\mathrm{GL}_{k}(\mathbb{F})
	\]
	for $\omega$ a primitive $k$-th root of unity.
\end{lemm}

\begin{proof}
	Suppose $A$ is the permutation matrix of a derangement $\sigma\in S_{n}$.
	Let the cycle decomposition of $\sigma$ be 
	\[
	\left(i_{1}i_{2}\ldots i_{k_{1}}\right)\left(i_{k_{1}+1}i_{k_{1}+2}\ldots i_{k_{2}}\right)\ldots\left(i_{k_{r-1}+1}i_{k_{r-1}+2}\ldots i_n\right).
	\]
	Then if $P$ is the permutation matrix of the permutation that takes
	$j$ to $i_{j}$ for all $1\le j\le n$, it is clear that $P^{-1}AP$
	is a block diagonal matrix which blocks of size $k_{1},k_{2}-k_{1},k_{3}-k_{2},\ldots,n-k_{r-1}$,
	in which each nonzero $k\times k$ block is the permutation matrix
	of a cyclic permutation, i.e., is of the form
	\[
	B=\left[\begin{matrix}0 & 0 & 0 & 0 & 1\\
	1 & 0 & 0 & 0 & 0\\
	0 & 1 & 0 & 0 & 0\\
	0 & 0 & \ddots & 0 & 0\\
	0 & 0 & 0 & 1 & 0
	\end{matrix}\right]\in\mathrm{GL}_{k}(\mathbb{F}).
	\]
	Such a matrix defines a representation of $\mathbb{Z}/k\mathbb{Z}$
	(in which the generator $1\in\mathbb{Z}/k\mathbb{Z}$ maps to $B$).
	Its character vanishes on every $x\in\mathbb{Z}/k\mathbb{Z}$ except
	the identity, on which it achieves the value $k$. Since this is precisely
	the sum of the irreducible characters of $\mathbb{Z}/k\mathbb{Z}$
	the result follows.
	
	More concretely, if $\omega$ is a primitive $k$-th root of unity
	in $\mathbb{F}$ then for the Vandermonde matrix $Q=\left(\omega^{-\left(i-1\right)\left(j-1\right)}\right)_{1\le i,j\le k}$
	we have 
	\[
	QBQ^{-1}=\left[\begin{matrix}\omega^{0}\\
	& \omega^{1}\\
	&  & \ddots\\
	&  &  & \omega^{k-1}
	\end{matrix}\right].
	\qedhere\]
\end{proof}
We need a basic property of transcendental field extensions.
The following is elementary and well known.
\begin{lemm}
	\label[lemma]{lem:transcendentals}
	Let $\mathbb{F}\subset \mathbb{L}$ be fields with $z_1,\ldots,z_n \in \mathbb{L}$ transcendental over $\mathbb{F}$. Then $\mathbb{F}(z_1,\ldots,z_n)$ (the minimal subfield of $\mathbb{L}$ containing $\mathbb{F}$ and $z_1,\ldots,z_n$) is isomorphic to the field of rational functions in $n$ variables over $\mathbb{F}$. In particular, if $p\in\mathbb{F}\left[x_{1},\ldots,x_{n}\right]$ is nonzero then $p(z_1,\ldots,z_n)\neq 0$.
\end{lemm}

We apply this lemma in the following form:
\begin{coro}\label{cor:transcendental_entries}
Let $\mathbb{F}$ be a field and let $\mathbb{L}=\mathbb{F}\left(z_{k,i,j}\right)_{1\le k\le r,1\le i,j\le n}$.
For each $1\le k\le r$ denote by $A_{k}\in M_{n}\left(\mathbb{L}\right)$
the matrix given by
\[
A_{k}=\left(z_{k,i,j}\right)_{1\le i,j\le n}.
\]
Let $w$ an element of the free algebra over $\mathbb{F}$ with generators
$b_{1},\ldots,b_{r},b_{1}^{-1},\ldots,b_{r}^{-1}$ and $c_1,\dots,c_s,c_1^{-1},\dots,c_s^{-1}$ (note that formally
$b_{i}$ and $b_{i}^{-1}$ as well as $c_i$ and $c_i^{-1}$ are unrelated generators, and not inverses
in this algebra).
For invertible matrices $B_{1},\ldots,B_{r},C_1,\dots,C_s\in M_{n}\left(\mathbb{L}\right)$,
denote by $w\left(B_{1},\ldots,B_{r},C_1,\dots,C_s\right)\in M_{n}\left(\mathbb{L}\right)$
the matrix obtained by substituting $B_{i}$ and $B_i^{-1}$ for $b_i$
and $b_{i}^{-1}$ and $C_{i}$ and $C_i^{-1}$ for $c_i$
and $c_{i}^{-1}$in the expression $w$.

If there exist invertible matrices $B_{1},\ldots,B_{r},C_1,\dots,C_s\in M_{n}\left(\mathbb{L}\right)$
such that the matrix $w\left(B_{1},\ldots,B_{r},C_1,\dots,C_s\right)$
is invertible  then $w\left(A_{1},\ldots,A_{r},C_1,\dots,C_s\right)$
is invertible too.
\end{coro}
\begin{proof}
	Using Cramer's formula, consider $p=\det\left(w\left(A_{1},\ldots,A_{r},C_1,\dots,C_s\right)\right)$
	as a rational function over $\mathbb{F}$ in variables the entries
	$\left\{ z_{k,i,j}\right\} _{1\le k\le r,1\le i,j\le n}$ of $A_{1},\ldots,A_{r}$.
	Represent it as a reduced fraction of polynomials $\frac{f}{g}$ in
	the variables.
	Then
	\[
		\det\left(w\left(B_{1},\ldots,B_{r},C_1,\dots,C_s\right)\right)
	\]
	is the value of this rational function when the entries of $B_{1},\ldots,B_{r}$
	are substituted for the variables. In particular, $f$ and $g$ are
	nonzero (because they give nonzero values with this substitution).
	Thus also $p\neq0$ by an application of \cref{lem:transcendentals} to each of $f$
	and $g$.
\end{proof}

\begin{coro}\label{cor:invertible_matrices}
	Let $\mathbb{F}$ be a
	field, let $\mathbb{L}=\mathbb{F}\left(z\right)$, and let $A,B\in M_{n}\left(\mathbb{F}\right)$
	be invertible matrices. Then $\det\left(zA+B\right)\neq0$.
\end{coro}

\begin{proof}
	Consider $\det\left(xA+B\right)$ as a polynomial in $x$: it is nonzero because substituting $x=0$ yields $\det(B)\neq 0$. By \cref{lem:transcendentals} we have $\det(zA+B)\neq 0$ as required.
\end{proof}

\section{Dowling groupoids and partial Dowling geometries}\label{sec:dowling}

We find it useful to think of partial Dowling geometries (see
the overview in \cref{sec:intro}) as matroidal encodings of certain
groupoids, which we call Dowling groupoids. In this section we introduce
the Dowling groupoid of a finitely presented group, explain how representations
of these groupoids are related to representations of the associated
groups, and define the partial Dowling geometries.

To define the Dowling groupoids and geometries we use group presentations satisfying certain combinatorial requirements. While algebraically some of them are very artificial, they make the combinatorics that follows more convenient.
Note that the relators in our presentations are not necessarily reduced. (To describe group presentations we freely use both relators and relations as convenient.)
\begin{defn}\label[definition]{def:symmetric_triangular}
	We call a group presentation $\langle S\mid R\rangle$ \emph{symmetric triangular} if it satisfies the following conditions.
	\begin{enumerate}
		\item $S$ and $R$ are finite and the neutral element $e$ is a generator in $S$.
		\item The generators $S$ are symmetric. That is, for $e\neq s\in S$ also $s^{-1}\in S$. Further, $ss^{-1}e$ is a relator in $R$.
		\item All relators in $R$ are of length three.
		\item The relators in $R$ are cyclically symmetric. That is, if $abc\in R$ is a relator for $a,b,c\in S$ then also $bca$ and $cab$ are relators in $R$.
		\item If $abc\in R$ is a relator then also $c^{-1}b^{-1}a^{-1}$ is a relator in $R$.
		\item $eee$ is a relator in $R$.
	\end{enumerate}
\end{defn}
Any finitely presented group has a symmetric triangular presentation. To obtain one from a given presentation $\langle S \mid R \rangle$, first add the neutral element $e$ to $S$ if necessary. Then symmetrize the generators (by adding a formal inverse $s^{-1}$ for each $s\in S\setminus\{e\}$ which does not already have one, together with the relation $s^{-1}se=e$). Then ``break up'' long relators into short ones as follows: given a relator $s_1 s_2 \ldots s_n$ in $R$, add generators $x_2, x_3, \ldots, x_{n-2}$, and symmetrize the generating set (to add inverses for the new generators). Then add the relations
\begin{align*}
	s_1 s_2 x_1^{-1} = e,\quad	x_1 s_3 x_2^{-1} = e,\quad\dots,\quad	x_{n-2} s_{n-1} s_n = e,
\end{align*}
and delete $s_1 s_2 \ldots s_n=e$ from $R$.
Symmetrize $R$ by adding the cyclic shifts of each relator and their inverses.
Finally, add the relator $eee$.

\begin{defn}\label[definition]{def:Dowling_groupoid}
	Let $G$ be a group given by a symmetric triangular presentation $\left\langle S\mid R\right\rangle $.
	The \emph{Dowling groupoid}
	associated to $\left\langle S\mid R\right\rangle $
	is the finitely presented groupoid $\mathcal{G}$ with the following
	presentation:
	\begin{enumerate}
		\item The objects are $\left\{ b_{1},b_{2},b_{3}\right\} $,
		\item Generators for the morphisms are given by 
		\[
		\left\{ g_{s,i,j}:b_{i}\rightarrow b_{j}\mid s\in S,\mbox{ and } i,j\in \{1,2,3\}\text{ with }i\neq j\right\} .
		\]
		\item For each $s\in S$ and each pair of distinct indices $i,j\in\left\{ 1,2,3\right\} $
		we impose the relation $g_{s,j,i}\circ g_{s,i,j}=\text{id}_{b_{i}}$.
		For each cyclic shift $\left(i,j,k\right)$ of $\left(1,2,3\right)$
		and for each relation $s^{\prime\prime}s^{\prime}s=e$ in $R$, we
		impose the relation
		\[
		g_{s^{\prime\prime},k,i}\circ g_{s^{\prime},j,k}\circ g_{s,i,j}=\mathrm{id}_{b_{i}}.
		\]
	\end{enumerate}
\end{defn}

\begin{center}
\begin{figure}
	\includegraphics[width=.6\linewidth]{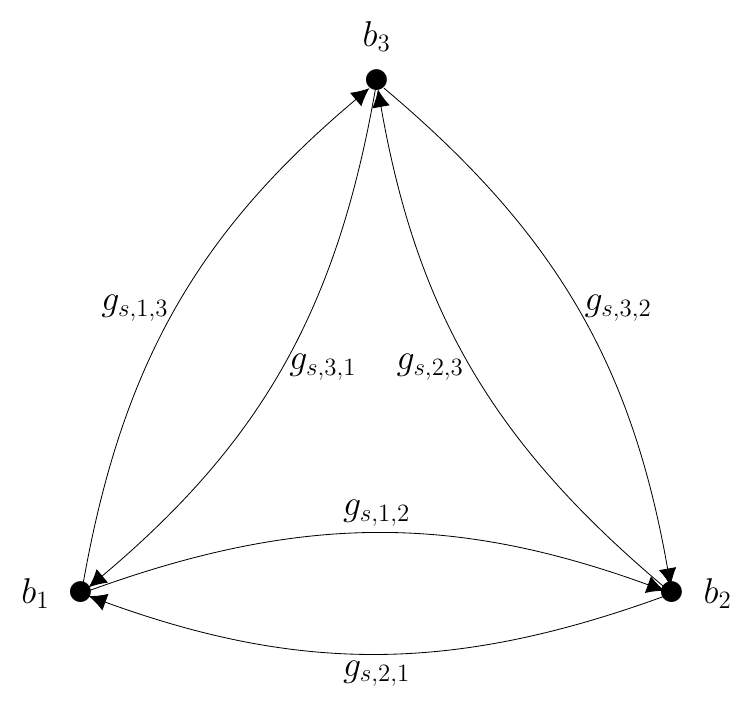}
	\caption{Morphisms in a Dowling groupoid that correspond to a generator $s \in S$.}
	\label{fig:gdg}
\end{figure}
\end{center}

\begin{rmrk}\label[remark]{rmk:redundant_relations}
	It is useful to note that since $\langle S \mid R \rangle$ is symmetric triangular, the following relations hold in $\mathcal{G}$:
	\begin{enumerate}
		\item For each permutation $\left(i,j,k\right)$ of $\left(1,2,3\right)$
				and for each $s\in S$, the relation 
				\[
				g_{e,j,k}\circ g_{s,i,j}=g_{s,j,k}\circ g_{e,i,j}
				\]
				holds.
		\item For each permutation $(i,j,k)$ of $(1,2,3)$, the relation $g_{e,k,i}\circ g_{e,j,k}\circ g_{e,i,j}=\mathrm{id}_{b_{i}}$ holds.
	\end{enumerate}
	Each relation of type (a) can be deduced from the relation
	\[
	g_{s^{-1},k,i}\circ g_{e,j,k}\circ g_{s,i,j} = \mathrm{id}_{b_i} = g_{s^{-1},k,i}\circ g_{s,j,k}\circ g_{e,i,j}
	\]
	which can itself be deduced from the defining relations of $\mathcal{G}$ and the fact that $s^{-1} e s = e$ is a relation of $\langle S \mid R \rangle$, because the presentation is symmetric triangular.
	
	Similarly, the six relations of type (b) follow from the defining relations of $\mathcal{G}$, together with the fact that $eee=e$ is a relation in $R$. Note that for $(i,j,k)$ an odd permutation of $(1,2,3)$ (which is not a cyclic shift) the relation $g_{e,k,i}\circ g_{e,j,k}\circ g_{e,i,j}=\mathrm{id}_{b_{i}}$ is the inverse of $g_{e,j,i}\circ g_{e,k,j} \circ g_{e,i,k}=\mathrm{id}_{b_i}$, where $(i,k,j)$ is a cyclic shift of $(1,2,3)$.
\end{rmrk}

\subsection{\texorpdfstring{Representations of $\mathcal{G}$ and of $G$}{Representations of Dowling groupoids and their groups}}\label{sec:dowling_representations}

Let $G$ be a group with a symmetric triangular presentation.
The Dowling groupoid $\mathcal{G}$ does not interest us in itself; it is
a sort of intermediate object between $G$ and the matroids constructed
further below. The point is that from a representation of $\mathcal{G}$ into some category $C$
(i.e. a functor $F:\mathcal{G}\rightarrow C$)
one can obtain a representation of $G$ in $C$ and vice versa. Here
$G$ is considered as a groupoid with one object $*$. This is shown
in several lemmas below. The proofs are rather obvious and readers
may wish to skip them (the purpose of this section is to verify
that the relations defining $\mathcal{G}$ have been chosen correctly).
\begin{lemm}
	\label[lemma]{lem:groupoid_one_object_rep}For each representation
	$F:\mathcal{G}\rightarrow C$ of $\mathcal{G}$ in a category $C$
	there is an isomorphic representation $F^{\prime}:\mathcal{G}\rightarrow C$
	which satisfies: 
	\begin{enumerate}
		\item\label{it:groupoid_lem_a} $F^{\prime}\left(b_{1}\right)=F^{\prime}\left(b_{2}\right)=F^{\prime}\left(b_{3}\right)$,
		\item\label{it:groupoid_lem_b} $F^{\prime}\left(g_{e,i,j}\right)=\text{id}_{F^{\prime}\left(b_{i}\right)}$
		for all $i,j$,
		\item\label{it:groupoid_lem_c} $F^{\prime}\left(g_{s,1,2}\right)=F^{\prime}\left(g_{s,2,3}\right)=F^{\prime}\left(g_{s,3,1}\right)$
		for each $s\in S$, and
		\item\label{it:groupoid_lem_d} $F^{\prime}\left(g_{s,2,1}\right)=F^{\prime}\left(g_{s,3,2}\right)=F^{\prime}\left(g_{s,1,3}\right)=F^{\prime}\left(g_{s,1,2}\right)^{-1}$ for each $s\in S$.
	\end{enumerate}
\end{lemm}
That $F^{\prime}$ is an isomorphic representation of $\mathcal{G}$
means that there is a natural isomorphism $F\rightarrow F^{\prime}$.
\begin{proof}
	Define $F^{\prime}$ on objects by setting $F^{\prime}\left(b_{i}\right)\coloneqq F(b_1)$ for all $1\le i\le 3$.
	Further define it on the generating morphisms $f:b_{i}\rightarrow b_{j}$ as follows:
	\[
	F'(f)\coloneqq	\begin{cases}
	F(g_{e,j,1}\circ f\circ g_{e,1,i}) &\mbox{ if } i\neq 1 \mbox{ and } j\neq 1,\\
	F(g_{e,j,1}\circ f) &\mbox{ if } i=1, \\
	F( f\circ g_{e,1,i}) &\mbox{ if } j=1.
	\end{cases}
	\]
	
	For each object $b_{i}$ of $\mathcal{G}$ we define an isomorphism
	$\eta_{b_{i}}:F\left(b_{i}\right)\rightarrow F'(b_i)$ by setting
	$\eta_{1}\coloneqq \id_{F\left(b_{1}\right)}$ and $\eta_{i}=F\left(g_{e,i,1}\right)$ for $i=2,3$.
	By definition of $F'$ this yields for each generating morphism $f:b_{i}\rightarrow b_{j}$
	of $\mathcal{G}$ the commutative diagram:
	\[
	\xymatrix{F\left(b_{i}\right)\ar[r]^{\eta_{b_{i}}}\ar[d]^{F\left(f\right)} & F^{\prime}\left(b_{i}\right)\ar[d]^{F^{\prime}\left(f\right)}\\
		F\left(b_{j}\right)\ar[r]^{\eta_{b_{j}}} & F^{\prime}\left(b_{j}\right).
	}
	\]
	For general morphisms of $\mathcal{G}$ the same diagrams commute, because they
	can be written as compositions of generating morphisms. Thus $F^{\prime}$
	is a functor, i.e. it respects composition: if $f_{2}\circ f_{1}=f_{3}$
	in $\mathcal{G}$ for some $f_{1}:b_{i}\rightarrow b_{j}$, $f_{2}:b_{j}\rightarrow b_{k}$,
	and $f_{3}:b_{1}\rightarrow b_{k}$ then the diagram
	\[
	\xymatrix{F\left(b_{i}\right)\ar[d]^{\eta_{b_{i}}}\ar[r]^{F\left(f_{1}\right)} & F\left(b_{j}\right)\ar[d]^{\eta_{b_{j}}}\ar[r]^{F\left(f_{2}\right)} & F\left(b_{k}\right)\ar[d]^{\eta_{b_{k}}}\\
		F^{\prime}\left(b_{i}\right)\ar[r]_{F^{\prime}\left(f_{1}\right)} & F^{\prime}\left(b_{j}\right)\ar[r]_{F^{\prime}\left(f_{2}\right)} & F^{\prime}\left(b_{k}\right)
	}
	\]
	commutes, implying that
	\[
	F^{\prime}\left(f_{2}\right)\circ F^{\prime}\left(f_{1}\right)\circ\eta_{b_{i}}=\eta_{b_{k}}\circ F\left(f_{2}\right)\circ F\left(f_{1}\right)
	\]
	and thus that $F^{\prime}\left(f_{2}\right)\circ F^{\prime}\left(f_{1}\right)=\eta_{b_{k}}\circ F\left(f_{2}\circ f_{1}\right)\circ\eta_{b_{i}}^{-1}=\eta_{b_{k}}\circ F\left(f_{3}\right)\circ\eta_{b_{i}}^{-1}$.
	But by definition we have $F^{\prime}\left(f_{3}\right)=\eta_{b_{k}}\circ F\left(f_{3}\right)\circ\eta_{b_{i}}^{-1}$,
	which shows
	\[
	F^{\prime}\left(f_{3}\right)=F^{\prime}\left(f_{2}\circ f_{1}\right)
	\]
	as desired. By definition the maps $\left\{ \eta_{b_{i}}\right\} _{b_{i} \text{ object in }\mathcal{G}}$
	define a natural isomorphism $F\rightarrow F^{\prime}$.
	
	We now prove each of the claimed properties of $F^{\prime}$ in
	turn:
	
	Property~\eqref{it:groupoid_lem_a} is satisfied by definition, and property~\eqref{it:groupoid_lem_b} follows
	from the following computations, in which we use the relations of
	$\mathcal{G}$:
	\begin{align*}
	F^{\prime}\left(g_{e,1,2}\right)=&F\left(g_{e,2,1}\circ g_{e,1,2}\right)=F\left(\mathrm{id}_{b_{1}}\right)=\mathrm{id}_{F^{\prime}\left(b_{1}\right)}\\
	F^{\prime}\left(g_{e,2,3}\right)=&F\left(g_{e,3,1}\circ g_{e,2,3}\circ g_{e,1,2}\right)=F\left(\mathrm{id}_{b_{1}}\right)=\mathrm{id}_{F^{\prime}\left(b_{1}\right)}\\
	F^{\prime}\left(g_{e,3,1}\right)=&F\left(g_{e,3,1}\circ g_{e,1,3}\right)=F\left(\mathrm{id}_{b_{1}}\right)=\mathrm{id}_{F^{\prime}\left(b_{1}\right)}.
	\end{align*}
	
	We now prove property~\eqref{it:groupoid_lem_c}.
	To this end observe that $F^{\prime}\left(g_{s,1,2}\right)=F\left(g_{e,2,1}\circ g_{s,1,2}\right)$
	and furthermore $F^{\prime}\left(g_{s,2,3}\right)= F\left(g_{e,3,1}\circ g_{s,2,3}\circ g_{e,1,2}\right)$.
	The relations of $\mathcal{G}$ imply that
	\begin{align*}
	F^{\prime}\left(g_{s,1,2}\right)=&F\left(g_{e,2,1}\circ g_{s,1,2}\right)\overset{(1)}{=}F\left(g_{e,3,1}\circ g_{e,2,3}\circ g_{s,1,2}\right)\\
	\overset{(2)}{=}&F\left(g_{e,1,3}\circ g_{s,2,3}\circ g_{e,1,2}\right)=F^{\prime}\left(g_{s,2,3}\right).
	\end{align*}
	where (1) is obtained by precomposing the identity $g_{e,2,1}\circ g_{e,1,2}=\mathrm{id}_{b_{1}}=g_{e,3,1}\circ g_{e,2,3}\circ g_{e,1,2}$
	with $g_{e,1,2}^{-1}$ and (2) follows from the relation $g_{e,2,3}\circ g_{s,1,2}=g_{s,2,3}\circ g_{e,1,2}$.
	The identity $F^{\prime}\left(g_{s,2,3}\right)=F^{\prime}\left(g_{s,3,1}\right)$
	follows similarly: we have
	\begin{align*}
	F^{\prime}\left(g_{s,3,1}\right)=&F\left(g_{s,3,1}\circ g_{e,1,3}\right)=F\left(g_{s,3,1}\circ g_{e,2,3}\circ g_{e,1,2}\right)\\
	=&F\left(g_{e,3,1}\circ g_{s,2,3}\circ g_{e,1,2}\right)=F^{\prime}\left(g_{s,2,3}\right).
	\end{align*}
	For property~\eqref{it:groupoid_lem_d}, using the fact that $g_{s,2,1}=g_{s,1,2}^{-1}$
	in $\mathcal{G}$ we see that $F^{\prime}\left(g_{s,2,1}\right)=F^{\prime}\left(g_{s,1,2}\right)^{-1}$.
	Similarly we have $g_{s,3,2}=g_{s,2,3}^{-1}$ and $g_{s,1,3}=g_{s,3,1}^{-1}$.
	It follows that
	\[
	F^{\prime}\left(g_{s,2,1}\right)=F^{\prime}\left(g_{s,3,2}\right)=F^{\prime}\left(g_{s,1,3}\right).\qedhere
	\]
\end{proof}
\begin{lemm}\label[lemma]{lem:groupoid_rep_to_group_rep}
	Consider $G$ as a groupoid with one object $*$ and morphisms the
	elements of the group $G$.
	Let $F:\mathcal{G}\rightarrow C$ be
	a representation satisfying properties~\eqref{it:groupoid_lem_a}-\eqref{it:groupoid_lem_d} of \cref{lem:groupoid_one_object_rep}.
	Then there is a functor
	\[
	F^{\prime}:G\rightarrow C
	\]
	defined by $F^{\prime}\left(*\right)=F\left(b_{1}\right)$ and on
	the generating morphisms by $F^{\prime}\left(s\right)=F\left(g_{s,1,2}\right)$.
\end{lemm}
\begin{proof}
	We only have to verify that $F^{\prime}\left(s^{\prime\prime}\right)\circ F^{\prime}\left(s^{\prime}\right)\circ F^{\prime}\left(s\right)=\mathrm{id}_{F^{\prime}\left(*\right)}$
	for any relation $s^{\prime\prime}s^{\prime}s=e$ in $R$.
	We compute
	\begin{align*}
	&F^{\prime}\left(s^{\prime\prime}\right)\circ F^{\prime}\left(s^{\prime}\right)\circ F^{\prime}\left(s\right)=F\left(g_{s^{\prime\prime},1,2}\right)\circ F\left(g_{s^{\prime},1,2}\right)\circ F\left(g_{s,1,2}\right)\\
	=&F\left(g_{s^{\prime\prime},3,1}\right)\circ F\left(g_{s^{\prime},2,3}\right)\circ F\left(g_{s,1,2}\right)=F\left(g_{s^{\prime\prime},3,1}\circ g_{s^{\prime},2,3}\circ g_{s,1,2}\right)=F\left(\mathrm{id}_{b_{1}}\right)\\
	=&\mathrm{id}_{F\left(b_{1}\right)}=\mathrm{id}_{F^{\prime}\left(*\right)}.\qedhere
	\end{align*}
\end{proof}
\begin{lemm}
	Consider $G$ as a groupoid with one object $*$. Let $F:G\rightarrow C$
	be a representation of $G$. Then there is a functor
	\[
	F^{\prime}:\mathcal{G}\rightarrow C
	\]
	defined on objects by $F^{\prime}\left(b_{i}\right)=F\left(*\right)$
	and on the generating morphisms by:
	\begin{align*}
	F^{\prime}\left(g_{s,1,2}\right)&=F^{\prime}\left(g_{s,2,3}\right)=F^{\prime}\left(g_{s,3,1}\right)=F\left(s\right) \mbox{ and }\\
	F^{\prime}\left(g_{s,2,1}\right)&=F^{\prime}\left(g_{s,3,2}\right)=F^{\prime}\left(g_{s,1,3}\right)=F\left(s\right)^{-1}
	\end{align*}
	for each $s\in S$. This functor satisfies properties~\eqref{it:groupoid_lem_a}-\eqref{it:groupoid_lem_d} of \cref{lem:groupoid_one_object_rep}.
\end{lemm}
\begin{proof}
	Once we prove $F^{\prime}$ is a functor, properties~\eqref{it:groupoid_lem_a}-\eqref{it:groupoid_lem_d} follow
	directly from the definition. Thus we only need to check that each
	relation between morphisms in $\mathcal{G}$ is respected by $F^{\prime}$. 
	
	Observe that if $s\in S$ and $i,j\in\left\{ 1,2,3\right\} $ are
	distinct then $F^{\prime}\left(g_{s,j,i}\right)\circ F^{\prime}\left(g_{s,i,j}\right)$
	equals either $F\left(s\right)\circ F\left(s\right)^{-1}$ or $F\left(s\right)^{-1}\circ F\left(s\right)$,
	depending on whether the pair $\left(i,j\right)$ is one of $\left\{ \left(1,2\right),\left(2,3\right),\left(3,1\right)\right\} $,
	and in either case the composition maps to the identity. Since $F^{\prime}\left(g_{e,i,j}\right)=F\left(e\right)=\mathrm{id}_{F\left(*\right)}$,
	it is clear that 
	\[
	F\left(g_{e,k,i}\right)\circ F\left(g_{e,j,k}\right)\circ F\left(g_{e,i,j}\right)=\mathrm{id}_{F\left(b_{i}\right)}
	\]
	whenever $i,j,k$ are distinct indices. Similarly, 
	\[
	F\left(g_{s,j,k}\right)\circ F\left(g_{e,i,j}\right)=F^{\prime}\left(s\right)^{\text{sgn}\left(j,k,i\right)}=F^{\prime}\left(s\right)^{\text{sgn}\left(i,j,k\right)}=F\left(g_{e,j,k}\right)\circ F\left(g_{s,i,j}\right).
	\]
	If $s^{\prime\prime}s^{\prime}s=e$ is a relation in $R$ and $\left(i,j,k\right)$
	is a cyclic shift of $\left(1,2,3\right)$ we then have $F^{\prime}\left(g_{s,i,j}\right)=F\left(s\right)$,
	$F^{\prime}\left(g_{s^{\prime},j,k}\right)=F\left(s^{\prime}\right)$,
	and $F^{\prime}\left(g_{s^{\prime\prime},k,i}\right)=F\left(s^{\prime\prime}\right)$,
	so that
	\begin{align*}
	&F^{\prime}\left(g_{s^{\prime\prime},k,i}\right)\circ F^{\prime}\left(g_{s^{\prime},j,k}\right)\circ F^{\prime}\left(g_{s,i,j}\right)=F\left(s^{\prime\prime}\right)\circ F\left(s^{\prime}\right)\circ F\left(s\right)\\
	=&F\left(s^{\prime\prime}s^{\prime}s\right)=F\left(e\right)=\mathrm{id}_{F\left(b_{i}\right)}.\qedhere
	\end{align*}
\end{proof}

\subsection{Partial Dowling geometries}

We define a class of matroids which extend the classical Dowling geometries from finite groups to finitely presented groups (we first defined these in~\cite{KY19}).
The structure closely parallels that of the Dowling groupoids defined above.
\begin{defn}
	\label[definition]{def:GDG}Let $G=\left\langle S\mid R\right\rangle $
	be a group together with a symmetric triangular presentation.
	The \emph{partial
		Dowling geometry} associated to the presentation $\left\langle S\mid R\right\rangle $
	is the rank~$3$ matroid $M$ on the ground set
	\[
	E\coloneqq\left\{ b_{1},b_{2},b_{3}\right\} \cup\left\{ s_{i}\mid s\in S,1\le i\le3\right\} 
	\]
	(i.e., three elements $b_{1},b_{2},b_{3}$ and three indexed copies
	of each element $s\in S$) and with the following flats of rank $2$
	(which are called lines, in analogy with affine geometry):
	\begin{itemize}
		\item For each $s\in S$, we place the element $s_{1}$ on the line spanned
		by $\left\{ b_{1},b_{2}\right\} $, and similarly $s_{2}$ and $s_{3}$
		are on the lines spanned by $\left\{ b_{2},b_{3}\right\} $ and $\left\{ b_{3},b_{1}\right\} $
		respectively. This means that each of the sets
		\[
		\left\{ b_{1},b_{2}\right\} \cup\left\{ s_{1}\mid s \in S\right\},\quad\left\{ b_{2},b_{3}\right\} \cup\left\{ s_{2}\mid s \in S\right\},\quad\left\{ b_{3},b_{1}\right\} \cup\left\{ s_{3}\mid s \in S\right\}
		\]
		is a flat.
		\item For each relation $s^{\prime\prime}s^{\prime}s=e$ in $R$ and any
		cyclic shift $(i,j,k)$ of the indices $(1,2,3)$, we take $\left\{ s_{i},s_{j}^{\prime},s_{k}^{\prime\prime}\right\} $
		to be a flat.
	\end{itemize}
\end{defn}

We call these matroids partial Dowling geometries as they are a finite restriction of the usual Dowling geometry of the group.
This is necessary for our purposes as the ground set of the Dowling geometry is infinite if the group is not finite.

\begin{prop}
	The partial Dowling geometry associated to a symmetric triangular group presentation is a matroid and $B=\left\{ b_{1},b_{2},b_{3}\right\} $ is one of its bases.
\end{prop}
\begin{proof}
	Any two of the rank-$2$ flats defined above intersect in at most
	one element, and $B$ is not contained in any of these flats. The
	result thus follows from \cite[Prop. 1.5.6]{Oxl11}.
\end{proof}

\begin{rmrk}
	Partial Dowling geometries are closely related to Zaslavsky's frame matroids of gain graphs~\cite{Zas89,Zas91}.
	For instance, the usual Dowling geometry of a finite group $G$ of rank $r$ is the full $ G $–expansion of the complete graph $K_r$ in Zaslavsky's notation.
	There is a corresponding construction for finitely generated groups, but it is not computable in general: doing so requires deciding whether certain words in the generators are trivial. 
\end{rmrk}

\begin{defn}\label{def:dowling_geometries}
	Let $G=\langle S \mid R \rangle$ be a group with a symmetric triangular presentation. We define a set $\mathcal{M}_{S,R}$ of partial Dowling geometries, which we call the \emph{set of partial Dowling geometries subordinate to $\langle S \mid R \rangle$}. 
	
	Denote by $T$ the set of all words $s''s's$, where $s,s',s'' \in S$ are three generators (not necessarily distinct). For each $X\subseteq T$ symmetrize the relations of $\langle S \mid R \cup X\rangle$ and denote by $M_X$ the partial Dowling geometry associated to the resulting group presentation. Then
	\[\mathcal{M}_{S,R} \coloneqq \{M_X \mid X\subseteq T\}.\]
\end{defn}
\begin{rmrk}
	Each partial Dowling geometry in $\mathcal{M}_{S,R}$ is the geometry associated to $\langle S \mid R\cup X\rangle$ for some $X$, and there is a quotient map $\langle S \mid R \rangle \to \langle S \mid R\cup X \rangle$ which is the identity on the generators.
	
	The family $\mathcal{M}_{S,R}$ can also be described as a collection of certain weak images of the partial Dowling geometry associated to $\langle S \mid R\rangle$, but we will not use this.
\end{rmrk}

\section{Probability space representations of matroids}\label{sec:probability_space_reps}

An entropic representation of a matroid is given by a collection of
random variables on a discrete probability space. We introduce some
new language to handle these more conveniently: 
rather than working with the entropy function, we prefer to work with the independence and determination properties of the variables. In terms of the matroids involved, this corresponds to working with independent sets and circuits.
We package everything
we need into the definition of a ``probability space representation''
and the accompanying notation. The discussion is essentially equivalent
to the probabilistic representations introduced by  Mat\'{u}\v{s} in \cite{Matus_background}.
\begin{defn}\label{def:independence_and_determination}
	Let $\left(\Omega,\mathcal{F},P\right)$
	be a probability space and let $\left\{ X_{e}\right\} _{e\in E}$
	be a finite collection of random variables on $\left(\Omega,\mathcal{F},P\right)$. (See \cref{sec:rand_variable_notation} for the notation.)
	\begin{enumerate}
		\item The variables $\left\{ X_{e}\right\} _{e\in E}$ are \emph{independent}
		if for any $\left(A_{e}\right)_{e\in E}\in\prod_{e\in E}\mathcal{F}_{e}$:
		\[
		P\left(\bigcap_{e\in E}X_{e}^{-1}\left(A_{e}\right)\right)=\prod_{e\in E}P\left(X_{e}^{-1}\left(A_{e}\right)\right).
		\]
		(This is the usual notion of independence of random variables.)
		\item Fix $c\in C\subseteq E$. The function
		$X_{c}$ is \emph{determined by $\left\{ X_{e}\right\} _{e\in C\setminus\{c\}}$
		}if there exists a measurable function
		\[
		f:\prod_{e\in C\setminus\{c\}}\Omega_{e}\rightarrow\Omega_{c}
		\]
		such that $f\circ\left(X_{e}\right)_{e\in C\setminus\{c\}}=X_{c}$. Such a function
		$f$ is called a \emph{determination function for $X_{c}$ given $\left\{ X_{e}\right\} _{e\in C\setminus\{c\}}$},
		or just a determination function for short.
	\end{enumerate}
\end{defn}
\begin{defn}\label{def:prob_space_rep}
	Let $M$ be a matroid on a
	finite set $E$. A probability space representation of $M$ consists
	of a discrete probability space $\left(\Omega,\mathcal{F},P\right)$
	and an indexed collection of random variables $\left\{ X_{e}\right\} _{e\in E}$
	on $\Omega$ such that the following conditions hold:
	\begin{enumerate}
		\item (Independence.) If $A\subseteq E$ is independent, the variables $\left\{ X_{e}\right\} _{e\in A}$
		are independent.
		\item (Determination.) If $C\subseteq E$ is a circuit and $c\in C$, then
		$X_{c}$ is determined by $\{ X_{e}\} _{e\in C\setminus\{ c\} }$. 
		\item (Non-triviality.) If $e\in E$ is not a loop, there are disjoint measurable
		$S,T\subsetneq\Omega_{e}$ such that $X_{e}^{-1}\left(S\right)$ and
		$X_{e}^{-1}\left(T\right)$ have nonzero probability.
	\end{enumerate}
\end{defn}
\begin{rmrk}
	The non-triviality condition implies, for instance, that $\Omega_{e}$
	is not a singleton. Together with the independence condition, it also
	ensures that if $e\in A\subseteq E$ where $A$ is independent then
	$X_{e}$ is \emph{not} determined by $\left\{ X_{f}\right\} _{f\in A\setminus\left\{ e\right\} }$. 
	
	Note that since the probability space is discrete, it is harmless
	to assume that all singletons have positive probability. With this
	additional assumption we have that $X_{A}$ is surjective for each
	independent $A\subseteq E$.
\end{rmrk}
	
	As in \cref{sec:rand_variable_notation}, 
	whenever we work with just one matroid on a ground set $E$ and one
	probability space representation in $\left(\Omega,\mathcal{F},P\right)$,
	we will denote the measurable spaces and functions associated to each
	element $e\in E$ by $\left(\Omega_{e},\mathcal{F}_{e}\right)$ and
	$X_{e}:\Omega\rightarrow\Omega_{e}$ respectively, without further
	explicit mention of the notation.

\begin{theorem}\label{thm:entropic_iff_prob_rep}
	Let $M$ be a connected matroid of rank at least two.
	Then $M$
	is entropic if and only if it has a probability space representation
	in a discrete probability space in which each singleton has nonzero
	probability. In this case, each of the random variables
	$\left\{ X_{e}\right\} _{e\in E}$ is uniformly distributed and the underlying
	probability space has a finite subset of probability $1$.
\end{theorem}
The first part of the theorem relies on standard facts concerning entropy functions and the second part of the theorem is a trivial generalization of a result by  Mat\'{u}\v{s} in \cite[p.190-191]{Matus_background} (which follows from the proof given in that paper).

\subsection{Entropic representations of partial Dowling geometries}\label{sec:entropic_reps}
In this section we extract group-theoretic information from entropic representations of partial Dowling geometries.
\begin{theorem}\label{thm:entropic_Dowling_reps}
	Let $G$ 
	be a group with a symmetric triangular presentation $\left\langle S\mid R\right\rangle$ and let $M$
	be the associated partial Dowling geometry.
	If $M$ is entropic
	then there exists $n\in\mathbb{N}$ such that there exists a group homomorphism $\rho:G\rightarrow S_{n}$ with $\rho\left(s\right)\neq\rho\left(s^{\prime}\right)$
	for distinct $s,s^{\prime}\in S$.
\end{theorem}

This follows from the following more technical result using \cref{lem:groupoid_one_object_rep,lem:groupoid_rep_to_group_rep}.

\begin{theorem}
	\label[theorem]{thm:Dowling_prob_rep} Let $\left\langle S\mid R\right\rangle $
	be a group with a symmetric triangular presentation. 
	Let $M=\left(E,\mathcal{C}\right)$
	be the associated partial Dowling geometry, and let $\mathcal{G}$
	be the corresponding groupoid. Suppose $M$ has a probability space
	representation in a discrete probability space $\left(\Omega,\mathcal{F},P\right)$,
	with each $e\in E$ assigned the measurable space $\left(\Omega_{e},\mathcal{F}_{e}\right)$
	and the measurable function $X_{e}:\Omega\rightarrow\Omega_{e}$.
	For $s\in S$ and each circuit $C\in\mathcal{C}$ of the form $\left\{ b_{i},b_{j},s_{i}\right\} $
	(with $i,j$ distinct) let 
	\begin{align*}
	&f_{s,i,j}:\Omega_{b_{i}}\times\Omega_{s}\rightarrow\Omega_{b_{j}}\text{ and}\\
	&f_{s,j,i}:\Omega_{b_{j}}\times\Omega_{s}\rightarrow\Omega_{b_{i}}
	\end{align*}
	be the two corresponding determination functions of the circuit.
	
	Further define
	\begin{align*}
	&\varphi_{s,i,j}:\Omega_{b_{i}}\times\Omega\rightarrow\Omega_{b_{j}}\times\Omega\\
	&\varphi_{s,i,j}\left(\omega_{i},\omega\right)=\left(f_{s,i,j}\left(\omega_{i,}X_{s_{i}}\left(\omega\right)\right),\omega\right)
	\end{align*}
	and similarly
	\begin{align*}
	&\varphi_{s,j,i}:\Omega_{b_{j}}\times\Omega\rightarrow\Omega_{b_{i}}\times\Omega\\
	&\varphi_{s,j,i}\left(\omega_{j},\omega\right)=\left(f_{s,j,i}\left(\omega_{j,}X_{s_{i}}\left(\omega\right)\right),\omega\right).
	\end{align*}
	Then there is a functor $F:\mathcal{G}\rightarrow\mathrm{FinSet}$
	defined on objects by $F\left(b_{i}\right)=\Omega_{b_{i}}\times\Omega$
	and on the generators of the morphisms by $F\left(g_{s,i,j}\right)=\varphi_{s,i,j}$.
	This functor is faithful (that is, it maps distinct generating morphisms
	to distinct morphisms.) More explicitly:
	\begin{enumerate}
		\item The functions $\varphi_{s,i,j}$ and $\varphi_{s,j,i}$ are mutually
		inverse.
		\item If $\left(i,j,k\right)$ is an even permutation of $\left(1,2,3\right)$
		and $s^{\prime\prime}s^{\prime}s=e$ is a relation in $R$ then
		\[
		\varphi_{s^{\prime\prime},k,i}\circ\varphi_{s^{\prime},j,k}\circ\varphi_{s,i,j}=\mathrm{id}_{\Omega_{b_{i}}\times\Omega}.
		\]
		\item If $s,s^{\prime}\in S$ are distinct elements and $i,j\in\left\{ 1,2,3\right\} $
		are distinct then $\varphi_{s,i,j}\neq\varphi_{s^{\prime},i,j}$.
	\end{enumerate}
\end{theorem}
\begin{proof}
	We assume, as we may by \cref{thm:entropic_iff_prob_rep}, that $\Omega$ (together with all probability spaces $\Omega_e$ for $e\in E$) is finite. Thus if $F$ defines a functor its values are in $\mathrm{FinSet}$ (rather than just $\mathrm{Set}$.) To show $F$ is a functor it suffices to prove the three statements above.
	\begin{enumerate}
		\item Let $\left(\omega_{i},\omega\right)\in\Omega_{b_{i}}\times\Omega$,
		and assume without loss of generality that $i$ precedes $j$ in the
		cyclic ordering of the indices. Denote $\omega_{s}=X_{s_{i}}\left(\omega\right)$.
		Then there exists $\omega^{\prime}\in\Omega$ such that $X_{b_{i}}\left(\omega^{\prime}\right)=\omega_{i}$
		and $X_{s_{i}}\left(\omega^{\prime}\right)=\omega_{s}$, since $X_{\left\{ b_{i},s_{i}\right\} }$
		is surjective. Denote $\omega_{j}=X_{b_j}\left(\omega^{\prime}\right)$.
		Then
		\[
		f_{s,i,j}\left(\omega_{i},\omega_{s}\right)=f_{s,i,j}\circ X_{\left\{ b_{i},s_{i}\right\} }\left(\omega^{\prime}\right)=X_{b_{j}}\left(\omega^{\prime}\right)=\omega_j,
		\]
		and similarly $f_{s,j,i}\left(\omega_{j},\omega_{s}\right)=f_{s,j,i}\circ X_{\left\{ b_{j},s_{i}\right\} }\left(\omega^{\prime}\right)=X_{b_{i}}\left(\omega^{\prime}\right)=\omega_i$.
		It follows that
		\[
		\varphi_{s,i,j}\left(\omega_{i},\omega\right)=\left(\omega_{j},\omega\right)\quad\text{and}\quad\varphi_{s,j,i}\left(\omega_{j},\omega\right)=\left(\omega_{i},\omega\right).
		\]
		\item Let $\left(\omega_{i},\omega\right)\in\Omega_{b_{i}}\times\Omega$,
		and denote $\omega_{s}=X_{s_{i}}\left(\omega\right)$, $\omega_{s^{\prime}}=X_{s_{j}^{\prime}}\left(\omega\right)$.
		Since $\left\{ b_{i},s_{i},s_{j}^{\prime}\right\} $ is an independent
		set, there exists $\omega^{\prime}\in\Omega$ such that
		\[
		X_{b_{i}}\left(\omega^{\prime}\right)=\omega_{i},\quad X_{s_{i}}\left(\omega^{\prime}\right)=\omega_{s},\quad\text{and }X_{s_{j}^{\prime}}\left(\omega^{\prime}\right)=\omega_{s^{\prime}}.
		\]
		Since $\left\{ s_{i},s_{j}^{\prime},s_{k}^{\prime\prime}\right\} \in\mathcal{C}$,
		the variable $X_{s_{k}^{\prime\prime}}$ is determined by the values
		of $X_{s_{i}}$ and $X_{s_{j}^{\prime}}$, and we have
		\[
		X_{s_{k}^{\prime\prime}}\left(\omega^{\prime}\right)=X_{s_{k}^{\prime\prime}}\left(\omega\right)
		\]
		because the same equalities hold for $X_{s_{i}}$ and $X_{s_{j}^{\prime}}$.
		Using this, we compute:
		\[
		\varphi_{s,i,j}\left(\omega_{i},\omega\right)=\left(f_{s,i,j}\left(\omega_{i},\omega_{s}\right),\omega\right)=\left(X_{b_{j}}\left(\omega^{\prime}\right),\omega\right)
		\]
		where the last equality holds because $X_{b_{i}}\left(\omega^{\prime}\right)=\omega_{i}$
		and $X_{s_{i}}\left(\omega^{\prime}\right)=\omega_{s}$. In precisely
		the same way,
		\begin{align*}
		\varphi_{s^{\prime},j,k}\left(X_{b_{j}}\left(\omega^{\prime}\right),\omega\right)&=\left(f_{s^{\prime},j,k}\left(X_{b_{j}}\left(\omega^{\prime}\right),\omega_{s^{\prime}}\right),\omega\right)=\left(X_{b_{k}}\left(\omega^{\prime}\right),\omega\right)\quad\text{and} \\
		\varphi_{s^{\prime\prime},k,i}\left(X_{b_{k}}\left(\omega^{\prime}\right),\omega\right)&=\left(f_{s^{\prime\prime},k,i}\left(X_{b_{k}}\left(\omega^{\prime}\right),X_{s_{k}^{\prime\prime}}\left(\omega^{\prime}\right)\right),\omega\right)=\left(\omega_{i},\omega\right)
		\end{align*}
		so that $\varphi_{s^{\prime\prime},k,i}\circ\varphi_{s^{\prime},j,k}\circ\varphi_{s,i,j}\left(\omega_{i},\omega\right)=\left(\omega_{i},\omega\right)$
		where $\left(\omega_{i},\omega\right)\in\Omega_{b_{i}}\times\Omega$
		is an arbitrary element.
		\item Assume without loss of generality that $i$ precedes $j$ in the cyclic
		ordering, and consider the circuits $C_{1}=\left\{ b_{i},b_{j},s_{i}\right\} $
		and $C_{2}=\left\{ b_{i},b_{j},s_{i}^{\prime}\right\} $ in $M$.
		Since $\left\{ s_{i},s_{i}^{\prime}\right\} $ is an independent subset,
		there exist elements $\omega,\omega^{\prime}\in\Omega$ such that
		$X_{s_{i}^{\prime}}\left(\omega\right)=X_{s_{i}^{\prime}}\left(\omega^{\prime}\right)$
		but $X_{s_{i}}\left(\omega\right)\neq X_{s_{i}}\left(\omega^{\prime}\right)$
		(here we used the non-triviality condition of probability space representations
		of matroids). 
		
		Fix $\omega_{i}\in\Omega_{b_{i}}$. Then by definition
		\[
		\varphi_{s^{\prime},i,j}\left(\omega_{i},\omega\right)=f_{s^{\prime},i,j}\left(\omega_{i},X_{s_{i}^{\prime}}\left(\omega\right)\right)=f_{s^{\prime},i,j}\left(\omega_{i},X_{s_{i}^{\prime}}\left(\omega^{\prime}\right)\right)=\varphi_{s^{\prime},i,j}\left(\omega_{i},\omega^{\prime}\right).
		\]
		Suppose for a contradiction that $\omega_{j}\coloneqq\varphi_{s,i,j}\left(\omega_{i},\omega\right)=\varphi_{s,i,j}\left(\omega_{i},\omega^{\prime}\right)$
		also. Since $\left\{ b_{i},s_{i}\right\} $ are independent we can
		find $\widetilde{\omega},\widetilde{\omega}^{\prime}\in\Omega$ such that
		$\left(X_{b_{i}},X_{s_{i}}\right)\left(\widetilde{\omega}\right)=\left(\omega_{i},X_{s_{i}}\left(\omega\right)\right)$
		and $\left(X_{b_{i}},X_{s_{i}}\right)\left(\widetilde{\omega}^{\prime}\right)=\left(\omega_{i},X_{s_{i}}\left(\omega^{\prime}\right)\right)$.
		Using the fact that $X_{b_{j}}$ is determined by $X_{b_{i}}$ and
		$X_{s_{i}}$ we obtain the equalities 
		\begin{align*}
		\left(X_{b_{i}},X_{s_{i}},X_{b_{j}}\right)\left(\widetilde{\omega}\right)&=\left(\omega_{i},X_{s_{i}}\left(\omega\right),\omega_{j}\right)\quad\text{and}\\ \left(X_{b_{i}},X_{s_{i}},X_{b_{j}}\right)\left(\widetilde{\omega}^{\prime}\right)&=\left(\omega_{i},X_{s_{i}}\left(\omega^{\prime}\right),\omega_{j}\right).
		\end{align*}
		In particular $\left(X_{b_{i}},X_{b_{j}}\right)\left(\widetilde{\omega}\right)=\left(\omega_{i},\omega_{j}\right)=\left(X_{b_{i}},X_{b_{j}}\right)\left(\widetilde{\omega}^{\prime}\right)$.
		But $X_{s_{i}}$ is determined by $X_{b_{i}}$ and $X_{b_{j}}$, so
		\[
		X_{s_{i}}\left(\omega\right)=X_{s_{i}}\left(\widetilde{\omega}\right)=X_{s_{i}}\left(\widetilde{\omega}^{\prime}\right)=X_{s_{i}}\left(\omega^{\prime}\right).
		\]
		This is a contradiction.\qedhere
	\end{enumerate}
\end{proof}

\section{Multilinear representations of matroids}\label{sec:multilinear}

Multilinear matroids are entropic~\cite{Mat99}, so the results of
\Cref{sec:entropic_reps} are valid for them as well: a multilinear
representation of a partial Dowling geometry gives rise, by the
correspondences described above, to a representation of the associated
groupoid. We prove a partial converse to this result, which states
that under certain conditions a matrix representation of a group $\left\langle S\mid R\right\rangle $
implies that the corresponding partial Dowling geometry is multilinear.
But first we digress and discuss multilinear matroid representations
on their own terms: We introduce an equivalent definition of multilinear
representability which is directly analogous to probability space
representations. We feel this definition helps clarify what is going
on:
groupoid representations are constructed using determination functions in a manner similar to the entropic case.

\subsection{Notation for vector spaces and linear maps}\label{sec:map_notation}
We introduce some notation which closely parallels the notation for probability spaces and random variables introduced in \cref{sec:rand_variable_notation}.

Let $\F$ be a field. An indexed collection of linear maps on a vector space $V$ over $\F$ consists of an index set $E$, a collection of vector spaces $\{W_e\}_{e\in E}$, and a collection of linear maps $\{T_e:V\to W_e\}_{e\in E}$. As in \cref{sec:rand_variable_notation}, we sometimes write ``let $\{T_e\}_{e\in E}$ be a collection of linear maps on $V$, and refer to the codomain of each $T_e$ by $W_e$ (without naming $W_e$ explicitly).

Given a tuple $S=(s_1,\ldots,s_n)$ of elements of $E$, we denote $W_S = \bigoplus_{i=1}^n W_{s_i}$ and define a linear map $T_S:V\to W_S$ by
\[T_S(v) = (T_{s_i}(v))_{i=1}^n.\]
If the order is inessential, the same notation can be used if $S$ is a set.

\subsection{Vector space representations}
\label{sec:vec_reps}

The following terminology is nonstandard, but useful because of the
close analogy with random variables, probability spaces, and probability
space representations of matroids.
All vector spaces in this section are over a fixed field $\mathbb{F}$ and assumed to be finite dimensional.

\begin{defn}\label[definition]{def:vec_rep}
	Let $V$ be a vector space, let $E$ be a finite set, and let $\{T_e\}_{e\in E}$ be a collection of linear maps on $V$.
	\begin{enumerate}
		\item The maps $\{T_e\}_{e\in E}$ are \emph{independent} if $\rk(T_E) = \sum_{e\in E}\dim W_e$.
		\item Fix $x\in E$. The map $T_x$ is \emph{determined} by $\{T_e\}_{e\in E\setminus\{x\}}$ if there exists a linear map $S:W_{E\setminus\{x\}}\to W_x$ such that
		\[T_x = S\circ T_{E\setminus\{x\}}.\]
	\end{enumerate}
\end{defn}

\begin{defn}
	\label[definition]{def:vect_space_rep}Let $M$ be a matroid on $E$.
	A \emph{vector space representation} of $M$ consists of $c\in \N$, a vector space $V$, a collection of vector spaces $\{W_{e}\}_{e\in E}$ with $\dim W_e=c$ for all $e\in E$, and a collection of linear maps $\{T_{e}:V\rightarrow W_{e}\}_{e\in E}$.
	These are required to satisfy:
	\begin{enumerate}
		\item If $A\subseteq E$ is independent in $M$ then the maps $\left\{ T_{e}\right\} _{e\in A}$
		are independent.
		\item If $c\in C\subseteq E$ is a circuit in $M$ then $T_{c}$ is determined
		by $\left\{ T_{e}\right\} _{e\in C\setminus\left\{ c\right\} }$. 
	\end{enumerate}
\end{defn}

We hope the proliferation of similar names (linear and multilinear
representations of matroids, vector space representations) does
not cause confusion.
Vector space representability for matroids is
equivalent to multilinear representability.

\begin{theorem}
	A simple matroid has a vector space representation if and only
	if it is multilinear.
\end{theorem}
This is a special case of the more general~\Cref{thm:multilinear_and_vector_space_reps}. It also appears, in somewhat implicit form, in \cite[Proposition 2.10]{BBP14}.

Given a representation of a finitely presented group we can construct a vector space representation of the associated partial Dowling geometry under certain conditions.

\begin{theorem}\label{thm:gdg_vector_space_reps}
	Let $G$
	be a group with a symmetric triangular presentation $\left\langle S\mid R\right\rangle$ and let $\rho:G\rightarrow\mathrm{GL}(W)$
	be a linear representation of $G$ in a vector space $W$. Suppose that
	\begin{enumerate}
		\item If $s,s'\in S$ are distinct then $\rho(s)-\rho(s^{\prime})$ is invertible,
		\item Whenever $s,s^{\prime},s^{\prime\prime}\in S$ (not necessarily
		distinct) satisfy $\rho(s s' s'')\neq \mathrm{id}_W$ 
		the linear transformation $\rho(s s' s'')-\mathrm{id}_W$ 
		is invertible, and
		\item For $s,s',s''\in S$ satisfying $\rho(s s' s'')=\mathrm{id}_W$, the equation $s s' s'' =e$ is a relation in~$R$.
	\end{enumerate}
	Then the partial Dowling geometry corresponding
	to the presentation $\left\langle S\mid R\right\rangle $ has a vector space representation.
	
	Moreover, if the representation $\rho$ just satisfies the assumptions (a) and (b) then some matroid of the partial Dowling geometries $\mathcal{M}_{S,R}$ subordinate to $\langle S \mid R \rangle$ has a vector space representation.
\end{theorem}
\begin{proof}
	This is the special case $\varepsilon=0$ of the more general~\Cref{thm:epsilon_reps} proved below.
\end{proof}

\section{Group scrambling}\label{sec:scrambling}

We introduce a two-step construction to modify finitely presented
groups. Its goal is to facilitate the encoding of word problems into
representation problems for partial Dowling geometries.
The main difficulty is that a linear representation of a group need not satisfy conditions (a,b) of \Cref{thm:gdg_vector_space_reps}.

The first step, which we call \emph{scrambling}, takes as input a
symmetric triangular presentation $\left\langle S\mid R\right\rangle $ of a group
$G$, and outputs a presentation $\left\langle S^{\prime}\mid R^{\prime}\right\rangle $
of $(G\ast F_R)\times\mathbb{Z}^{N}$ where $F_R$ is the free group on the generating set $R$, $\ast$ is the free product of groups, and $N\in \N$ is some natural number.
Any matrix representation of $G$ extends to a representation of $\langle S' \mid R' \rangle$ which satisfies conditions (a,b) of \Cref{thm:gdg_vector_space_reps}.

The second step, which we call \emph{augmentation}, takes as input the
result of the first step together with the original presentation $\left\langle S\mid R\right\rangle $
and a generator $s\in S$. It outputs a presentation of $\left(G\ast F_R\ast F_{4}\right)\times\mathbb{Z}^{N}$
(where $F_{4}=\left\langle z_{1},\ldots,z_{4}\right\rangle $ is the free group
on four generators). The resulting presentation has $z_{1}$ and $sz_{1}s$
as two of its generators; it has a matrix representation satisfying
the conditions of~\Cref{thm:gdg_vector_space_reps} if and only if
there is a matrix representation $\rho$ of $G$ such that $\rho\left(s\right)\neq\rho\left(e\right)$.
The ``only if'' direction follows from the fact that $z_{1}\neq sz_{1}s$
only if $s\neq e$. The ``if'' follows from a direct construction
of a representation, which is rather lengthy and forms a significant
part of what follows.

Throughout this section we work over $\C$ in order to ensure the existence of roots of unity.
Our main aim is to show that certain matroids are entropic, and for this it suffices to show that they are multilinear over some field by~\cite{Mat99}.
Therefore, working over $\C$ results in no loss of generality.

We will use \emph{Tietze transformations} to modify finite group presentations.
These are standard procedures so that two finite presentations define isomorphic group if and only if there exists a sequence of Tietze transformations moving one presentation to the other, see~\cite[Section II.2]{LS77} for details.
\subsection{Sufficiently generic elements}
In the rest of this section we face the following sort of problem several times: given some finitely presented group $G=\langle S \mid R \rangle$, a free group $F$ on some finite set of generators, an element $g\in G\ast F$, and a linear representation $\rho :G\ast F \to \mathrm{GL}_n(\C)$, show that $\rho(g)-I_n$ is invertible or zero.

We have some control over $\rho$. In particular, we are able to ensure that for each $s\in S$ the matrix $\rho(s)$ is either $I_n$ or the permutation matrix of a derangement. This motivates the next definition.

\begin{defn}\label{def:free_sufficiently_generic}
	Let $F_S$ be a free group on the set of generators $S$ and $F_T$ a free group on the set of generators $T$.
	Fix an element $x\in F_S\ast F_T$ and let  $\rho:F_S\ast F_T \to \mathrm{GL}_n(\C)$ be a linear representation.
	We consider the following two properties of $\rho$:
		\begin{enumerate}
		\item For each $s\in S$, $\rho(s)$ is $I_n$ or the permutation matrix of a derangement.
		\item The indexed collection of entries of the matrices $\{\rho(t)\}_{t \in T}$ is algebraically independent over $\Q$.
	\end{enumerate}
	We say that $x$ is a \emph{sufficiently generic word} if for all linear representations  $\rho:F_S\ast F_T \to \mathrm{GL}_n(\C)$ satisfying the conditions (a) and (b), the matrix $\rho(x)-I_n$ is either invertible or $0$.
\end{defn}

\begin{defn}\label{def:sufficiently_generic}
	Let $G$ be a group and let $S$ be a finite set together with a map $\varphi: S\to G$. Let $F_T$ a free group on the set of generators $T$. Denote by $\varphi: F_S \to G$ the group homomorphism mapping each $s \in S \subset F_S$ to the corresponding element $\varphi (s)$ of $G$. An element $x\in G*F_T$ is \emph{sufficiently generic relative to $\varphi$} if there exists a sufficiently generic word $\widetilde{x} \in F_S * F_T$ such that $x$ is conjugate to the image of $\widetilde{x}$ under the map $\varphi*\mathrm{id}_{F_{T}}:F_S*F_{T}\rightarrow G*F_{T}$.
\end{defn}
\begin{rmrk}
	Often, $S$ is either a subset of $G$ or a set of generators in a group presentation $\langle S \mid R \rangle \simeq G$, in which case $\varphi$ is the obvious map $S \to G$. In general, whenever the map $\varphi$ is clear from the context, we omit it and discuss sufficiently generic elements relative to $S$.
\end{rmrk}

We prove that various elements of $G*F_T$ are sufficiently generic.
\begin{lemm}\label{lem:sufficiently_generic}
	Let $g ,g'\in G$ and let $t \in T$. The following elements of $G\ast F_T$ are sufficiently generic relative to $\{g,g'\}$:
	\begin{enumerate}[(i)]
		\item\label{it:suff1} The element $gtgt^{-1}$,
		\item\label{it:suff2} the commutator $\left[t, g \right]$, and
		\item\label{it:suff3} the commutator $\left[ g,tg' \right]$.
	\end{enumerate}
	In particular, in parts \ref{it:suff1},\ref{it:suff2}, the elements are sufficiently generic relative to $\{g\}$.
	Moreover, if $\rho: F_{\{g,g'\}} * F_T \to \mathrm{GL}_n(\mathbb{C})$ is a representation such that $\rho(g)$ is the permutation matrix of a derangement and $\rho(t)$ has entries which are algebraically independent over $\mathbb{Q}$ then the element $w$ of each of \ref{it:suff1},\ref{it:suff2},\ref{it:suff3} satisfies that $\rho(w) - I_n$ is invertible.
\end{lemm}
\begin{proof}
	It suffices to prove the claim with $G$ replaced by the free group $F=\langle g,g' \rangle$.
	For the rest of the proof $G$ denotes this free group.
	
	Let $\rho :G\ast F_T\to \mathrm{GL}_n(\C)$ be a representation satisfying the assumptions (a) and (b) of~\Cref{def:free_sufficiently_generic}.
	So in particular $\rho(g)$ is either $I_n$ or the permutation matrix of a derangement and $\rho(t)$ is a matrix with algebraically independent entries.
	If $\rho(g) = I_n$ then $\rho(x)-I_n$ is $0$ for each of the considered words $x$ and we are done.
	We therefore assume that $\rho(g)$ is the permutation matrix of a derangement.
	
	By \cref{lem:derangement_matrix} each of the matrices $\rho(g)$ and $\rho(g^{-1})$ is conjugate
	to a block diagonal matrix in which each nonzero $m\times m$ block is a diagonal matrix of the form 
	\[
	\left[\begin{matrix}\omega^{0}\\
	& \ddots\\
	&  & \omega^{m-1}
	\end{matrix}\right]
	\]
	for $\omega$ a primitive $m$-th root of unity, and $m\ge 2$ for all blocks.
	Thus by changing basis we may assume that $\rho(g)$ has this form.

	\begin{enumerate}[(i)]
		\item To show $\rho\left(gtgt^{-1}\right)-I_n$ is invertible
		it suffices to show $\rho\left(t\right)-\rho\left(gtg\right)$
		is invertible.
		In a basis in which $\rho\left(g\right)$ has
		the form above, substitute a block diagonal matrix for $\rho(t)$
		in which each diagonal block is of the form
		\[
		\left[\begin{matrix}0 & 0 & 1\\
		0 & \iddots & 0\\
		1 & 0 & 0
		\end{matrix}\right].
		\]
		In this basis, each diagonal block of $\rho\left(t\right)-\rho\left(gtg\right)$
		is of the form
		\begin{align*}
		&\left[\begin{matrix}0 & 0 & 1\\
		0 & \iddots & 0\\
		1 & 0 & 0
		\end{matrix}\right]-\left[\begin{matrix}\omega^{0}\\
		& \ddots\\
		&  & \omega^{m-1}
		\end{matrix}\right]\left[\begin{matrix}0 & 0 & 1\\
		0 & \iddots & 0\\
		1 & 0 & 0
		\end{matrix}\right]\left[\begin{matrix}\omega^{0}\\
		& \ddots\\
		&  & \omega^{m-1}
		\end{matrix}\right]\\
		=&\left[\begin{matrix}0 & 0 & 1\\
		0 & \iddots & 0\\
		1 & 0 & 0
		\end{matrix}\right]-\left[\begin{matrix}0 & 0 & \omega^{m-1}\\
		0 & \iddots & 0\\
		\omega^{m-1} & 0 & 0
		\end{matrix}\right]=\left(1-\omega^{m-1}\right)\left[\begin{matrix}0 & 0 & 1\\
		0 & \iddots & 0\\
		1 & 0 & 0
		\end{matrix}\right],
		\end{align*}
		which has rank $m$, so each block is invertible.
		Thus by \Cref{cor:transcendental_entries} the matrix $\rho\left(gtgt^{-1}\right)-I_n$ is also invertible.
		\item Using~\Cref{cor:transcendental_entries} again it suffices to show that $\left[A,\rho\left(g\right)\right]-I_n$
		is invertible
		for some invertible matrix $A$.
		Again, working in a basis in which
		$\rho\left(g\right)$ has the block diagonal form described
		above and taking an $A$ with the same block structure, it
		suffices to show this on each diagonal block separately. Note that for invertible
		matrices $A,B$, the matrix $\left[A,B\right]-I_n$ is invertible if
		and only if $AB-BA$ is invertible. To see this, note that
		\[
		\left(AB-BA\right)\left(BA\right)^{-1}=ABA^{-1}B^{-1}-I_n.
		\]
		Thus for each integer $m\ge2$ and each primitive $m$-th root of
		unity $\omega$ we need to find an $m\times m$ matrix $A$ such that
		\[
		A\left[\begin{matrix}\omega^{0}\\
		& \ddots\\
		&  & \omega^{m-1}
		\end{matrix}\right]-\left[\begin{matrix}\omega^{0}\\
		& \ddots\\
		&  & \omega^{m-1}
		\end{matrix}\right]A
		\]
		is invertible. Take the matrix $A$ that acts on the standard basis
		$e_{1},\ldots,e_{n}$ of the column space $\C^{n}$ by $Ae_{i}=e_{i+1}$
		for $i<n$, and $Ae_{n}=e_{1}$. Thus
		\[
		A=\left[\begin{matrix}0 &  &  & 1\\
		1 & 0\\
		& \ddots & \ddots\\
		&  & 1 & 0
		\end{matrix}\right]
		\]
		where the unfilled entries are zero. Hence
		\begin{align*}
		&A\left[\begin{matrix}\omega^{0}\\
		& \ddots\\
		&  & \omega^{m-1}
		\end{matrix}\right]-\left[\begin{matrix}\omega^{0}\\
		& \ddots\\
		&  & \omega^{m-1}
		\end{matrix}\right]A\\
		=&\left[\begin{matrix}0 &  &  & \omega^{m-1}\\
		\omega^{0} & 0\\
		& \ddots & \ddots\\
		&  & \omega^{m-2} & 0
		\end{matrix}\right]-\left[\begin{matrix}0 &  &  & \omega^{0}\\
		\omega^{1} & 0\\
		& \ddots & \ddots\\
		&  & \omega^{m-1} & 0
		\end{matrix}\right]=\left(1-\omega\right)A\left[\begin{matrix}\omega^{0}\\
		& \ddots\\
		&  & \omega^{m-1}
		\end{matrix}\right]
		\end{align*}
		which is invertible, as a product of an invertible scalar and two
		invertible matrices.
		\item As $\rho(g')$ is a permutation matrix by assumption, $\rho(tg')$ is a matrix with algebraically independent entries. Thus this case follows from the previous one.\qedhere
	\end{enumerate}
\end{proof}

\subsection{Scrambled groups and their representations}
We encode the properties that our scrambling construction satisfies into a definition,
and work with it axiomatically to defer the discussion of the implementation. The actual construction is
postponed to \cref{sec:scrambling_construction}.
\begin{defn}\label{def:scrambling}
	Let $G$ be a group given by a symmetric triangular presentation $\langle S\mid R\rangle$.
	We	call a finitely presented group $G^{\prime}=\left\langle S^{\prime}\mid R^{\prime}\right\rangle $
	a \emph{scrambling} of $\left\langle S\mid R\right\rangle $ if it
	satisfies the following properties:
	\begin{enumerate}[(PS1), labelwidth=-6mm,  labelindent=2em,leftmargin =!]
		\item\label{it:sc1} $\langle S'\mid R'\rangle$ is a symmetric triangular presentation.
		\item\label{it:sc2} There is an isomorphism $\mu:G^{\prime}\to (G\ast F_R)\times\mathbb{Z}^{N}$
		for some $N\ge0$ where $F_R$ is the free group on the letters $f_r$ for $r\in R$. We denote the projections onto the factors by 
		\begin{align*}
		\pi_{G}&:(G\ast F_R)\times\mathbb{Z}^{N}\rightarrow G,\\
		\pi_{\mathbb{Z}}&:(G\ast F_R)\times\mathbb{Z}^{N}\rightarrow\mathbb{Z}^{N},\\
		\pi_{F,\mathbb{Z}}^\mathrm{ab}&:(G*F_R)\times\Z^N\to \Z^{|R|}\times\Z^N,
		\end{align*}
		where $\pi_{F,\mathbb{Z}}^\mathrm{ab}$ is the composition of the projection to $F_R\times \Z^N$ with the abelianization homomorphism of $F_R$.
		Slightly abusing notation we identify $G'$ with $(G\ast F_R)\times\mathbb{Z}^{N}$ via~$\mu$.
		\item\label{it:sc4} If $s,s^{\prime}\in S^{\prime}$ are distinct then $\pi_{F,\mathbb{Z}}^\mathrm{ab}\left(s\right)\neq\pi_{F,\mathbb{Z}}^\mathrm{ab}\left(s^{\prime}\right)$.
		\item\label{it:sc5} For any $s,s',s''\in S'$ (not necessarily distinct) either 
		\begin{enumerate}[(i)]
			\item $\pi_{F,\mathbb{Z}}^\mathrm{ab}(s^{\prime\prime}s^{\prime}s)\neq 0$,
			\item $s^{\prime\prime}s^{\prime}s=e$ in $G'$, or
			\item $s''s's$ is a sufficiently generic element in $G \ast F_R$ relative to the map $S \to G$ given by the presentation $G = \langle S \mid R \rangle$. (Note that $s''s's$ is in $G\ast F_R \simeq (G \ast F_R)\times\{0\}$ if (i) does not hold.)
		\end{enumerate}
		\item\label{it:sc6} There is a function $i:S\hookrightarrow S'$ such that $\pi_{G}\circ i=\mathrm{id}_{G}\restriction_{S}$.
		\item\label{it:sc7} There is a basis $B=\left\{ b_{1},\ldots,b_{N}\right\} $ of $\mathbb{Z}^{N}$
		and a function $j:B\rightarrow S^{\prime}$ such that $\mu\circ j\left(b_{i}\right)=\left(e_{G\ast F_R},b_{i}\right)$
		for each $1\le i\le N$.
		
		(The functions $i$ and $j$ are to be given explicitly.)
		\item\label{it:sc8} For each $s\in S$ we have $\mu(i(s))\in (G\ast \{e_F\})\times \Z^N\le (G\ast F_R)\times \Z^N$.
		Further, there is a $1\le k\le N$ such that $\pi_{\Z} \left(i\left(s\right)\right)=\sum_{m=1}^{N}c_{m}b_{m}$
		with $c_{k}\ge5$, and such that for each $s^{\prime}\in S^{\prime}$,
		the absolute value of the $b_{k}$-coefficient of $\pi_{\mathbb{Z}}\left(s^{\prime}\right)$
		is at most $c_{k}+1$.
	\end{enumerate}
\end{defn}
We will frequently use the following immediate consequence of the definition of a group scrambling.
\begin{prop}\label{cor:scrambling}
	In the notation of~\Cref{def:scrambling} consider the equation
	\[\pi_{F,\Z}^{\mathrm{ab}}(x''x'x)=0\]
	where $x,x',x''\in G'$.
	If we fix $x=g$ and $x''=g''$ for some generators $g,g''\in S'$ then there is at most one $x'\in S'$ that satisfies this equation.
\end{prop}
\begin{proof}
	Since the group $\Z^{|R|}\times \Z^N$ is abelian the equation $\pi_{F,\Z}^{\mathrm{ab}}(x''x'x)=0$ is equivalent to $\pi_{F,\Z}^{\mathrm{ab}}(x') =-\pi_{F,\Z}^{\mathrm{ab}}(gg'')$ assuming $x=g$ and $x''=g''$.
	Therefore by property \ref{it:sc4} if there exists a generator $g'$ such that $x'=g'$ fulfills the equation this generator must be unique.
\end{proof}

Let $G = \langle S \mid R \rangle$ be a group with a given symmetric triangular presentation. We prove that certain matrix representations of $G$ extend to nice representations of scramblings of $G$.

\begin{prop}
	\label[proposition]{prop:scrambling_reps}Let $\left\langle S^{\prime}\mid R^{\prime}\right\rangle $
	be a scrambling of $G=\left\langle S\mid R\right\rangle $, so that $\langle S'\mid R'\rangle \simeq (G\ast F_R)\times\mathbb{Z}^{N}$.
	Let	$\rho:G\rightarrow\mathrm{GL}_{n}(\C)$ be a representation satisfying that for each $g\in G$ the matrix $\rho(g)$ is either the permutation matrix of a derangement or the identity matrix.
	Then there exists a representation 
	\[
	\widetilde{\rho}:(G\ast F_R)\times\mathbb{Z}^{N}\rightarrow\mathrm{GL}_{n}(\mathbb{C})
	\]
	which satisfies: 
	\begin{enumerate}
		\item If $s,s^{\prime}\in S^{\prime}$ are distinct then $\widetilde{\rho}\left(s\right)-\widetilde{\rho}\left(s^{\prime}\right)$
		is invertible.
		\item For $s,s^{\prime},s^{\prime\prime}\in S^{\prime}$ (not necessarily
		distinct) the matrix $\widetilde{\rho}\left(s^{\prime\prime}s^\prime s \right)-I_n$
		is either invertible or zero.
		\item For each $g\in G$ we have $\widetilde{\rho}\left(g\right)=\rho\left(g\right)$.
	\end{enumerate}
\end{prop}

\begin{proof}
	Choose algebraically independent elements 
	$\left\{y_{r,i,j}\right\}_{{r \in R, 1\le i,j\le n}} \cup\left\{z_{1},\ldots,z_{N}\right\}\subset \mathbb{C}$ 
	over $\mathbb{Q}$.

	The free group $F_R$ has generators $f_r$ for $r\in R$.
	For each such generator $f_r$ we define
	\[
		\widetilde{\rho}(f_r)=(y_{r,i,j})_{1\le i,j\le n}.
	\]
	For $g\in G$ define $\widetilde{\rho}\left(g\right)=\rho\left(g\right)$.
	This extends to a representation $\widetilde{\rho}:G\ast F_R \to \mathrm{GL}_n(\C)$ because $G\ast F_R$ is a free product, and $F_R$ is free. Thus $\langle S \cup \{f_r\}_{r\in R} \mid R \rangle$ is a presentation of $G * F_R$, and it is clear that $\widetilde{\rho}$ maps all words that represent relators to the identity matrix.
	
	This representation extends further to a representation of $(G\ast F_R)\times\mathbb{Z}^{N}$ as follows. For $v=\left(v_{1},\ldots,v_{N}\right)\in\mathbb{Z}^{N}$ define
	\[
	\widetilde{\rho}\left(v\right)=\left(\prod_{i=1}^{N}z_{i}^{v_{i}}\right)\cdot I_{n}.
	\]
	If $g\in G\ast F_R$ and $v\in\mathbb{Z}^{N}$, define $\widetilde{\rho}\left(gv\right)=\widetilde{\rho}\left(g\right)\widetilde{\rho}\left(v\right)$.
	Any element of $(G\ast F_R)\times\mathbb{Z}^{N}$ can be written in exactly
	one way in the form $gv$, so $\widetilde{\rho}$ is well defined. It
	is a homomorphism essentially because if $v\in\mathbb{Z}^{N}$ then
	$\widetilde{\rho}\left(v\right)$ is a scalar matrix, and hence commutes
	with all matrices in the image of $\widetilde{\rho}$. More explicitly, we have
	\begin{align*}
	\widetilde{\rho}\left(g_{1}v_{1}\cdot g_{2}v_{2}\right)&=\widetilde{\rho}\left(\left(g_{1}g_{2}\right)\left(v_{1}v_{2}\right)\right)=\left(\widetilde{\rho}\left(g_{1}\right)\widetilde{\rho}\left(g_{2}\right)\right)\left(\widetilde{\rho}\left(v_{1}\right)\widetilde{\rho}\left(v_{2}\right)\right)\\
	&=\widetilde{\rho}\left(g_{1}\right)\widetilde{\rho}\left(v_{1}\right)\widetilde{\rho}\left(g_{2}\right)\widetilde{\rho}\left(v_{2}\right)=\widetilde{\rho}\left(g_{1}v_{1}\right)\widetilde{\rho}\left(g_{2}v_{2}\right),
	\end{align*}
	for $g_1,g_2\in G\ast F_R$ and $v_1,v_2\in \Z^N$.
	Observe that if $v\in\mathbb{Z}^{n}$ is nonzero then $\widetilde{\rho}\left(v\right)$
	is of the form $\lambda I_{n}$ where $\lambda$ is transcendental
	over $\mathbb{Q}$.
	
	We now prove the three claimed properties.
	It is convenient to define an auxiliary representation $\widetilde{\rho}^\mathrm{ab}:(G\ast F_R)\times \Z^N$ which is defined in the same way as $\widetilde\rho$ by extending $\rho$ except that $\widetilde{\rho}^\mathrm{ab}(f_r)=y_{r,1,1}I_n$ for all $r\in R$.
	\begin{enumerate}
		\item Let $s,s^{\prime}\in S^{\prime}$ be distinct elements. Denote $v=\pi_{F,\mathbb{Z}}^\mathrm{ab}\left(s\right)$
		and $v^{\prime}=\pi_{F,\mathbb{Z}}^\mathrm{ab}\left(s^{\prime}\right)$, as well
		as $z=\widetilde{\rho}^\mathrm{ab}\left(v\right)$ and $z^{\prime}=\widetilde{\rho}^\mathrm{ab}\left(v^{\prime}\right)$.
		By property \ref{it:sc4} of scramblings $v\neq v^{\prime}$, so $z^{-1}z^{\prime}$
		is transcendental 
		over $\Q$.
		Denote $g=\pi_{G}\left(s\right)$
		and $g^{\prime}=\pi_{G}\left(s^{\prime}\right)$.
		It suffices to prove that the matrix $\rho\left(g\right)z-\rho\left(g'\right)z'$ is invertible by \Cref{lem:transcendentals}: this matrix is obtained from $\widetilde{\rho}(s) - \widetilde\rho(s')$ by substituting different values instead of the transcendental matrix entries $\{y_{r,i,j}\}_{r,i,j}$. More explicitly, each matrix $\widetilde\rho(f_r)=(y_{r,i,j})_{i,j}$ is replaced by $y_{r,1,1} I_n$. Thus, instead of each off-diagonal entry $y_{r,i,j}$ ($i \neq j$) we substitute $0$, and instead of each diagonal entry $y_{r,i,i}$ we substitute $y_{r,1,1}$.

		Since $z^{-1}z^{\prime}$ is transcendental over $\Q$,
		 \cref{cor:invertible_matrices}
		implies that
		\[
		\det\left(\rho\left(g\right)-z^{-1}z^{\prime}\rho\left(g^{\prime}\right)\right)\neq0.
		\]
		Hence also $\det\left( \rho\left(g\right)z-\rho\left(g'\right)z' \right)\neq0$.
		\item Let $s,s^{\prime},s^{\prime\prime}\in S^{\prime}$ be not necessarily
		distinct generators.
		Then by property \ref{it:sc5} of scramblings exactly one of the following three cases holds:
		\begin{enumerate}[label=\textbf{Case \arabic*:}, labelwidth=-1cm,  labelindent=2em,leftmargin =!]
			\item Suppose $\pi_{F,\Z}^\mathrm{ab}(s''s's)\neq 0$.
			Denote $z = \widetilde{\rho}^\mathrm{ab}\left(\pi_{F,\Z}^\mathrm{ab}\left(s^{\prime\prime} s^\prime s\right)\right)$ as well as $g = \pi_G(s'' s' s)$. By construction $z\cdot \widetilde{\rho}^\mathrm{ab}(g) = \widetilde{\rho}^\mathrm{ab}(s'' s' s)$. Since $z$ is transcendental over $\Q$,
			\cref{cor:invertible_matrices} shows $\det(\widetilde{\rho}^\mathrm{ab}(s'' s' s) - I_n) \neq 0$. By \cref{cor:transcendental_entries} also $\widetilde{\rho}(s'' s' s) - I_n$ is invertible.
			\item If $s''s's=e$ then $\widetilde{\rho}(s''s's)-I_n = 0$.
			\item Suppose $s''s's$ is sufficiently generic relative to $S$. By construction, for each $g\in G$ the matrix $\widetilde{\rho}(g)$ is either the identity matrix or a permutation matrix of a derangement, and the entries of the matrices representing the free generators of $F_R$ are mutually transcendental elements over the prime field.
			So by definition of sufficiently generic elements the matrix $\widetilde{\rho}(s''s's)-I_n$ is invertible.
		\end{enumerate}
		\item This is immediate from the construction of $\widetilde{\rho}$.\qedhere
	\end{enumerate}
\end{proof}

\subsection{The augmentation construction}

We construct and prove the necessary properties of an augmentation of the presentation $\left(G\ast F_R\ast \left\langle z_{1},\ldots,z_{4}\right\rangle \right)\times\mathbb{Z}^{N}$ which we obtained from the above scrambling construction.
These properties are encoded by the Propositions \ref{prop:augmentation} and \ref{prop:augmentation_reps}.

\begin{construction}\label{con:augmentation} Let $\left\langle S^{\prime}\mid R^{\prime}\right\rangle $
	be a scrambling of the group $G=\left\langle S\mid R\right\rangle $ given by a symmetric triangular presentation,
	and let $s\in S$ be a given generator. We use the same notation as
	in \cref{def:scrambling}: $G^{\prime}=\left\langle S^{\prime}\mid R^{\prime}\right\rangle $
	is isomorphic to $(G\ast F_R)\times\mathbb{Z}^{N}$ for some given $N\in\mathbb{N}$,
	$B=\left\{ b_{1},\ldots,b_{N}\right\} $ is a basis of $\mathbb{Z}^{N}$,
	and $\mu$, $\pi_{\mathbb{Z}}$, $\pi_{F,\Z}$, $\pi^{\mathrm{ab}}_{F,\Z}$, $i$, and $j$ are the
	same maps as in that definition.
	
	In what follows we construct a new finitely presented group $G^{\prime\prime}=\left\langle S^{\prime\prime}\mid R^{\prime\prime}\right\rangle $
	by iteratively adding generators and relations to $S^{\prime}$
	and $R^{\prime}$. 
	\begin{enumerate}[(C1), labelwidth=-6mm,  labelindent=2em,leftmargin =!]
		\item\label{it:ca1} Add four generators $z_{1},\ldots,z_{4}$ to $S^{\prime}$. For each
		$1\le i\le N$ and each $1\le k\le4$ we add the following generators
		and relations in order to ensure that $j\left(b_{i}\right)$ commutes
		with $z_{k}$ in $G^{\prime\prime}$:
		\begin{enumerate}[(a)]
			\item Add a generator $u_{z_{k},i}$ and its inverse $u_{z_{k},i}^{-1}$. 
			\item Add the relations $u_{z_{k},i}u_{z_{k},i}^{-1}e=e$, $j\left(b_{i}\right)z_{k}u_{z_{k},i}^{-1}=e$,
			and $u_{z_{k},i}j\left(b_{i}\right)^{-1}z_{k}^{-1}=e$.
			\begin{rmrk}
				Note that the first of these relations ensures that $u_{z_{k},i}$
				and $u_{z_{k},i}^{-1}$ are actually inverses in $G^{\prime\prime}$;
				the second is equivalent to $u_{z_{k},i}=j\left(b_{i}\right)z_{k}$;
				and substituting the second relation into the third yields $j\left(b_{i}\right)z_{k}j\left(b_{i}\right)^{-1}z_{k}^{-1}=e$.
				We ``break up'' relations in this way in the rest of this construction
				and in \cref{con:scrambling} to ensure that indeed all relations in the constructed presentation have length three.
			\end{rmrk}
			
		\end{enumerate}
		\item\label{it:ca2} Add a new generator $t$ to $S^{\prime}$. The following ensures that
		$t=sz_1s$ in $G^{\prime\prime}$: Denote $s^{\prime}=i\left(s\right)$,
		and express $-2\cdot\pi_{F,\mathbb{Z}}^{\mathrm{ab}}\left(s^{\prime}\right)\in\mathbb{Z}^{N}$
		as a minimal-length sum
		\[
		\varepsilon_{1}b_{k_{1}}+\varepsilon_{2}b_{k_{2}}+\ldots+\varepsilon_{r}b_{k_{r}}
		\]
		of elements of $B$, where $\varepsilon_{1},\ldots,\varepsilon_{r}\in\left\{ -1,1\right\} $.
		Recall that $\pi_{F,\mathbb{Z}}^{\mathrm{ab}}(s')$ is generated by elements in $B$ by property \ref{it:sc8}.
		We add generators and relations to ``break up'' the relation 
		\[
		t=z_{4}^{-1}\left(z_{4}\left(\left(z_{3}^{-1}\left(\left(z_{3}s^{\prime}\right)z_{1}s^{\prime}\right)\right) z_{2} b_{k_{1}}^{\varepsilon_{1}} z_{2} b_{k_{2}}^{\varepsilon_{2}} z_{2}\ldots b_{k_{r-1}}^{\varepsilon_{r-1}} z_{2} b_{k_{r}}^{\varepsilon_{r}}\right)\underbrace{z_{2}^{-1}\ldots z_{2}^{-1}}_{\text{\ensuremath{r} times}}\right).
		\]
		Explicitly: 
		\begin{enumerate}[(a)]
			\item Add generators $v_{1},\ldots,v_{4}$, one for each of the words
			\[ z_{3}s^{\prime},\quad
			\left(z_{3}s^{\prime}\right)z_{1}, \quad\left(z_{3}s^{\prime}\right)z_{1}s^{\prime},
			\left(z_{3}^{-1}\left(\left(z_{3}s^{\prime}\right)z_{1}s^{\prime}\right)\right)=s^{\prime}z_{1}s^{\prime}.\]
			Then add their inverses, together with relations 
			\[
			v_{1}v_{1}^{-1}e=e,\ldots,v_{4}v_{4}^{-1}e=e
			\]
			and the relations
			\[
			z_{3}s^{\prime}v_{1}^{-1}=e,\quad v_{1}z_{1}v_{2}^{-1}=e,\quad v_{2}s^{\prime}v_{3}^{-1}=e,\quad z_{3}^{-1}v_{3}v_{4}^{-1}=e.
			\]
			These relations ensure that $v_{1}=z_{3}s^{\prime}$, $v_{2}=\left(z_{3}s^{\prime}\right)z_{1}$,
			$v_{3}=\left(z_{3}s^{\prime}\right)z_{1}s^{\prime}$, and $v_{4}=s^{\prime}z_{1}s^{\prime}$
			in the resulting group.
			\item Add further generators $v_{5}=v_{4+1}$ up to $v_{4+2r}$, one for
			each of the words 
			\[
				\left(z_{3}^{-1}\left(\left(z_{3}s^{\prime}\right)z_{1}s^{\prime}\right)\right) z_{2}, \dots, 
		\left(z_{3}^{-1}\left(\left(z_{3}s^{\prime}\right)z_{1}s^{\prime}\right)\right) z_{2} b_{k_{1}}^{\varepsilon_{1}} z_{2} b_{k_{2}}^{\varepsilon_{2}} z_{2}\dots b_{k_{r-1}}^{\varepsilon_{r-1}} z_{2} b_{k_{r}}^{\varepsilon_{r}}.
			\]
			Add the inverses of these generators, together with the appropriate
			relations (analogously to the above).
			\item Add generators $v_{5+2r}$ up to $v_{5+3r}$ for each of the words
			\[
			z_{4}v_{4+2r},z_{4}v_{4+2r}z_{2}^{-1},\dots,z_{4} v_{4+2r}\underbrace{z_{2}^{-1}\dots z_{2}^{-1}}_{\text{\ensuremath{r} times}}.
			\]
			Add inverses for these generators, and add the appropriate relations
			(exactly as above).
			\item Add a generator $t$, together with its inverse and the relation $tt^{-1}e=e$.
			Then add the relation $z_{4}^{-1}v_{5+3r}t^{-1}=e$ to ensure $t=sz_{1}s$
			in $G^{\prime\prime}$.
		\end{enumerate}
		\item\label{it:ca3} Symmetrize the set of relations.
	\end{enumerate}
\end{construction}

We abuse notation slightly and denote by $b_{i}$
	(for $1\le i\le N$) the element $j\left(b_{i}\right)$ in $G^{\prime\prime}$.
	As for general elements of $G''$, we use multiplicative notation for $b_i$ in this context.
	Thus for $\varepsilon\in\left\{ -1,1\right\} $,
	$b_{i}^{\varepsilon}$ denotes an element of $G^{\prime\prime}$,
	but $\varepsilon b_{i}$ denotes an element of $\mathbb{Z}^{N}$.

\begin{prop}\label{prop:augmentation}
	In the notation of the construction,
	$G^{\prime\prime}=\left\langle S^{\prime\prime}\mid R^{\prime\prime}\right\rangle $
	is isomorphic to $\left(G\ast F_R\ast \left\langle z_{1},\ldots,z_{4}\right\rangle \right)\times\mathbb{Z}^{N}$
	by an isomorphism which maps each element of $S^{\prime}\subset S^{\prime\prime}$
	to the corresponding element of \[(G\ast F_R)\times\mathbb{Z}^{N}\le\left(G\ast F_R*\left\langle z_{1},\dots,z_{4}\right\rangle \right)\times\mathbb{Z}^{N},\]
	and $z_1,\dots,z_4$ to the elements of the same
	name in $\left(G*F_R *\left\langle z_{1},\ldots,z_{4}\right\rangle \right)\times\mathbb{Z}^{N}$.
	
	This isomorphism maps $t\in S^{\prime\prime}$ to $\left(sz_{1}s,0\right)\in\left(G\ast F_R*\left\langle z_{1},\ldots,z_{4}\right\rangle \right)\times\mathbb{Z}^{N}$.
\end{prop}

\begin{proof}
	The proposition defines a map $G^{\prime\prime}\rightarrow\left(G\ast F_R\ast\left\langle z_{1},\ldots,z_{4}\right\rangle \right)\times\mathbb{Z}^{N}$,
	and this is clearly surjective. It is injective: first note that step
	\ref{it:ca1} of the construction ensures that every element of $G^{\prime\prime}$
	commutes with each $j\left(b_{k}\right)$. Consider a relation added
	during step \ref{it:ca2}, skipping over all relations of the form $yy^{-1}e=e$
	for $y$ a new generator. Each such relation is of the form $x_{1}\ldots x_{n} y^{-1}=e$,
	for $y$ one of the new generators which does not appear in any of
	the previous relations (except $yy^{-1}e=e$). Thus, traversing this
	list in reverse, we may apply Tietze transformations to remove each
	relation along with the generator $y$. The same procedure can be
	applied to the relations $u_{z_{k},i}u_{z_{k},i}^{-1}e=e$ and $j\left(b_{i}\right)z_{k}u_{z_{k},i}^{-1}=e$
	and the generators $u_{z_{k},i}$ (for all $1\le k\le4$ and $1\le i\le N$),
	thus eliminating all new generators in $S^{\prime\prime}$ except
	for $z_{1},\ldots,z_{4}$ and their inverses. At the end of this process is finished we end up with the group presentation
	\[
	\left\langle S^{\prime}\cup\left\{ z_{1},\ldots,z_{4}\right\} \mid R^{\prime}\cup\left\{ j\left(b_{i}\right)z_{k}j\left(b_{i}\right)^{-1}z_{k}^{-1}\right\} _{\substack{1\le i\le N\\
			1\le k\le4
		}
	}\right\rangle ,
	\]
	which is isomorphic to $\left(G\ast F_R*\left\langle z_1,\dots,z_4\right\rangle \right)\times\mathbb{Z}^{N}$
	in the desired manner. 
	
	Observe that $t=s^{\prime}z_{1}s^{\prime} z_{2} b_{k_{1}}^{\varepsilon_{1}} z_{2} b_{k_{2}}^{\varepsilon_{2}} z_2\dots b_{k_{r-1}}^{\varepsilon_{r-1}} z_{2} b_{k_{r}}^{\varepsilon_{r}} z_{2}^{-r}$
	in $G^{\prime\prime}$, where $s^{\prime}\in S^{\prime}$ maps to
	$\left(s,\pi_{\mathbb{Z}}\left(s'\right)\right)\in\left(G*\left\langle z_{1},\dots,z_{4}\right\rangle \right)\times\mathbb{Z}^{N}$
	(this is its image in $G\ast F_R\times\mathbb{Z}^{N}$ under $\mu$). 
\end{proof}
This proposition allows us to identify $G^{\prime\prime}$ with $\left(G\ast F_R*\left\langle z_{1},\ldots,z_{4}\right\rangle \right)\times\mathbb{Z}^{N}$.

\begin{notation}\label[notation]{not:deg_maps}
	Consider the quotient map $G'' \to \langle z_1,\ldots,z_4 \rangle$. Composing the abelianization homomorphism $\langle z_1,\ldots,z_4 \rangle \to \Z^4$ on this map we obtain a homomorphism 
	\[\deg_z:G'' \to \Z^4.\]
	Define homomorphisms $\deg_{z_i}:G'' \to \Z$ for each $1\le i\le 4$, so that $\deg_{z_i}(x)$ is the total degree of $z_i$ in $x$, and $deg_z(x) = (\deg_{z_1}(x), \ldots, \deg_{z_4}(x))$.
\end{notation}

\begin{prop}\label{prop:augmentation_reps}
	Let $G=\left\langle S\mid R\right\rangle $
	be a group given by a symmetric triangular presentation and let $G^{\prime}=\left\langle S^{\prime}\mid R^{\prime}\right\rangle $
	be a scrambling. Let $s\in S$ and let $G^{\prime\prime}=\left\langle S^{\prime\prime}\mid R^{\prime\prime}\right\rangle \simeq\left(G\ast F_R*\left\langle z_{1},\ldots,z_{4}\right\rangle \right)\times\mathbb{Z}^{N}$
	be the associated augmentation.
	Then the following two conditions are equivalent:
	\begin{enumerate}[(i)]
		\item There exists a
			representation $\rho:G\rightarrow\mathrm{GL}_{n}(\mathbb{C})$ for some $n\in\mathbb{N}$ with $\rho\left(s\right)\neq\rho\left(e\right)$.
		\item There exists a representation $\widetilde{\rho}:G''\rightarrow\mathrm{GL}_{n}(\mathbb{C})$ for some $n\in\mathbb{N}$ which satisfies:
		\begin{enumerate}
			\item If $x,x^{\prime}\in S^{\prime\prime}$ are distinct then $\widetilde{\rho}\left(x\right)-\widetilde{\rho}\left(x^{\prime}\right)$
			is invertible, 

			\item For $x,x^{\prime},x^{\prime\prime}\in S^{\prime\prime}$ a not necessarily
			distinct triple of generators,   $\widetilde{\rho}(x^{\prime\prime}x^{\prime}x)-I_n$ is invertible or $0$.
		\end{enumerate}
	\end{enumerate}
\end{prop}

\begin{proof}
	Assume (ii) holds. The isomorphism from $G^{\prime\prime}$ to $\left(G\ast F_R*\left\langle z_{1},\ldots,z_{4}\right\rangle \right)\times\mathbb{Z}^{N}$
		stemming from~\Cref{prop:augmentation} maps the generator $t\in S^{\prime\prime}$
		(see \cref{con:augmentation}) to $sz_{1}s$.
		Further observe that $z_{1}\in S^{\prime\prime}$.
	Since $z_{1},t$ are distinct generators, $\widetilde{\rho}\left(z_{1}\right)-\widetilde{\rho}\left(t\right)$
	is invertible, and in particular $\widetilde{\rho}\left(t\right)=\widetilde{\rho}\left(sz_{1}s\right)\neq\widetilde{\rho}\left(z_{1}\right)$.
	Thus $\widetilde{\rho}\left(s\right)\neq\widetilde{\rho}\left(e\right)$.
	Restricting $\widetilde{\rho}$ to $G\le\left(G\ast F_R*\left\langle z_{1},\ldots,z_{4}\right\rangle \right)\times\mathbb{Z}^{N}$
	we obtain (i).
	
	Assuming (i) holds, let $\rho:G\rightarrow\mathrm{GL}_{n}(\mathbb{C})$
	be a representation such that $\rho\left(s\right)\neq\rho\left(e\right)$.
	By applying \cref{lem:derangement_rep}, changing $n$ as necessary,
	we obtain a new representation $\rho$ of $G$ with the property
	that every $\rho(x)$ for $x\in S$ is the permutation matrix of a derangement or the identity matrix and $\rho(s)\neq I_n$.
	By \cref{prop:scrambling_reps}, $\rho$ extends to
	a representation $\rho^{\prime}$ of $G^{\prime}\simeq G\ast F_R\times\mathbb{Z}^{N}$
	over $\mathbb{C}$ satisfying conditions analogous to (a) and (b).
	Let $\{\xi_{k,i,j}\}_{1\le k\le 4,\,1\le i,j\le n}$ be a collection of complex numbers which are algebraically independent over $\Q$ and algebraically independent over all entries in the matrices of the image of $\rho'$.
	Now extend $\rho^{\prime}$ to $\widetilde{\rho}:G^{\prime\prime}\rightarrow\mathrm{GL}_{n}(\mathbb{C})$
	by defining (on generators) $\widetilde{\rho}\left(z_{k}\right)=\left(\xi_{k,i,j}\right)_{1\le i,j\le n}\in\mathrm{GL}_{n}(\mathbb{C})$
	for each $1\le k\le4$.
	This defines a representation $\widetilde{\rho}:\left(G\ast F_R*\left\langle z_{1}\ldots,z_{4}\right\rangle \right)\times\mathbb{Z}^{N}\simeq G^{\prime\prime}\rightarrow\mathrm{GL}_{n}(\mathbb{C})$
	because elements of the $\mathbb{Z}^{N}$-factor map to scalar matrices,
	which commute with all matrices in $\mathrm{GL}_{n}(\mathbb{C})$, and
	because any representation of $G\ast F_R$ extends to a representation of
	$G\ast F_R*\left\langle z_{1},\ldots,z_{4}\right\rangle $ once the images
	of $z_{1},\ldots,z_{4}$ are chosen (there is no constraint on these
	images because there are no nontrivial relations involving any of
	$z_{1},\ldots,z_{4}$).
	
	We verify that conditions (a) and (b) hold for this representation
	by considering the various pairs and triples of generators in $S^{\prime\prime}$.

		As a first step, we verify that if $\deg_{z}\left(x^{\prime\prime}x^{\prime}x\right)\neq0$
		then $\widetilde{\rho}\left(x^{\prime\prime}x^\prime x\right)-I_n$
		is invertible.
		So assume for $x,x',x''\in S''$ that $\deg_{z_{i}}\left(x^{\prime\prime} x^{\prime}x\right)\neq 0$
		for some $i$.
	Considering 
	\[
	\det\left(\widetilde{\rho}\left(x^{\prime\prime}x^\prime x\right)-I_n \right)
	\]
	as a polynomial in the entries of $\widetilde{\rho}\left(z_{i}\right)$
	we see that the determinant doesn't vanish, because it doesn't vanish
	if we substitute a transcendental multiple of the identity matrix
	by \cref{cor:invertible_matrices}.
	We also get invertibility if
	$\pi_{F,\mathbb{Z}}^\mathrm{ab}\left(x^{\prime\prime}x^{\prime}x\right)\neq0$,
	by the same argument.
	
	Therefore to verify the conditions (a) and (b) it suffices to check ordered pairs $x,x'$ of generators
	which have equal values under $\deg_{z}$ and $\pi_{F,\mathbb{Z}}^\mathrm{ab}$
	and ordered triples $x,x',x''$ with $\deg_{z}(x'' x' x)= \pi_{F,\mathbb{Z}}^\mathrm{ab}(x''x'x)=0$.
	By the next proposition (\cref{prop:trivial_or_generic_products}), the only such pairs are $x=t,x'=z_1$, the inverse pair $x=t^{-1},z_1^{-1}$, and their re-orderings. To see that $\widetilde\rho(x x') - I_n$ is invertible in this case, we note that $t= sz_1 s$ in $G''$ by \cref{prop:augmentation}, and $\rho(s)\neq e$ is the permutation matrix of a derangement. 
	From \cref{lem:sufficiently_generic}, it follows that
	$\rho(t z_1^{-1}) - I_n$ is invertible
		and hence so are $\rho(t) - \rho(z_1)$ and $\rho(z_1^{-1}) - \rho(t^{-1})$ as desired.

	Similarly, for each ordered triple $x,x',x''$ with $\deg_z(x''x'x)=\pi_{F,\Z}^{\mathrm{ab}}(x''x'x) =0$, we need to check that $\widetilde{\rho}(x'' x' x) - I_n$ is invertible or $0$.
	It suffices that $x'' x' x$ is either the identity element or sufficiently generic in $G*(\langle z_1,\ldots,z_4\rangle * F_R)$ relative to $S$ (notice that $x'' x' x \in G*(\langle z_1,\ldots,z_4\rangle * F_R)$ because its projection to $\Z^N$ is trivial by assumption. Since $(\langle z_1,\ldots,z_4\rangle * F_R)$ is a free group, we can discuss its sufficient genericity). That this holds is precisely the statement of the next proposition.
\end{proof}

\begin{prop}\label{prop:trivial_or_generic_products}
	Let $G=\left\langle S\mid R\right\rangle $
	be a group given by a symmetric triangular presentation and let $G^{\prime}=\left\langle S^{\prime}\mid R^{\prime}\right\rangle $
	be a scrambling. Let $s\in S$ and let $G^{\prime\prime}=\left\langle S^{\prime\prime}\mid R^{\prime\prime}\right\rangle \simeq\left(G\ast F_R*\left\langle z_{1},\ldots,z_{4}\right\rangle \right)\times\mathbb{Z}^{N}$
	be the associated augmentation. Then for any $x,x',x'' \in S''$ such that
	\[\deg_{z}(x''x'x)=e\quad \text{and}\quad \pi_{F,\Z}^\mathrm{ab}(x'' x' x) =0,\]
	each of the six products 
	\[x''x'x,\ x''xx',\ x'x''x,\ x'xx'',\ xx''x',\ xx'x''\]
	over a permutation of $x'',x',x$ is either trivial or sufficiently generic relative to $S$.
	
	Further, the only pairs of elements $x,x' \in S''$ satisfying both $\deg_z(x)=\deg_z(x')$ and $\pi_{F,\mathbb{Z}}^\mathrm{ab}(x)=\pi_{F,\mathbb{Z}}^\mathrm{ab}(x')$ are $t,z_1$ and the inverse pair $t^{-1}, z_1^{-1}$. 
\end{prop}
\begin{rmrk}
	By assumption, the element $x''x'x$ in the statement satisfies $x''x'x \in (G \ast F_R \ast \{z_1,\ldots,z_4\})\times\{0\} \simeq G \ast (\langle z_1,\ldots,z_4 \rangle \ast F_R)$, so it makes sense to discuss sufficient genericity in $G \ast (\langle z_1,\ldots,z_4 \rangle \ast F_R)$ relative to the generating set $S$ of the presentation $G = \langle S\mid R\rangle$.
\end{rmrk}
\begin{proof}
	\Cref{tab:augmentation} contains, out of each pair
	of mutually inverse generators $\{x,x^{-1}\}$ of $S^{\prime\prime}$,
	an element with nonnegative degrees in $z_{1},\ldots,z_{4}$ and shows their degrees under the map $\deg_z$.
	\begin{table}[hbt]
		\addtolength{\tabcolsep}{0.9pt}
		\def\arraystretch{1.2}{
			\[
			\begin{array}{ll}
				\text{Generator \ensuremath{x}} & \deg_{z}\left(x\right)\\
				\hline \rowcolor{LightGray}
				z_{4} & \left(0,0,0,1\right)\\
				u_{z_{4},i}=j\left(b_{i}\right)z_{4} & \left(0,0,0,1\right)\\\rowcolor{LightGray}
				z_{3} & \left(0,0,1,0\right)\\
				u_{z_{3},i}=j\left(b_{i}\right)z_{3} & \left(0,0,1,0\right)\\\rowcolor{LightGray}
				z_{2} & \left(0,1,0,0\right)\\
				u_{z_{2},i}=j\left(b_{i}\right)z_{2} & \left(0,1,0,0\right)\\\rowcolor{LightGray}
				z_{1} & \left(1,0,0,0\right)\\
				u_{z_{1},i}=j\left(b_{i}\right)z_{1} & \left(1,0,0,0\right)\\\rowcolor{LightGray}
				z_{3}s^{\prime} & \left(0,0,1,0\right)\\
				z_{3}s^{\prime}z_{1} & \left(1,0,1,0\right)\\\rowcolor{LightGray}
				z_{3}s^{\prime}z_{1}s^{\prime} & \left(1,0,1,0\right)\\
				s^{\prime}z_{1}s^{\prime} & \left(1,0,0,0\right)\\\rowcolor{LightGray}
				s^{\prime}z_{1}s^{\prime}z_{2}\cdot\ldots b_{k_{i-1}}^{\varepsilon_{i-1}}\cdot z_{2}\quad \text{ for \ensuremath{1\le i\le r}} & \left(1,i,0,0\right)\\
				s^{\prime}z_{1}s^{\prime}z_{2}b_{k_{1}}^{\varepsilon_{1}}\ldots z_{2}b_{k_{i}}^{\varepsilon_{i}}\quad \text{ for \ensuremath{1\le i\le r}} & \left(1,i,0,0\right)\\\rowcolor{LightGray}
				z_{4}s^{\prime}z_{1}s^{\prime}z_{2}b_{k_{1}}^{\varepsilon_{1}}\ldots z_{1}b_{k_{r}}^{\varepsilon_{r}}z_{2}^{-i}\quad \text{ for \ensuremath{0\le i\le r}} & \left(1,r-i,0,1\right)\\
				sz_{1}s & \left(1,0,0,0\right)\\\rowcolor{LightGray}
				\text{any }x\in S^{\prime} & \left(0,0,0,0\right)
			\end{array}
			\]
		}
		\caption{The generators of $S''$ together with their degrees $\deg_{z}$.}
		\label{tab:augmentation}
	\end{table}

	Considering
		the values of $\deg_{z}\left(x\right)$ of the rows in \Cref{tab:augmentation} as vectors in $\mathbb{Z}^{4}$,
		we are thus looking for dependencies of the form $\pm R_{1}\pm R_{2}\pm R_{3}=0$ where $R_1,R_2,R_3$ are three of these vectors.
	Such dependencies correspond exactly to those triples $x,x',x'' \in S''$ such that
	\[\deg_{z}(x''x'x)=e.\]
	
	Since all rows in the table have nonnegative degrees (and no row is
	$0$ except the last, which corresponds to generators in $S^{\prime}$),
	at least one of the coefficients in such a dependence must be negative
	and at least one must be positive. Thus we may assume (by permuting
	the indices if necessary) that the equation has the form $R_{1}+R_{2}=R_{3}$.

	Given a dependence $R_{1}+R_{2}=R_{3}$, we consider triples of generators $x$,
	$x^{\prime}$, $x^{\prime\prime-1}$ with values under $\mathrm{deg}_z$ equal
	to $R_{1}$, $R_{2}$, and $R_{3}$, respectively, and also satisfying $\pi_{F,\Z}^\mathrm{ab}(x'' x' x) = 0$ (or, equivalently, $\pi_{F,\Z}^\mathrm{ab}(x) + \pi_{F,\Z}^\mathrm{ab}(x') = \pi_{F,\Z}^\mathrm{ab}(x^{\prime\prime -1})$.)
	For the rest of this proof, ``sufficiently generic'' is short for ``sufficiently generic in $G \ast (F_R \ast \langle z_1,\ldots,z_4\rangle)$ relative to $S$''.
	It suffices to check that $x'' x' x$ and $x' x'' x$ are either trivial or sufficiently generic, since each of the four other products $x x' x'',\ x x'' x',\ x'' x x',\ x' x x''$ is a cyclic shift of one of these, and cyclic shifts are conjugate to each other.

	We now enumerate all cases by going over the possible vectors of $R_{3}\in\mathbb{Z}^{4}$.
		\begin{enumerate}[label=\textbf{Case \arabic*:}, labelwidth=-8mm,  labelindent=2em,leftmargin =!]
		\item Suppose $R_{3}=\left(0,0,1,0\right)$. Then without loss of generality
		$R_{1}=\left(0,0,1,0\right)$ and $R_{2}=\left(0,0,0,0\right)$. 
		In this case $x^{\prime\prime-1}$ and $x$ are among $z_{3}$, $j\left(b_{i}\right)z_{3}$
		(for some $1\le i\le N$), and $z_{3}s^{\prime}$, while $x^{\prime}\in S^{\prime}$.
		It follows that 
		\[
		\pi_{F,\mathbb{Z}}^\mathrm{ab}\left(x^{\prime\prime}x\right)\in\left\{ 0,\pm b_{i},\pm\pi_{\mathbb{Z}}\left(s^{\prime}\right),\pm b_{i}\pm\pi_{\mathbb{Z}}\left(s^{\prime}\right)\right\} .
		\]
		
		\begin{enumerate}[label=\textbf{Case 1.\arabic*:}, labelwidth=-10mm,  labelindent=12mm,leftmargin =!]
			\item If $\pi_{F,\mathbb{Z}}^\mathrm{ab}\left(x^{\prime\prime}x\right)=0$ it follows
			that $x^{\prime\prime}=x^{-1}$. The relation $\pi_{\mathbb{Z}}\left(x^{\prime\prime}x^{\prime}x\right)=0$
			then yields $x^{\prime}=e$ by~\Cref{cor:scrambling}.
			Thus the elements $x''x'x$ and $x'x''x$ are both $e$.
			\item If $\pi_{F,\mathbb{Z}}^\mathrm{ab}\left(x^{\prime\prime}x\right)=\pm b_{i}$ we
			may assume $x^{\prime\prime-1}=z_{3}$ and $x=j\left(b_{i}\right)z_{3}$.
			\Cref{cor:scrambling} yields $x^{\prime}=j\left(b_{i}\right)^{-1}$.
			Since $j\left(b_{i}\right)$ commutes with $z_{3}$, the elements $x''x'x$ and $x'x''x$ are both $e$.
			\item If $\pi_{F,\mathbb{Z}}^\mathrm{ab}\left(x^{\prime\prime}x\right)=\pm\pi_{F,\mathbb{Z}}^\mathrm{ab}\left(s^{\prime}\right)$
			we may assume $x^{\prime\prime-1}=z_{3}$ and $x=z_{3}s^{\prime}$.
			Thus $\pi_{F,\mathbb{Z}}^\mathrm{ab}\left(x^{\prime}\right)=-\pi_{F,\mathbb{Z}}^\mathrm{ab}\left(s^{\prime}\right)$,
			and by \Cref{cor:scrambling} we must have $x^{\prime}=s^{\prime-1}$.
			In this case we obtain
			\[
			x''x'x=z_{3}^{-1}s^{\prime-1}z_{3}s=\left[z_{3},s^{\prime-1}\right]\quad\text{and}\quad x'x''x=s^{\prime-1}z_{3}^{-1}z_{3}s^{\prime}=e.
			\]
			We show in~\Cref{lem:augmentation_words} that $\left[z_{3},s^{\prime-1}\right]$ is sufficiently generic.
			\item If $\pi_{F,\mathbb{Z}}^\mathrm{ab}\left(x^{\prime\prime}x\right)=\pm b_{i}\pm\pi_{F,\mathbb{Z}}^\mathrm{ab}\left(s^{\prime}\right)$,
			we may assume $x^{\prime\prime-1}=j\left(b_{i}\right)z_{3}$ and $x=z_{3}s^{\prime}$.
			Thus $\pi_{F,\mathbb{Z}}^\mathrm{ab}\left(x^{\prime}\right)=b_{i}-\pi_{\mathbb{Z}}\left(s^{\prime}\right)$,
			and by \Cref{cor:scrambling} we must have $x^{\prime}=j\left(b_{i}\right)s^{\prime-1}$.
			In this case the elements $x^{\prime\prime},x^{\prime},x$ differ
			from the elements of the previous case by $j\left(b_{i}\right)^{\pm1}$.
			Since $j\left(b_{i}\right)$ commutes with all other generators and
			thus cancels in the elements $x''x'x$ and $x'x''x$,
			the computation is the same as in the previous case.
		\end{enumerate}
		\item Suppose $R_{3}=\left(0,1,0,0\right)$ or $R_{3}=\left(0,0,0,1\right)$.
		Then without loss of generality $R_{1}=R_{3}$ and $R_{2}=\left(0,0,0,0\right)$.
		Each possibility in this case was checked in the $R_{3}=\left(0,0,1,0\right)$-case with $z_{3}$
		in place of $z_{2}$ or $z_{4}$.
		\item Suppose $R_{3}=\left(1,0,0,0\right)$. Then without loss of generality $R_{1}=\left(1,0,0,0\right)$
		and $R_{2}=\left(0,0,0,0\right)$. 
		In this case $x^{\prime\prime-1}$ and $x$ are among $z_{1}$, $j\left(b_{i}\right)z_{1}$
		(for some $1\le i\le N$), $s^{\prime}z_{1}s^{\prime}$, and $sz_{1}s$,
		while $x^{\prime}\in S^{\prime}$. It follows that
		\[
		\pi_{F,\mathbb{Z}}^\mathrm{ab}\left(x^{\prime\prime}x\right)\in\left\{ 0,\pm b_{i},\pm2\pi_{\mathbb{Z}}\left(s^{\prime}\right),\pm\left(2\pi_{\mathbb{Z}}\left(s^{\prime}\right)-b_{i}\right)\right\} .
		\]
		There is no $x^{\prime}\in S^{\prime}$ with $\pi_{F,\mathbb{Z}}^\mathrm{ab}\left(x^{\prime}\right)\in\left\{ \pm2\pi_{\mathbb{Z}}\left(s^{\prime}\right),\pm\left(2\pi_{\mathbb{Z}}\left(s^{\prime}\right)-b_{i}\right)\right\} $
		by property \ref{it:sc8} of scramblings, so either $x^{\prime\prime}=x^{-1}$
		(in which case $x^{\prime}=e$ and the two elements $x''x'x$ and $x'x''x$ are $e$) or one the following cases occurs:
		\begin{enumerate}[label=\textbf{Case 3.\arabic*:}, labelwidth=-10mm,  labelindent=12mm,leftmargin =!]
			\item $\left\{ x^{\prime\prime-1},x\right\} =\left\{ z_{1},j\left(b_{i}\right)z_{1}\right\} $:
			this case was considered in the $R_{3}=\left(0,0,1,0\right)$-case, with $z_{3}$ in place of $z_{1}$.
			\item $\left\{ x^{\prime\prime-1},x\right\} =\left\{ z_{1},sz_{1}s\right\} $.
			In this case $x^{\prime}=e$ and the two elements $x''x'x$ and $x'x''x$ are both conjugate to  $sz_{1}sz_{1}^{-1}$ of which we show in~\Cref{lem:augmentation_words} that it is sufficiently generic.
			\item $\left\{ x^{\prime\prime-1},x\right\} =\left\{ j\left(b_{i}\right)z_{1},sz_{1}s\right\}$.
			In this case $x^{\prime}=j\left(b_{i}\right)^{\pm1}$, and since $j\left(b_{i}\right)$
			commutes with all other generators the resulting elements are just
			those of the previous case.
		\end{enumerate}
		\item Suppose $R_{3}=\left(1,0,1,0\right)$. Then $x^{\prime\prime-1}$ is either
		$z_{3}s^{\prime}z_{1}$ or $z_{3}s^{\prime}z_{1}s^{\prime}$. There
		are two cases to consider:
		\begin{enumerate}[label=\textbf{Case 4.\arabic*:}, labelwidth=-10mm,  labelindent=12mm,leftmargin =!]
			\item $R_{1}=\left(1,0,0,0\right)$ and $R_{2}=\left(0,0,1,0\right)$.
			
			In this case $x^{\prime}$ is either $z_{3}$ or $j\left(b_{i}\right)z_{3}$
			for some $1\le i\le N$, and $x$ is one of $z_{1}$, $j\left(b_{k}\right)z_{1}$
			(for some $1\le k\le N$), $s^{\prime}z_{1}s^{\prime}$, and $sz_{1}s$.
			We consider the possibilities for $x^{\prime\prime-1}$:
			\begin{enumerate}[label=\textbf{Case 4.1.\arabic*:}, labelwidth=-12mm,  labelindent=14mm,leftmargin =!]
				\item $x^{\prime\prime-1}=z_{3}s^{\prime}z$: Since $\pi_{F,\mathbb{Z}}^\mathrm{ab}\left(s^{\prime}\right)$
				is not of the form $\pm b_{i}$ $\pm b_{i}\pm b_{i}$, or $\pm2\pi_{\mathbb{Z}}\left(s^{\prime}\right)\pm b_{i}$
				by property \ref{it:sc8} of scramblings, it is impossible to obtain $\pi_{F,\mathbb{Z}}^\mathrm{ab}\left(x^{\prime\prime}x^{\prime}x\right)=0$
				in this case. 
				\item $x^{\prime\prime-1}=z_{3}s^{\prime}z_{1}s^{\prime}$: As in the previous
				case, to obtain $\pi_{F,\mathbb{Z}}^\mathrm{ab}\left(x^{\prime\prime}x^{\prime}x\right)=0$
				we must have $x^{\prime}=z_{3}$ and $x=s^{\prime}z_{1}s^{\prime}$.

					In this case we obtain
				\begin{align*}
					x''x'x&=\left(z_{3}s^{\prime}z_{1}s^{\prime}\right)^{-1}z_{3}\left(s^{\prime}z_{1}s^{\prime}\right)=e\,\text{and}\\
					x'x''x&=z_{3}\left(z_{3}s^{\prime}z_{1}s^{\prime}\right)^{-1}\left(s^{\prime}z_{1}s^{\prime}\right)=\left[z_{3},\left(s^{\prime}z_{1}s^{\prime}\right)^{-1}\right].
				\end{align*}

					We prove that $\left[z_{3},\left(s^{\prime}z_{1}s^{\prime}\right)^{-1}\right]$ is sufficiently generic in~\Cref{lem:augmentation_words}.

			\end{enumerate}
		\item Suppose $R_{1}=\left(1,0,1,0\right)$ and $R_{2}=\left(0,0,0,0\right)$.
		In this case $x$ is either $z_{3}s^{\prime}z_{1}$ or $z_{3}s^{\prime}z_{1}s^{\prime}$
		and $x^{\prime}\in S^{\prime}$. We consider the possibilities for
		$x^{\prime\prime-1}$:
		\begin{enumerate}[label=\textbf{Case 4.2.\arabic*:}, labelwidth=-12mm,  labelindent=14mm,leftmargin =!]
			\item $x^{\prime\prime-1}=z_{3}s^{\prime}z_{1}$: If $x=z_{3}s^{\prime}z_{1}$
			then $\pi_{F,\mathbb{Z}}^\mathrm{ab}\left(x^{\prime\prime}x\right)=0$, and by~\Cref{cor:scrambling} we must have $x^{\prime}=e$.
			In this case the elements $x''x'x$ and $x'x''x$ are both $e$.
			If $x=z_{3}s^{\prime}z_{1}s^{\prime}$ then $\pi_{F,\mathbb{Z}}^\mathrm{ab}\left(x^{\prime\prime}x\right)=\pi_{F,\mathbb{Z}}^\mathrm{ab}\left(s^{\prime}\right)$,
			and again by~\Cref{cor:scrambling} we must have $x^{\prime}=s^{\prime-1}$.
			Thus we obtain
			\begin{align*}
				x''x'x=\left(z_{3}s^{\prime}z_{1}\right)^{-1}s^{\prime-1}\left(z_{3}s^{\prime}z_{1}s^{\prime}\right)&=\left[\left(z_{3}s^{\prime}z_{1}\right)^{-1},s^{\prime-1}\right]\mbox{ and}\\
				x'x''x=s^{\prime-1}\left(z_{3}s^{\prime}z_{1}\right)^{-1}\left(z_{3}s^{\prime}z_{1}s^{\prime}\right)&=e.
			\end{align*}
			We prove that $\left[\left(z_{3}s^{\prime}z_{1}\right)^{-1},s^{\prime-1}\right]$
			is sufficiently generic in~\Cref{lem:augmentation_words}.
			\item $x^{\prime\prime-1}=z_{3}s^{\prime}z_{1}s^{\prime}$. If $x=z_{3}s^{\prime}z_{1}$
			then by exchanging the roles of $x$ and $x^{\prime\prime}$ (and
			inverting all three generators) we reduce to the previous case. If
			$x=z_{3}s^{\prime}z_{1}s^{\prime}$ then $\pi_{F,\mathbb{Z}}^\mathrm{ab}\left(x^{\prime\prime}x\right)=0$,
			and by~\Cref{cor:scrambling} we must have $x^{\prime}=e$.
			In this case the elements $x''x'x$ and $x'x''x$ are both $e$.
		\end{enumerate}
	\end{enumerate}
	
	\item Suppose $R_{3}=\left(1,i,0,0\right)$ for $1\le i\le r$. There are two
	cases to consider in this case:
	\begin{enumerate}[label=\textbf{Case 5.\arabic*:}, labelwidth=-10mm,  labelindent=12mm,leftmargin =!]
		\item $R_{1}=\left(1,i-1,0,0\right)$ and $R_{2}=\left(0,1,0,0\right)$.
		
		In this case $x^{\prime\prime-1}$ is either $s^{\prime}z_{1}s^{\prime}z_{2}\ldots b_{k_{i-1}}^{\varepsilon_{i-1}} z_{2}$
		or $s^{\prime}z_{1}s^{\prime}z_{2}\dots z_{2}b_{k_{i}}^{\varepsilon_{i}}$.
		If $i>1$, $x$ is either $s^{\prime}z_{1}s^{\prime}z_{2}\ldots z_{2}b_{k_{i-1}}^{\varepsilon_{i-1}}$
		or $s^{\prime}z_{1}s^{\prime}z_{2}\ldots\cdot z_{2}$, and thus
		$\pi_{F,\mathbb{Z}}^\mathrm{ab}\left(x^{\prime\prime}x\right)$ is either $0$,
		$-\varepsilon_{i}b_{k_{i}}$, $-\varepsilon_{i-1}b_{k_{i-1}}$, or
		$-\varepsilon_{i-1}b_{k_{i-1}}-\varepsilon_{i}b_{k_{i}}$. The generator
		$x^{\prime}$ must be either $z_{2}$ or $j\left(b_{k}\right)z_{2}$
		for some $1\le k\le N$, so in the case $\pi_{F,\mathbb{Z}}^\mathrm{ab}\left(x^{\prime\prime}x\right)=-\varepsilon_{i-1}b_{k_{i-1}}-\varepsilon_{i}b_{k_{i}}$
		there is nothing to check (because it implies $\pi_{F,\mathbb{Z}}^\mathrm{ab}\left(x^{\prime\prime}x^{\prime}x\right)\neq0$).
		In each of the other cases, all elements $b_{j}^{\varepsilon_{j}}$
		and $j\left(b_{k}\right)$ vanish from the product (because they commute
		with all other generators, and cancel out).
		Thus we obtain
		\begin{align*}
			x''x'x=\left(s^{\prime}z_{1}s^{\prime}z_{2}^{i}\right)^{-1}z_{2}\left(s^{\prime}z_{1}s^{\prime}z_{2}^{i-i}\right)&=\left[\left(s^{\prime}z_{1}s^{\prime}z_{2}^{i}\right)^{-1},z_{2}\right]\mbox{ and}\\
			x'x''x=z_{2}\left(s^{\prime}z_{1}s^{\prime}z_{2}^{i}\right)^{-1}\left(s^{\prime}z_{1}s^{\prime}z_{2}^{i-i}\right)&=e.
		\end{align*}
		We prove that $\left[\left(s^{\prime}z_{1}s^{\prime}z_{2}^{i}\right)^{-1},z_{2}\right]$
		is sufficiently generic in~\Cref{lem:augmentation_words}.
		
		If $i=1$ we have that $\pi_{F,\mathbb{Z}}^\mathrm{ab}\left(x^{\prime\prime}\right)$
		equals either $-2\pi_{F,\mathbb{Z}}^\mathrm{ab}\left(s\right)$ or $2\pi_{F,\mathbb{Z}}^\mathrm{ab}\left(s\right)-\varepsilon_{k_{1}}b_{k_{1}}$,
		and $x$ may also equal one of $z_{1}$, $j\left(b_{k}\right)z_{1}$
		(for some $1\le k\le N$), $s^{\prime}z_{1}s^{\prime}$, and $sz_{1}s$
		(the other possible values for $x$ have been dealt with in the case
		$i>1$). Thus $\pi_{F,\mathbb{Z}}^\mathrm{ab}\left(x\right)$ is one of $0$, $b_{k}$
		(for some $1\le k\le N$), and $2\pi_{F,\mathbb{Z}}^\mathrm{ab}\left(s^{\prime}\right)$.
		By property \ref{it:sc8} of scramblings, if $\pi_{F,\mathbb{Z}}^\mathrm{ab}\left(x\right)\neq2\pi_{F,\mathbb{Z}}^\mathrm{ab}\left(s^{\prime}\right)$
		then $\pi_{F,\mathbb{Z}}^\mathrm{ab}\left(x^{\prime\prime}x^{\prime}x\right)\neq0$
		(note that $x^{\prime}$ is either $z_{2}$ or its product with some
		$j\left(b_{k^{\prime}}\right)$, and thus $\pi_{F,\mathbb{Z}}^\mathrm{ab}\left(x^{\prime}\right)$
		is either $0$ or some basis element). Therefore we need only consider
		the case where $x=s^{\prime}z_{1}s^{\prime}$.
		Since all basis elements
		of $\mathbb{Z}^{N}$ cancel in the product, it suffices to compute
		the elements $x''x'x$ and $x'x''x$
		for $x^{\prime\prime-1}=s^{\prime}z_{1}s^{\prime}z_{2}$, $x=s^{\prime}z_{1}s^{\prime}$,
		and $x^{\prime}=z_{2}$.
		This yields
		\begin{align*}
			x''x'x&=\left(s^{\prime}z_{1}s^{\prime}z_{2}\right)^{-1}z_{2}\left(s^{\prime}z_{1}s^{\prime}\right)=\left[z_{2}^{-1},s^{\prime}z_{1}s^{\prime}\right]\mbox{ and}\\
			x'x''x&=z_{2}\left(s^{\prime}z_{1}s^{\prime}z_{2}\right)^{-1}\left(s^{\prime}z_{1}s^{\prime}\right)=e.
		\end{align*}
		We prove that $\left[z_{2}^{-1},s^{\prime}z_{1}s^{\prime}\right]$ is sufficiently generic
		in~\Cref{lem:augmentation_words}.
		\item $R_{1}=\left(1,i,0,0\right)$ and $R_{2}=\left(0,0,0,0\right)$.
		
		In this case, $x^{\prime}\in S^{\prime}$ while $x$ and $x^{\prime\prime-1}$
		are each equal to one of $s^{\prime}z_{1}s^{\prime}z_{2}\dots b_{k_{i-1}}^{\varepsilon_{i-1}} z_{2}$
		and $s^{\prime}z_{1}s^{\prime}z_{2}\dots z_{2}b_{k_{i}}^{\varepsilon_{i}}$.
		It follows that $\pi_{F,\mathbb{Z}}^\mathrm{ab}\left(x^{\prime\prime}x\right)\in\left\{ \pm\varepsilon_{i}b_{k_{i}},0\right\} $
		and therefore by~\Cref{cor:scrambling} $x^{\prime}=e$ or
		$x^{\prime}=b_{k_{i}}^{\pm1}$.
		In all cases, the elements $x''x'x$ and $x'x''x$ are both $e$.
	\end{enumerate}
	\item Suppose $R_{3}=\left(1,i,0,1\right)$ for $0\le i\le r$. There are three
	cases to consider in this case. In all of them we must have $x^{\prime\prime-1}=z_{4}s^{\prime}z_{1}s^{\prime}z_{2}b_{k_{1}}^{\varepsilon_{1}}\dots z_{2}b_{k_{r}}^{\varepsilon_{r}}z_{2}^{r-i}=z_{4}sz_{1}sz_{2}^{i}$.
	\begin{enumerate}[label=\textbf{Case 6.\arabic*:}, labelwidth=-10mm,  labelindent=12mm,leftmargin =!]
		\item $R_{1}=\left(1,i-1,0,1\right)$ and $R_{2}=\left(0,1,0,0\right)$
		(if $i\neq0$).
		
		In this case $x=z_{4}s^{\prime}z_{1}s^{\prime}z_{2}b_{k_{1}}^{\varepsilon_{1}}\ldots z_{2}b_{k_{r}}^{\varepsilon_{r}}z_{2}^{r-i+1}=z_{4}szs_{1}z_{2}^{i-1}$.
		Since $\pi_{F,\mathbb{Z}}^\mathrm{ab}\left(x^{\prime\prime}x\right)=0$ we must
		also have $\pi_{F,\mathbb{Z}}^\mathrm{ab}\left(x^{\prime}\right)=0$. Thus $x^{\prime}=z_{2}$.
		This yields
		\begin{align*}
			x''x'x=&\left(z_{4}sz_{1}sz_{2}^{i}\right)^{-1}z_{2}\left(z_{4}sz_{1}sz_{2}^{i-1}\right)=\left[\left(z_{4}sz_{1}sz_{2}^{i}\right)^{-1},z_{2}\right] \mbox{ and}\\
			x'x''x=&z_{2}\left(z_{4}sz_{1}sz_{2}^{i}\right)^{-1}\left(z_{4}sz_{1}sz_{2}^{i-1}\right)=e.
		\end{align*}
		We prove that $\left[\left(z_{4}sz_{1}sz_{2}^{i}\right)^{-1},z_{2}\right]$
		is sufficiently generic in~\Cref{lem:augmentation_words}.
		\item $R_{1}=\left(1,i,0,0\right)$ and $R_{2}=\left(0,0,0,1\right)$.
		
		Suppose first $i\neq0$. Then $x$ is equal to one of the words $s^{\prime}z_{1}s^{\prime}z_{2}\dots b_{k_{i-1}}^{\varepsilon_{i-1}} z_{2}$
		and $s^{\prime}z_{1}s^{\prime}z_{2}\dots b_{k_{i-1}}^{\varepsilon_{i-1}} z_{2}b_{k_{i}}^{\varepsilon_{i}}$,
		while $x^{\prime}$ is either $z_{4}$ or $j\left(b_{k}\right)z_{4}$
		for some $1\le k\le N$. Supposing $\pi_{F,\mathbb{Z}}^\mathrm{ab}\left(x^{\prime\prime}x^{\prime}x\right)=0$,
		we may ignore any basis element of $\mathbb{Z}^{N}$ in $x''x'x$ and $x'x''x$
		as these basis elements commute with all other generators and cancel
		each other.
		Modulo the $\mathbb{Z}^{N}$ factor, $x$ is equivalent
		to $sz_{1}s$ and $x^{\prime}$ is equivalent to $z_{4}$.
		This yields
		\begin{align*}
			x''x'x&=\left(z_{4}sz_{1}s\right)^{-1}z_{4}\left(sz_{1}s\right)=e\mbox{ and}\\
			x'x''x&=z_{4}\left(z_{4}sz_{1}s\right)^{-1}\left(sz_{1}z\right)=\left[z_{4},\left(sz_{1}s\right)^{-1}\right].
		\end{align*}
		We prove that $\left[z_{4},\left(sz_{1}s\right)^{-1}\right]$ is sufficiently generic in~\Cref{lem:augmentation_words}.
		
		If $i=0$ there are more possibilities for $x$: it may additionally
		be one of $z_{1}$, $j\left(b_{k^{\prime}}\right)z_{1}$ (for some
		$1\le k^{\prime}\le N$), $s^{\prime}z_{1}s^{\prime}$, and $sz_{1}s$.
		The possibilities for $x^{\prime}$ remain the same. Again, assuming
		$\pi_{F,\mathbb{Z}}^\mathrm{ab}\left(x^{\prime\prime}x^{\prime}x\right)=0$, we
		may work modulo the $\mathbb{Z}^{N}$ factor; thus the last two possibilities
		for $x$ are both equivalent to $sz_{1}s$, which has already been
		considered. The first two possibilities for $x$ are equivalent to
		$z_{1}$.
		Thus the elements $x''x'x$ and $x'x''x$
		are equivalent to one of $e$, $\left[z_{4},\left(sz_{1}s\right)^{-1}\right]$,
		\begin{align*}
			\left(z_{4}sz_{1}s\right)^{-1}z_{4}z_{1}&=s^{-1}z_{1}^{-1}s^{-1}z_{1} \mbox{ and}\\
			z_{4}\left(z_{4}sz_{1}s\right)^{-1}z_{1}&=z_{4}\left(s^{-1}z_{1}^{-1}s^{-1}\right)z_{4}^{-1}z_{1}.
		\end{align*}
		We prove that these are sufficiently generic
		in~\Cref{lem:augmentation_words}.
		\item $R_{1}=\left(1,i,0,1\right)$ and $R_{2}=\left(0,0,0,0\right)$.
		
		In this case $x^{\prime}=e$ and $x^{\prime\prime-1}=x$. 
		Thus the elements $x''x'x$ and $x'x''x$ are both $e$.
	\end{enumerate}
	\item Suppose $R_{3}=\left(0,0,0,0\right)$. Then also $R_{1}=R_{2}=\left(0,0,0,0\right)$.
	Since all three generators are then in $S^{\prime}$, there is nothing
	to check: condition (b) holds by property \ref{it:sc5} of scramblings. \qedhere
\end{enumerate}

We now verify the second part of the statement, on pairs $x,x^{\prime}\in S^{\prime\prime}$
satisfying both $\deg_{z}\left(x\right)=\deg_{z}\left(x^{\prime}\right)$
and $\pi_{F,\mathbb{Z}}^{\mathrm{ab}}\left(x\right)=\pi_{F,\mathbb{Z}}^{\mathrm{ab}}\left(x^{\prime}\right)$.
For each possible value of $\deg_{z}\left(x\right)$ we consider the
corresponding set of rows:
\begin{enumerate}[label=\textbf{Case \arabic*:}, labelwidth=-8mm,  labelindent=2em,leftmargin =!]
	\item $\deg_z(x) = \left(0,0,0,1\right)$: the corresponding generators are $z_{4}$
	and those generators of the form $j\left(b_{i}\right)z_{4}$. Any
	two of these have distinct values under $\pi_{F,\mathbb{Z}}^{\mathrm{ab}}$
	(because $\pi_{F,\mathbb{Z}}^{\mathrm{ab}}\left(z_{4}\right)=0$ and
	$\pi_{F,\mathbb{Z}}^{\mathrm{ab}}\left(j\left(b_{i}\right)\right)$
	takes different values for different indices $i$, all of which are
	nonzero).
	\item $\deg_z(x) = \left(0,0,1,0\right)$: the corresponding generators are $z_{3}$
	and those of the form $j\left(b_{i}\right)z_{3}$, and the verification
	is the same as that for $\left(0,0,0,1\right)$.
	\item $\deg_z(x) = \left(0,1,0,0\right)$ is identical to the previous case, with $z_{2}$
	replacing $z_{3}$.
	\item $\deg_z(x) = \left(1,0,1,0\right)$: the possible generators are $z_{3}s^{\prime}z_{1}$
	and $z_{3}s^{\prime}z_{1}s^{\prime}$. These have different values
	under $\pi_{F,\mathbb{Z}}^{\mathrm{ab}}$.
	\item $\deg_z(x) = \left(1,i,0,0\right)$ for some $1\le i\le r$: the possible generators
	are $s^{\prime}z_{1}s^{\prime}z_{2}\cdot\ldots b_{k_{i-1}}^{\varepsilon_{i-1}}\cdot z_{2}$
	and $s^{\prime}z_{1}s^{\prime}z_{2}b_{k_{1}}^{\varepsilon_{1}}\ldots z_{2}b_{k_{i}}^{\varepsilon_{i}}$.
	These have different values under $\pi_{F,\mathbb{Z}}^{\mathrm{ab}}$
	(differing by $\pi_{F,\mathbb{Z}}^{\mathrm{ab}}\left(b_{k_{i}}^{\varepsilon_{i}}\right)\neq0$).
	\item $\deg_z(x) = \left(1,0,0,0\right)$: the corresponding generators are $z_{1}$,
	generators of the form $j\left(b_{i}\right)z_{1}$, $s^{\prime}z_{1}s^{\prime}$,
	and $sz_{1}s$. Except for the pair $\left\{ z_{1},t=sz_{1}s\right\} $,
	any two of these have distinct values under $\pi_{F,\mathbb{Z}}^{\mathrm{ab}}$
	(this follows directly from \ref{it:sc8}). 
	\item $\deg_z(x) = \left(1,r-i,0,1\right)$ for some $0\le i\le r$: there is only one generator
	with this degree.
	\item $\deg_z(x) = \left(0,0,0,0\right)$: these cases follow from property \ref{it:sc4} in the definition of group scrambling.
\end{enumerate}
\end{proof}

\begin{lemm}\label{lem:augmentation_words}
	In the notation of~\Cref{prop:trivial_or_generic_products} and denoting the commutator of $x$ and $y$ by $\left[x,y\right] = xyx^{-1}y^{-1}$, the following elements are sufficiently generic relative to $S$:
	\begin{enumerate}[(1)]
		\item $sz_{1}sz_{1}^{-1}$, $s^{-1}z_{1}^{-1}s^{-1}z_{1}$, and $z_{1}\left(s^{-1}z_{1}^{-1}s^{-1}\right)$,
		\item $\left[z_{3},s^{\prime-1}\right]$ and $\left[\left(z_{3}s^{\prime}z_{1}\right)^{-1},s^{\prime-1}\right]$, and
		\item $\left[z_{3},\left(sz_{1}s\right)^{-1}\right]$, $\left[\left(sz_{1}sz_{2}^{i}\right)^{-1},z_{2}\right]$, $\left[z_{2}^{-1},sz_{1}s\right]$, $\left[\left(z_{4}sz_{1}sz_{2}^{i}\right)^{-1},z_{2}\right]$, and $\left[z_{4},\left(sz_{1}s\right)^{-1}\right]$,
		\item $z_{4}\left(s^{-1}z_{1}^{-1}s^{-1}\right)z_{4}^{-1}z_{1}$.
	\end{enumerate}
\end{lemm}
\begin{proof}	
	For the following, let $\rho: F_S \ast \langle z_1,\ldots,z_4 \rangle \to \mathrm{GL}_n(\C)$ be any homomorphism mapping each $s\in S$ to the permutation matrix of a derangement or to $I_n$ and mapping $z_1,\ldots,z_4$ to matrices with algebraically independent entries over $\Q$. 
	
	Observe that the value of each of the elements in the lemma's statement under $\pi_{\mathbb{Z}}$
	is $0$. Therefore, all occurrences of $s^{\prime}$ can be replaced
	with $s$ without changing the words' value in the group $G^{\prime\prime}$. We then obtain elements in $G''$ which are words in $s$ and $z_1,\ldots,z_4$, and it suffices to show that $\rho(w)-I_n$ is invertible or $0$ for $w$ each of these words. 
	
	\begin{enumerate}[(1)]
		\item The inverse of $z_{1}\left(s^{-1}z_{1}^{-1}s^{-1}\right)$ is $sz_{1}sz_{1}^{-1}$ and the inverse of $s^{-1}z_{1}^{-1}s^{-1}z_{1}$
	 	is conjugate to $sz_{1}sz_{1}^{-1}$.
	 	So only one of these words needs to be checked by~\Cref{def:sufficiently_generic}.
		The element $sz_{1}sz_{1}^{-1}$ is sufficiently generic by~\Cref{lem:sufficiently_generic}~\ref{it:suff1}. 
		
		Note that if $\rho(s)$ is the permutation matrix of a derangement then, for each of these words $w$, by applying~\Cref{lem:sufficiently_generic} as above we find that $\rho(w) - I_n$ is invertible (and not zero).
		\item Following the above remark it is enough to consider the elements
		\[\left[z_{3},s^{-1}\right]\mbox{ and }\left[\left(z_{3}sz_{1}\right)^{-1},s^{-1}\right].\]
		The first is sufficiently generic by~\Cref{lem:sufficiently_generic}~\ref{it:suff2}.
		For $w = \left[\left(z_{3}sz_{1}\right)^{-1},s^{-1}\right]$, observe that
		if $\rho(s) = I_n$ then $\rho(w)-I_n = 0$. Otherwise, $\rho(s)$ is the permutation matrix of a derangement. By \cref{cor:transcendental_entries}, if there exist matrices $B_1,B_2$ such that $[(B_1 \rho(s) B_2)^{-1}, \rho(s)^{-1}] - I_n$ is invertible then so is $\rho(w) - I_n$. Taking $B_2 = \rho(s)^-1$ and $B_1 = \rho(z_3)^{-1}$, we obtain 
		\[[(B_1 \rho(s) B_2)^{-1}, \rho(s)^{-1}] - I_n = [\rho(z_3), \rho(s)^{-1}],\]
		which (since $\rho(s)^{-1}$ is the permutation matrix of a derangement) is again invertible by ~\Cref{lem:sufficiently_generic}~\ref{it:suff2}.

		\item 
		Let $w$ be any of these commutators and consider $\rho\left(w\right)$
		as a matrix with entries which are polynomials in the entries of the
		matrices $\left\{ \rho\left(z_{i}\right)\right\} _{i=1}^{4}$.
		It is clear that for any pair of invertible matrices $A,B$ we can
		arrange for $\rho\left(w\right)$ to equal $\left[A,B\right]$
		by choosing the entries of $\left\{ \rho\left(z_{i}\right)\right\} _{i=1}^{4}$
		appropriately. For example, for $\left[\left(z_{4}sz_{1}sz_{2}^{i}\right)^{-1},z_{2}\right]$
		we can set $\rho\left(z_{2}\right)=B$, $\rho\left(z_{1}\right)=I$,
		and take $\rho\left(z_{4}\right)$ to be the unique matrix
		such that $\rho\left(\left(z_{4}sz_{1}sz_{2}^{i}\right)^{-1}\right)=A$.
		Similarly, for $\left[z_{4},\left(sz_{1}s\right)^{-1}\right]$ we
		can set $\rho\left(z_{4}\right)=A$ and take $\rho\left(z_{1}\right)$
		to be the unique matrix such that $\rho\left(\left(sz_{1}s\right)^{-1}\right)=B$.
		If we take matrices $A,B$ that have algebraically independent entries then $[A,B]-I_n$ is invertible (for instance by~\Cref{lem:sufficiently_generic}~\ref{it:suff2}), and hence so is $\rho(w)-I_n$.
		\item Denote $w = z_{4}\left(s^{-1}z_{1}^{-1}s^{-1}\right)z_{4}^{-1}z_{1}$. If $\rho(s) = I_n$ then $\rho(w) - I_n = \rho([z_4,z_1^{-1}]) - I_n$ is invertible. Otherwise, $\rho(s)$ is the permutation matrix of a derangement. Considering $\rho(w) - I_n$ as a matrix with entries which are polynomials in the entries of $\rho(z_4)$, we can substitute $I_n$ for $\rho(z_4)$ to obtain $\rho(s^{-1}z_1^{1}s^{-1}z_1) - I_n$, which is invertible by case (1).
		\qedhere
	\end{enumerate}
\end{proof}

\subsection{The scrambling construction}\label{sec:scrambling_construction}
We describe a construction fulfilling the axioms for group scramblings
(see \cref{def:scrambling}).

\begin{construction}\label{con:scrambling}
	Let $G=\left\langle S\mid R\right\rangle $
	be a group given by a symmetric triangular presentation.
	We construct a finitely presented group $G^{\prime}=\left\langle S^{\prime}\mid R^{\prime}\right\rangle $
	together with an isomorphism $\varphi:G^{\prime}\rightarrow G\times\mathbb{Z}^{S\sqcup R}$
	in a sequence of steps. In each step (except the first preprocessing
	step) a group $G_{i}=\left\langle S_{i}\mid R_{i}\right\rangle $
	and a homomorphism $\varphi_{i}:G_{i}\rightarrow G\times\mathbb{Z}^{S\sqcup R}$
	is constructed. It is always the case that $S_{i}\subset S_{i+1}$,
	$R_{i}\subset R_{i+1}$, and $\varphi_{i+1}\restriction_{S_{i}}=\varphi_{i}\restriction_{S_{i}}$.
	We take $G^{\prime}$ and $\varphi$ to be the group presentation
	and homomorphism of the last step.
	
	In what follows we denote by $B=\left\{ b_{s}\right\} _{s\in S}\cup\left\{ b_{r}\right\} _{r\in R}$
	a basis for $\mathbb{Z}^{S\sqcup R}$.
	\begin{enumerate}[(CS1), labelwidth=-6mm,  labelindent=2em,leftmargin =!]
		\item\label{it:cs1} (A preprocessing step.) We modify $\left\langle S\mid R\right\rangle $
		to arrange that no relation $abc=e$ in $R$ contains the same generator
		twice (though it may contain a generator and its inverse) as follows.
		If some $s\in S$ appears twice or three times in some relation in
		$R$, add new elements $s'$ and $s^{\prime\prime}$ to $S$, and add
		the relations $ss^{\prime-1}e=e$ and $ss^{\prime\prime-1}e=e$ to
		$R$. Then, in any relation in which $s$ appears more than once,
		replace the second (and if present, the third) occurrence by $s^{\prime}$
		(or $s^{\prime\prime}$). Repeat this process until each relation
		is a product of three distinct generators (a generator and its inverse
		are considered distinct for this purpose). Then symmetrize the set
		of relations. It is clear how the resulting finitely presented group
		is isomorphic to the original one.
		\item\label{it:cs2} For each $s\in S\setminus\left\{ e\right\} $ define symbols $x_{s}$,
		$x_{s}^{-1}$. We call $x_{s}^{-1}$ the \emph{formal inverse} of
		$x_{s}$. We consider mutually inverse generators $s,s^{-1}$ in $S$
		as distinct for this purpose. In particular, for any such pair there
		are four symbols: $x_{s}$, $x_{s^{-1}}$, $x_{s}^{-1}$, and $x_{s^{-1}}^{-1}$.
		Furthermore, define symbols $w_r$, $w_r^{-1}$ for each $r\in R$.		
		Set 
		\begin{align*}
		S_{0}=&\left\{ x_{s}, x_{s}^{-1}\right\} _{s\in S}\cup \left\{w_{r}, w_{r}^{-1}\right\} _{r\in R}\cup\left\{ e\right\} ,\\
		R_{0}=&\left\{ x_{s}x_{s}^{-1}e=e\right\} _{s\in S}\cup\left\{ w_{r}w_{r}^{-1}e=e\right\}_{r\in R} ,
		\end{align*}
		and $G_{0}=\left\langle S_{0}\mid R_{0}\right\rangle $, so that $G_{0}$
		is a free group on $\left|S\right|+|R|$ generators (note: $x_{s}$ and
		$x_{s^{-1}}$ are \emph{not} inverses in $G_{0}$ for any pair $s,s^{-1}\in S$).
		Then define
		\[
		\varphi_{0}:G_{0}\rightarrow (G\ast F_R)\times\mathbb{Z}^{S\sqcup R}
		\]
		where $F_R$ is the free group with generators $f_r$ for $r\in R$ by setting
		\[
		\varphi_{0}\left(x_{s}\right)=\left(s,5b_{s}\right)\quad \varphi_{0}\left(w_{r}\right)=\left(f_r,0 \right)
		\]
		for each generator $x_{s}$ and $w_r$, and extending to $G_{0}$.
		\item\label{it:cs3} In this step we add generators for the $\Z^{S\sqcup R}$ part together with the appropriate commutators as relations to ensure they commute with all other generators.
		\begin{enumerate}
			\item For each $s\in S$ define symbols $t_{s}$ and $t_{s}^{-1}$, and
			for each $r\in R$ define symbols $t_{r}$ and $t_{r}^{-1}$ (again,
			mutually inverse generators $s,s^{-1}$ in $S$ are distinct for this
			purpose). Define $T^{+}=\left\{ t_{s}\right\} _{s\in S}\cup\left\{ t_{r}\right\} _{r\in R}$,
			let $T^{-}=\left\{ t_{s}^{-1}\right\} _{s\in S}\cup\left\{ t_{r}^{-1}\right\} _{r\in R}$
			be the formal inverses of those symbols in $T^{+}$, and define $T=T^{+}\cup T^{-}$.
			
			Choose a linear ordering $\prec$ on $T^{+}$.
			\item For each $s\in S$ and $t\in T^{+}$ define new symbols $u_{s,t}$
			and $u_{s,t}^{-1}$.
			Similarly for each $r\in R$ and $t\in T^{+}$ define new symbols $u_{r,t}$
			and $u_{r,t}^{-1}$.
			Lastly for distinct $t_{1},t_{2}\in T^{+}$
			with $t_{1}\prec t_{2}$ define the new symbols $u_{t_{1},t_{2}}$
			and $u_{t_{1},t_{2}}^{-1}$. Denote the set of all these symbols by
			$U$, and define
			\[
			S_{1}=S_{0}\cup T\cup U.
			\]
			If $t_{1}\succ t_{2}$ are elements of $T^{+}$, define the additional
			notation $u_{t_{1},t_{2}}$ for $u_{t_{2},t_{1}}$ (it is not a distinct
			symbol). Similarly let $u_{t_{1},t_{2}}^{-1}$ denote $u_{t_{2},t_{1}}^{-1}$.
			\item Write the following relations: for each pair of mutually inverse symbols
			$y,y^{-1}$ in $T\cup U$, write the relation $yy^{-1}e=e$. For each
			$s\in S$ and $t\in T^{+}$, write the relations $x_{s}tu_{s,t}^{-1}=e$
			and $u_{s,t}x_{s}^{-1}t^{-1}=e$ (here $t^{-1}\in T^{-}$ is the formal
			inverse of $t\in T^{+}$).
			Similarly for each $r\in R$ and $t\in T^{+}$, write the relations $w_{r}tu_{r,t}^{-1}=e$
			and $u_{r,t}w_{r}^{-1}t^{-1}=e$.
			Denote the set of all these relations by
			$R_{T}$. Define
			\[
			R_{1}=R_{0}\cup R_{T}.
			\]
		\end{enumerate}
		Define $G_{1}=\left\langle S_{1}\mid R_{1}\right\rangle $, and define
		\[
		\varphi_{1}:G_{1}\rightarrow (G\ast F_R)\times\mathbb{Z}^{S\sqcup R}
		\]
		by setting $\varphi_{1}\left(t_{s}\right)=b_{s}$ for each $s\in S$,
		$\varphi_{1}\left(t_{r}\right)=b_{r}$ for each $r\in R$, $\varphi_{1}\left(x_{s}\right)=\varphi_{0}\left(x_{s}\right)$,
		and extending to all other generators and elements as the relations
		dictate (for example, if $t\in T^{+}$ and $s\in S$ then $\varphi_{1}\left(u_{s,t}\right)=\varphi_{1}\left(x_{s}\right)\cdot\varphi_{1}\left(t\right)$).
		
		Note that $\varphi_{1}$ is surjective: $\left(e,b_{r}\right)$ and
		$\left(e,b_{s}\right)$ are in the image for each $r\in R$ and each
		$s\in S$. Similarly $\left(s,5b_{s}\right)$ is in the image for
		each $s\in S$.
		\item\label{it:cs4} Order $R$ arbitrarily, and denote the relations by $r_{1},\ldots,r_{n}$.
		For each relation $r=r_{j}$, in order, for $j=1,\ldots,n$:
		\begin{enumerate}
			\item Write $r$ as $abc=e$ for $a,b,c\in S$ (by step (1) these are distinct).
			\item Define generators and relations to ``break up'' the relation 
			\[
			w_r^{-1} (w_rx_{a}t_{r}^{5}x_{b}t_{r}^{5}x_{c}\left(t_{a}^{-1}t_{r}^{-1}t_{b}^{-1}t_{r}^{-1}t_{c}^{-1}\right)^{5})=e
			\]
			from left to right. Explicitly, write out the word on the left hand
			side of the relation without the $w_r^{-1}$:
			\begin{equation}\label{eq:scrambling_word}
			\tag{$\spadesuit$}w_rx_{a}\underbrace{t_{r}\ldots t_{r}}_{\times5}x_{b}\underbrace{t_{r}\ldots t_{r}}_{\times5}x_{c}\underbrace{t_{a}^{-1}t_{r}^{-1}t_{b}^{-1}t_{r}^{-1}t_{c}^{-1}\ldots t_{a}^{-1}t_{r}^{-1}t_{b}^{-1}t_{r}^{-1}t_{c}^{-1}}_{\times5}.
			\end{equation}
			Define symbols $y_{r,1},\ldots,y_{r,36}$ (together with formal
			inverses $y_{r,1}^{-1}, \dots, y_{r,36}^{-1}$), one for each prefix
			of this word, omitting the empty prefix, the first prefix $w_r$ and the final three prefixes (the entire word, and
			the entire word with the last or the two last letters omitted).
			Denote the set of all
			these symbols by $Y_{r}$. Write the following relations:
			\[
			w_rx_a y_{r,1}^{-1}=e,\quad y_{r,1}t_{r}y_{r,2}^{-1}=e,\quad\ldots\quad y_{r,35}t_{b}^{-1}y_{r,36}=e.
			\]
			(Multiplying by the symbols $y_{r,i}$ from the right and substituting
			the previous relation into each relation in turn, these read $y_{r,1}=w_rx_a$,
			$y_{r,2}=w_r x_{a} t_{r}$, and so on up to $y_{r,36}=w_rx_{a}t_{r}\ldots t_{b}^{-1}$,
			which equals the entire word without the final two letters $t_{r}^{-1}t_{c}^{-1}$).
			Finally, write the relation
			\[
			w_r^{-1} y_{r,36}u_{t_{c},t_{r}}^{-1}=e.
			\]
			Also, for each $y_{r,i}$ write the relation $y_{r,i}y_{r,i}^{-1}e=e$. 
			
			Denote the set of all these relations by $R_{r}$. Then define $S_{j+1}=S_{j}\cup Y_{r}$
			and $R_{j+1}=R_{j}\cup R_{r}$, $G_{j+1}=\left\langle S_{j+1}\mid R_{j+1}\right\rangle $,
			and extend $\varphi_{j}:G_{j}\rightarrow (G\ast F_R)\times\mathbb{Z}^{S\sqcup R}$
			to 
			\[
			\varphi_{j+1}:G_{j+1}\rightarrow (G\ast F_R)\times\mathbb{Z}^{S\sqcup R}
			\]
			in the manner dictated by the relations (this is possible because
			every generator $y\in S_{j+1}\setminus S_{j}$ satisfies a relation
			which defines it in terms of previous generators.) 
			
			Observe that $\varphi_{j+1}$ is a homomorphism: it maps every relator
			of $R_{j+1}$ to the identity. For ``trivial'' relators of the form
			$yy^{-1}e=e$ this is obvious, and similarly for the $36$ relators
			\[
			w_rx_a y_{r,1}^{-1},\quad y_{r,1}t_{r}y_{r,2}^{-1},\quad\ldots\quad, y_{r,35}t_{b}^{-1}y_{r,36},
			\]
			since they define $y_{r,1},\ldots,y_{r,36}$ in terms of the previous
			generators. For the relator $w_r^{-1} y_{r,36}u_{t_{c},t_{r}}^{-1}$,
			note that $u_{t_{c},t_{r}}=t_{c}t_{r}$, and when we substitute previous
			relations into it we obtain 
			\[
			w_r^{-1}w_rx_{a}t_{r}^{5}x_{b}t_{r}^{5}x_{c}\left(t_{a}^{-1}t_{r}^{-1}t_{b}^{-1}t_{r}^{-1}t_{c}^{-1}\right)^{5}=e.
			\]
			When the left hand side is evaluated under $\varphi_{j}$ we obtain
			precisely $abc$, but this product is the identity in $G$, as desired.
		\end{enumerate}
		\item (Postprocessing.) Let $G_{n+1}=\left\langle S_{n+1}\mid R_{n+1}\right\rangle $
		be the presentation of the last step and $\varphi_{n+1}:G_{n+1}\rightarrow (G\ast F_R)\times\mathbb{Z}^{S\sqcup R}$
		the corresponding homomorphism. Symmetrize the set of relations. For
		any relation $abc=e$ in $R_{n+1}$ in which $a=e$ or $a\in T$,
		add the relations $bac=e$ and $bca=e$. Then symmetrize the set of
		relations again. This does not change the group: each generator in
		$T$ commutes with all other generators.
	\end{enumerate}
\end{construction}
\begin{rmrk}
	It is obvious from the construction that it is computable. The presentation
	$\left\langle S^{\prime}\mid R^{\prime}\right\rangle $ can be computed
	from $\left\langle S\mid R\right\rangle $, and the homomorphism $\varphi$
	can be computed in the sense that we can explicitly write the image
	of each generator in $S^{\prime}$ (as a tuple consisting of a word
	in the generators $S$ and an explicitly-given element of $\mathbb{Z}^{S\sqcup R}$).
\end{rmrk}

\begin{theorem}
	\label[theorem]{thm:scrambling_construction}Let $G=\left\langle S\mid R\right\rangle $
	be a group given by symmetric triangular presentation. Let $G^{\prime}=\left\langle S^{\prime}\mid R^{\prime}\right\rangle $
	and $\varphi:G^{\prime}\rightarrow G\times\mathbb{Z}^{S\sqcup R}$
	be the output of \cref{con:scrambling} applied to $\left\langle S\mid R\right\rangle $.
	Then $\left\langle S^{\prime}\mid R^{\prime}\right\rangle $ is a
	group scrambling of $\left\langle S\mid R\right\rangle $ in the sense
	of \cref{def:scrambling}.
\end{theorem}

\begin{proof}
	[Proof of \cref{thm:scrambling_construction}]
	Properties \ref{it:sc4} and \ref{it:sc5} require some case enumeration and are therefore split up into the Lemmas~\ref{lem:scrambling1} and~\ref{lem:scrambling2}.
	\begin{description}
		\item[\ref{it:sc1}] The generating set $S'$ is symmetric by construction, and similarly
		all relators in~$R'$ have length three. The relators are cyclically symmetric as we symmetrized the relations in the last step of the construction. 
		\item[\ref{it:sc2}] Denote $N=\left|S\sqcup R\right|$.
		We prove that $\mu=\varphi:G^{\prime}\rightarrow (G\ast F_R)\times\mathbb{Z}^{S\sqcup R}\simeq (G\ast F_R)\times\mathbb{Z}^{N}$
		is an isomorphism:
		\begin{enumerate}[(1),leftmargin=*]
			\item It is a homomorphism, as explained in the construction.
			\item It is surjective because $\varphi_{1}:G_{1}=\left\langle S_{1}\mid R_{1}\right\rangle \rightarrow (G\ast F_R)\times\mathbb{Z}^{S\sqcup R}$
			is surjective, where $S_{1}\subset S^{\prime}$ and $\varphi\left(s\right)=\varphi_{1}\left(s\right)$
			for each $s\in S_{1}$. 
			\item It is injective: just like in the proof of \cref{prop:augmentation},
			all generators except for $\left\{ x_{s}\right\} _{s\in S}\cup\left\{ t_{s}\right\} _{s\in S}\cup\left\{ t_{r}\right\} _{r\in R}$
			can be eliminated using Tietze transformations. The relations then
			simplify to:
			\begin{enumerate}[(a),leftmargin=*]
				\item The commutators $\left[t_{s},x\right]=e$ for each generator $x\neq t_{s}$,
				\item The commutators $\left[t_{r},x\right]=e$ for each generator $x\neq t_{r}$,
				\item For each mutually inverse pair $s,s^{-1}\in S'$, the relation $x_{s^{-1}}x_{s}=e$.
				\item For each relation $abc=e$ in $\left\langle S\mid R\right\rangle $,
				the relation
				\[
				w_r^{-1}w_rx_{a}t_{r}^{5}x_{b}t_{r}^{5}x_{c}\left(t_{a}^{-1}t_{r}^{-1}t_{b}^{-1}t_{r}^{-1}t_{c}^{-1}\right)^{5}=e.
				\]
			\end{enumerate}
			Since $\left\{ t_{s}\right\} _{s\in S}\cup\left\{ t_{r}\right\} _{r\in R}$
			commute with all generators, the relation of the form
			\[
			w_r^{-1}w_rx_{a}t_{r}^{5}x_{b}t_{r}^{5}x_{c}\left(t_{a}^{-1}t_{r}^{-1}t_{b}^{-1}t_{r}^{-1}t_{c}^{-1}\right)^{5}=e
			\]
			can be replaced by $x_{a}t_{a}^{-5}x_{b}t_{b}^{-5}x_{c}t_{c}^{-5}=e$.
			Using further Tietze transformations, introduce for each $s\in S$
			a new generator $\widetilde{x}_{s}$ and the relation $\widetilde{x}_{s}=x_{s}t_{s}^{-5}$.
			Since each generator $x_{s}$ can be expressed as $\widetilde{x}_{s}t_{s}^{5}$,
			the generators $\left\{ x_{s}\right\} _{s\in S}$ can be eliminated
			(again by Tietze transformations). This yields a presentation with
			generators $\left\{ \widetilde{x}_{s}\right\} _{s\in S}\cup\left\{ t_{s}\right\} _{s\in S}\cup\left\{ w_{r}\right\} _{r\in R}\cup\left\{ t_{r}\right\} _{r\in R}$,
			with relations similar to the above: the relations of type (a) and
			(b) are the same, relations of type (c) are replaced by $\widetilde{x}_{s}\widetilde{x}_{s^{-1}}=e$,
			and each relation of type (d) is replaced by $\widetilde{x}_{a}\widetilde{x}_{b}\widetilde{x}_{c}=e$
			for each relation $abc=e$ in $R$. It is clear that the resulting
			group is isomorphic to $(G\ast F_R)\times\mathbb{Z}^{S\sqcup R}$ with $\mu$
			mapping each $\widetilde{x}_{s}$ to the corresponding $s\in S$, each $w_r$ to the free generator $f_r$, and each element
			of $\left\{ t_{s}\right\} _{s\in S}\cup\left\{ t_{r}\right\} _{r\in R}$
			to the corresponding basis element of $\mathbb{Z}^{S\sqcup R}$.
		\end{enumerate}
		\item[\ref{it:sc6}] Define $i:S\rightarrow S^{\prime}$ by $i\left(s\right)=x_{s}$. Then
		$\pi_{G}\left(i\left(s\right)\right)=s$.
		\item[\ref{it:sc7}] Denote $B=\left\{ b_{s}\right\} _{s\in S}\cup\left\{ b_{r}\right\} _{r\in R}$.
		This is a basis of $\mathbb{Z}^{S\sqcup R}\simeq\mathbb{Z}^{N}$.
		Define $j:B\rightarrow S^{\prime}$ by $j\left(b_{z}\right)=t_{z}$
		(for all $z\in S\sqcup R$). Then $\mu\left(j\left(b_{z}\right)\right)=\left(e,b_{z}\right)$.
		\item[\ref{it:sc8}]Let $s\in S$. Then $\pi_{\mathbb{Z}}\left(i\left(s\right)\right)=\pi_{\mathbb{Z}}\left(x_{s}\right)=5b_{s}$,
		so (expressed in the basis $B=\left\{ b_{s}\right\} _{s\in S}\cup\left\{ b_{r}\right\} _{r\in R}$)
		the $b_{s}$-coefficient of $\pi_{\mathbb{Z}}\left(i\left(s\right)\right)$
		is at least $5$. We verify that the $b_{s}$-coefficient of $\pi_{\mathbb{Z}}\left(s^{\prime}\right)$
		(expressed in $B$) is at most $6$ for each $s^{\prime}\in S^{\prime}$:
		\begin{enumerate}[(1),leftmargin=*]
			\item If $x$ is one of the generators added in step~\ref{it:cs2} then $\pi_{\mathbb{Z}}\left(x\right)=\pm5b$
			for some $b\in B$, and its $b_{s}$-coefficient is clearly at most
			$6$ in absolute value.
			\item If $x$ is one of the generators $w_r$ and $w_r^{-1}$ it is zero under the projection $\pi_\Z$.
			\item Similarly, any of the generators $t_{s}$ or $t_{r}$ added in step~\ref{it:cs3} have $b_{s}$-coefficient at most $1$ in absolute value.
			\item The generators $u_{s,t}$ added in step~\ref{it:cs3} have $b_{s}$-coefficient
			at most $6$ in absolute value; the generators $u_{t_{1},t_{2}}$
			have $b_{s}$-coefficient at most $1$ in absolute value.
			\item If $x$ is one of the generators added in step~\ref{it:cs4}, there is a relation
			$abc=e$ in $R$ such that $x$ is a proper prefix of 
			\[
			w_rx_{a}\underbrace{t_{r}\ldots t_{r}}_{\times5}x_{b}\underbrace{t_{r}\ldots t_{r}}_{\times5}x_{c}\underbrace{t_{a}^{-1}t_{r}^{-1}t_{b}^{-1}t_{r}^{-1}t_{c}^{-1}\ldots t_{a}^{-1}t_{r}^{-1}t_{b}^{-1}t_{r}^{-1}t_{c}^{-1}}_{\times5}.
			\]
			If $s\notin\left\{ a,b,c\right\} $ then $\pi_{\mathbb{Z}}\left(x\right)$
			has $b_{s}$-coefficient $0$. If $s\in\left\{ a,b,c\right\} $, then
			since $a,b,c$ are distinct it is easy to see that $\pi_{\mathbb{Z}}\left(x\right)$
			has $b_{s}$-coefficient nonnegative and at most~$5$.\qedhere
		\end{enumerate}
	\end{description}
\end{proof}
\begin{lemm}\label{lem:scrambling1}
	Properties  \ref{it:sc4} and \ref{it:sc5} hold for the generators added in the steps \ref{it:cs1}-\ref{it:cs3} of Construction~\ref{con:scrambling}.
\end{lemm}
\begin{proof}
	\Cref{tab:scrambling} contains the generators
	defined in steps \ref{it:cs1}-\ref{it:cs3} of the construction (one representative from
	each mutually inverse pair) together with their values under $\pi_{F,\mathbb{Z}}^\mathrm{ab}$.
	(We slightly abuse notation: if $t\in T^{+}=\left\{ t_{s}\right\} _{s\in S}\cup\left\{ t_{r}\right\} _{r\in R}$
	then $b_{t}$ refers to $b_{s}$ or $b_{r}$ according to the value
	of $t$. Further, we denote by $f_r$ both the generators of the free group $F_R$ and its abelianization $\Z^{|R|}$.)
	\begin{table}[hbt]
	\addtolength{\tabcolsep}{0.9pt}
	\def\arraystretch{1.3}{
	\[
	\begin{array}{ll}
	\text{Generator }x & \pi_{F,\mathbb{Z}}^\mathrm{ab}\left(x\right)\\
	\hline
	\rowcolor{LightGray}
	 e & 0\\
	x_{s}\text{ for each \ensuremath{s\in S}} & 5b_{s}\\\rowcolor{LightGray}
	w_{r}\text{ for each \ensuremath{r\in R}} & f_r\\
	t_{s}\text{ for each \ensuremath{s\in S}} & b_{s}\\\rowcolor{LightGray}
	t_{r}\text{ for each \ensuremath{r\in R}} & b_{r}\\
	u_{s,t}\text{ for \ensuremath{s\in S,t\in T^{+}}} & 5b_{s}+b_{t}\\\rowcolor{LightGray}
	u_{r,t}\text{ for \ensuremath{r\in R,t\in T^{+}}} & f_r+b_{t}\\
	u_{t_{1},t_{2}}\text{ for distinct \ensuremath{t_{1},t_{2}\in T^{+}}} & b_{t_{1}}+b_{t_{2}}
	\end{array}
	\]
	}
	\caption{The generators added in the steps \ref{it:cs1}-\ref{it:cs3} of the scrambling construction.}
	\label{tab:scrambling}
	\end{table}
	
	For Property \ref{it:sc4} it suffices to verify that the values of any two
	of the generators in the table under $\pi_{F,\mathbb{Z}}^{\mathrm{ab}}$ are distinct,
	and that the value of any generator under $\pi_{F,\mathbb{Z}}^{\mathrm{ab}}$ is different
	than the value of the inverse of another generator. This is clear
	by inspection of the rows of the table. 
	
	We now verify Property \ref{it:sc5} by checking that if generators $x,x',x''$ from~\Cref{tab:scrambling} satisfy $\pi_{F,\Z}^\mathrm{ab}(x''x'x)=0$ then the word $x''x'x$ is trivial in $G'$ or sufficiently generic relative to the map from $S$ to $G\ast F_R$.
	So assume these generators satisfy $\pi_{F,\Z}^\mathrm{ab}(x''x'x)=0$.
	This means we have $\pi_{F,\Z}^\mathrm{ab}\left(x\right)+\pi_{F,\Z}^\mathrm{ab}\left(x^{\prime}\right)+\pi_{F,\Z}^\mathrm{ab}\left(x^{\prime\prime}\right)=0$.
	But the $\pi_{F,\Z}^\mathrm{ab}$-values in Table~\ref{tab:scrambling} are
	positive linear combinations of the elements of the basis of $\Z^{|R|}\times \Z^N$, so at least
	one of the generators $x,x',x''$ is the inverse of a generator listed in this table.
	By negating the equation
	if necessary, we may assume without loss of generality that exactly one is the inverse of a generator listed in~\Cref{tab:scrambling}, and after replacing $x''$ by $x''^{-1}$ and renaming the generators
	if necessary we obtain the equation
	\begin{align}
	\tag{$\pi_{F,\Z}^\mathrm{ab}$}&\pi_{F,\Z}^\mathrm{ab}(x)+\pi_{F,\Z}^\mathrm{ab}(x')=\pi_{F,\Z}^\mathrm{ab}(x'').\label{eq:piZ_equation}
	\end{align}
	Subsequently we need to show that for generators $x,x',x''$ that satisfy this equation the elements $xx'x''^{-1}$, $x'x''^{-1}x$, $x''^{-1}xx'$, $x'xx''^{-1}$, $xx''^{-1}x'$, $x''^{-1}x'x$ are all trivial or sufficiently generic relative to the map from $S$ to $G\ast F_R$.
	Note that the cyclic shifts of these words arise by conjugating with  $x^{-1}$ or $x'^{-1}$.
	As conjugation preserves the set of sufficiently generic words by~\Cref{def:sufficiently_generic} it suffices to consider the elements $xx'x''^{-1}$ and $x'xx''^{-1}$ in the following.
	
	So we now look for all generators $x,x^{\prime},x^{\prime\prime}$ in~\Cref{tab:scrambling} satisfying Equation~\eqref{eq:piZ_equation} and check if the elements $xx'x''^{-1}$ and $x'xx''^{-1}$ are trivial in $G'$ or sufficiently generic relative to the map from $S$ to $G\ast F_R$.
	We split the argument into cases based on the value of $x^{\prime\prime}$.
	We can exclude the cases $x=e$ or $x'=e$ as this would imply $\pi_{F,\Z}^\mathrm{ab}(x'x''^{-1})=e_F$ or $\pi_{F,\Z}^\mathrm{ab}(xx''^{-1})=e_F$ respectively, and there exist no nontrivial solution to these equations by property \ref{it:sc4} which we already verified above.
	If $g_1,g_2,g_3$ are generators in Table~\ref{tab:scrambling} we don't distinguish between the solutions $x=g_1$, $x'=g_2$, $x''=g_3$ and $x=g_2$, $x'=g_1$, $x''=g_3$ as we analyze the words $xx'x''^{-1}$ and $x'xx''^{-1}$ for each such solution in both cases.
	\begin{enumerate}[label=\textbf{Case \arabic*:}, labelwidth=-8mm,  labelindent=3em,leftmargin =!]
		\item Suppose $x^{\prime\prime}=e$. Then $x=w_r$, $x'=w_r^{-1}$ for some $r\in R$. Both words $xx'x''^{-1}$ and $x'xx''^{-1}$ are trivial in $G'$ in this case.
		\item Suppose $x^{\prime\prime}=x_{s}$ for some $s\in S$. There is no solution in this case.
		\item Suppose $x^{\prime\prime}=w_{r}$ for some $r\in R$. There is no solution in this case.
		\item Suppose $x^{\prime\prime}=t_{s}$ for some $s\in S$. There is no solution in this case.
		\item Suppose $x^{\prime\prime}=t_{r}$ for some $r\in R$, Then $x=w_r^{-1}$, $x'=u_{r,t_r}$ and both associated words are trivial in $G'$.
		\item Suppose $x^{\prime\prime}=u_{s,t}$ for $s\in S$ and $t\in T^{+}$. The
		unique solution is $x=x_{s}$
		and $x^{\prime}=t$. In this case $x^{\prime\prime-1}x^{\prime}x=u_{s,t}^{-1}x_{s}t=e$
		is a relator in $R^{\prime}$, as is $u_{s,t}^{-1}tx_{s}=e$.
		\item Suppose $x^{\prime\prime}=u_{r,t}$ for $r\in R$ and $t\in T^{+}$. The
		unique solution is $x=w_{r}$
		and $x^{\prime}=t$. The resulting words are again trivial in $G'$.
		\item Suppose $x^{\prime\prime}=u_{t_{1},t_{2}}$ for $t_{1},t_{2}\in T^{+}$.
		The unique solution is $x=t_{1}$
		and $x^{\prime}=t_{2}$. In this case $x^{\prime\prime-1}x^{\prime}x=u_{t_{1},t_{2}}^{-1}t_{1}t_{2}=e$
		is a relation in $R^{\prime}$.\qedhere
	\end{enumerate}
\end{proof}

\begin{lemm}\label{lem:scrambling2}
	Properties  \ref{it:sc4} and \ref{it:sc5} hold for all generators added in the scrambling construction.
\end{lemm}
\begin{proof}
	We inductively verify that these properties still hold when further
	generators are added for each relator in $R$ in step \ref{it:cs4} of Construction~\ref{con:scrambling}.

		Denote the relators in $R$ by $r_{1},\ldots,r_{n}$ according to the arbitrary order chosen in Construction~\ref{con:scrambling}.
	By induction on
	$1\le j\le n$ we show that:
	\begin{enumerate}[(i)]
		\item For each $j=1,\ldots,n$, the generators $\left\{ y_{r_{j},i}^{\pm1}\right\} _{i=1}^{36}$
		all satisfy that their value under $\pi_{F,\Z}^\mathrm{ab}$ involves $b_{r_{j}}$ or $f_{r_j}$ and no other element of $\left\{ b_{r},f_r\right\} _{r\in R}$. 
		\item The generators added in steps \ref{it:cs1}-\ref{it:cs3} of the construction, together
		with the generators $\bigcup_{k\le j}\left\{ y_{r_{k},i}^{\pm1}\right\} _{i=1}^{36}$
		satisfy conditions \ref{it:sc4} and \ref{it:sc5}.
	\end{enumerate}
	Let the $j$-th relation $r_{j}$ be $abc=e$.
	Table~\ref{table:prefixes} lists one generator from each mutually
	inverse pair $\left\{ y_{r_{j},i},y_{r_{j},i}^{-1}\right\} $ added in step \ref{it:cs4} for this relation.
	\begin{table}[hbt]
	\addtolength{\tabcolsep}{0.9pt}
	\def\arraystretch{1.4}{
	\[
	\begin{array}{ll}
	\text{Generator }x & \pi_{F,\Z}^\mathrm{ab}\left(x\right)\\
	\hline \rowcolor{LightGray}
	y_{r_{j},i}\text{ for \ensuremath{1\le i\le6}} & f_{r_j}+5b_{a}+\left(i-1\right)b_{r_{j}}\\
	y_{r_{j},i}\text{ for \ensuremath{7\le i\le12}} & f_{r_j}+5b_{a}+5b_{b}+(i-2)b_{r_{j}}\\\rowcolor{LightGray}
	y_{r_{j},13} & f_{r_j}+5b_{a}+5b_{b}+5b_{c}+10b_{r_{j}}\\
	y_{r_{j},13+5k}\text{ for \ensuremath{1\le k\le}4} & f_{r_j}+(5-k)b_{a}+(5-k)b_{b}+(5-k)b_{c}+(10-2k)b_{r_{j}}\\\rowcolor{LightGray}
	y_{r_{j},14+5k}\text{ for \ensuremath{0\le k\le4}} & f_{r_j}+(4-k)b_{a}+(5-k)b_{b}+(5-k)b_{c}+(10-2k)b_{r_{j}}\\
	y_{r_{j},15+5k}\text{ for \ensuremath{0\le k\le4}} & f_{r_j}+(4-k)b_{a}+(5-k)b_{b}+(5-k)b_{c}+(9-2k)b_{r_{j}}\\\rowcolor{LightGray}
	y_{r_{j},16+5k}\text{ for \ensuremath{0\le k\le4}} & f_{r_j}+(4-k)b_{a}+(4-k)b_{b}+(5-k)b_{c}+(9-2k)b_{r_{j}}\\
	y_{r_{j},17+5k}\text{ for \ensuremath{0\le k\le3}} & f_{r_j}+(4-k)b_{a}+(4-k)b_{b}+(5-k)b_{c}+(8-2k)b_{r_{j}}
	\end{array}
	\]
}
	\caption{The generators corresponding to the relator $r_j$ added in the step \ref{it:cs4} of the scrambling construction.}
	\label{table:prefixes}
	\end{table}

	Part (i) is clear: the new generators all satisfy that their value
	under $\pi_{F,\Z}^\mathrm{ab}$ involves $b_{r_{j}}$ or $f_{r_j}$
	and no other element of $\left\{ b_{r},f_r\right\} _{r\in R}$. We need
	to verify (ii), i.e. conditions \ref{it:sc4} and~\ref{it:sc5}.
	
	For condition \ref{it:sc4} it suffices to check that any two of the new generators
	have distinct values under $\pi_{F,\Z}^\mathrm{ab}$, and that their values
	under $\pi_{F,\mathbb{Z}}^{\mathrm{ab}}$ are distinct from those of the generators
	added in steps \ref{it:cs1}-\ref{it:cs3} of the construction.
	Note that the values of $\pi_{F,\mathbb{Z}}$ in~\Cref{table:prefixes} are all distinct, so it suffices to compare these generators with the ones added in steps \ref{it:cs1}-\ref{it:cs3}.
	These all have a $b_{r_j}$-coefficient at most one.
	We list all generators which after applying  $\pi_{F,\Z}^\mathrm{ab}$ have $b_{r_{j}}$-coefficient one or involve the generator $f_{r_j}$ in~\Cref{tab:property4}.
	\begin{table}[hbt]
	\addtolength{\tabcolsep}{0.9pt}
	\def\arraystretch{1.4}{
		\[
	\begin{array}{ll}
	\text{Generator \ensuremath{x}} & \pi_{F,\Z}^\mathrm{ab}\left(x\right)\\
	\hline \rowcolor{LightGray} y_{r_{j},1} & f_{r_j}+5b_{a}\\
	y_{r_{j},36} & f_{r_j}+b_{c}+b_{r_{j}}\\\rowcolor{LightGray}
	t_{r_{j}} & b_{r_{j}}\\
	u_{s,t_{r_{j}}}\text{ for \ensuremath{s\in S}} & 5b_{s}+b_{r_{j}}\\\rowcolor{LightGray}
	u_{r_j,t}\text{ for \ensuremath{t\in T^{+}}}& f_{r_j}+b_t\\
	u_{t_{r_{j}},t}\text{ for \ensuremath{t\in T^{+}\setminus\left\{ t_{r_{j}}\right\} }} & b_{r_{j}}+b_{t}\\\rowcolor{LightGray}
	w_{r_j}& f_{r_j}
	\end{array}
	\]
	}
	\caption{The generators of $S'$ which have $b_{r_j}$-coefficient one or involve the generator $f_{r_j}$ after applying $\pi_{F,\Z}^\mathrm{ab}$.}
	\label{tab:property4}
	\end{table}
	Apparently, any two of these have different values under
	$\pi_{F,\Z}^\mathrm{ab}$.
	and it is not possible that $\pi_{F,\mathbb{Z}}^{\mathrm{ab}}\left(x\right)=\pi_{F,\mathbb{Z}}^{\mathrm{ab}}\left(x^{\prime-1}\right)=-\pi_{F,\mathbb{Z}}^{\mathrm{ab}}\left(x^{\prime}\right)$.
	
	For the rest of this proof, ``sufficiently generic'' is short for ``sufficiently generic in $G \ast (F_R \ast \langle z_1,\ldots,z_4\rangle)$ relative to $S$''

		For condition \ref{it:sc5} we proceed similarly as in~\Cref{lem:scrambling1}, namely we consider generators $x,x',x''$ in the Tables~\ref{tab:scrambling} and~\ref{table:prefixes} that satisfy Equation~\eqref{eq:piZ_equation}.

	We verify that for these generators the words $xx'x''^{-1}$ and $x'xx''^{-1}$ are trivial in $G'$ or sufficiently generic.
	
	By the above discussion we may exclude the cases in which $x=e$ or $x'=e$.
		As in the proof of~\Cref{lem:scrambling1}, we don't distinguish between the solutions $x=g_1$, $x'=g_2$, $x''=g_3$ and $x=g_2$, $x'=g_1$, $x''=g_3$ for generators $g_1,g_2,g_3$ as we analyze the words $xx'x''^{-1}$ and $x'xx''^{-1}$ for each such solution in both cases.
	We may also restrict to cases in which at least one of the generators involves a generator $y_{r_{j},k}$ for $1\le k\le36$ as otherwise the condition was already checked in~\Cref{lem:scrambling1}.
	We begin by assuming that $x''=y_{r_{j},k}$ for some $1\le k\le36$. This yields the following cases.
	\begin{enumerate}[label=\textbf{Case \arabic*:}, labelwidth=-8mm,  labelindent=3em,leftmargin =!]
		\item Suppose $x''=y_{r_{j},1}$. In this case the only solution is $x=x_a$, $x'=w_{r_j}$ which yields the words $x_aw_rx_{a}^{-1}w_r^{-1}$ and $w_rx_ax_{a}^{-1}w_r^{-1}$. The former is sufficiently generic by \cref{lem:sufficiently_generic}~\ref{it:suff2} and the latter is clearly trivial in $G'$.
		\item Suppose $x''=y_{r_{j},k}$ with $2\le k\le 6$. In this case there are the following solutions:
		\begin{enumerate}[(1),leftmargin=*]
			\item If $k=2$ there is $x=u_{a,t_{r_j}}$, $x'=w_r$ which yields $x_at_{r_j}w_rt_{r_j}^{-1}x_{a}^{-1}w_r^{-1}$ and $w_rx_at_{r_j}t_{r_j}^{-1}x_{a}^{-1}w_r^{-1}$.
			Since $t_{r_j}$ commutes with all other involved letters this case reduces to the previous one.
			\item If $k=2$ there is $x=x_a$, $x'=u_{r_j,t_{r_j}}$ which yields $x_aw_{r_j}t_{r_j}t_{r_j}^{-1}x_a^{-1}w_{r_j}^{-1}$ and $w_{r_j}t_{r_j}x_at_{r_j}^{-1}x_a^{-1}w_{r_j}^{-1}$.
			These words again reduce to the $k=1$ case.
			\item If $k=3$ there is $x=u_{a,t_{r_j}}$, $x'=u_{r_j,t_{r_j}}$ which yields $x_at_{r_j}w_{r_j}t_{r_j}t_{r_j}^{-2}x_a^{-1}w_{r_j}^{-1}$ and $w_{r_j}t_{r_j}x_at_{r_j}t_{r_j}^{-2}x_a^{-1}w_{r_j}^{-1}$ which also reduces to the $k=1$ case.
			\item For all $k=2,\dots,6$ there is $x=y_{r_j,k-1}$, $x'=t_{r_j}$ which yields trivial words in $G'$ in both orders.
		\end{enumerate}
	\item Suppose $x''=y_{r_{j},k}$ with $7\le k\le 12$. In this case we have the solution $x=y_{r_j,k-1}$, $x'=t_{r_j}$ which yields either the trivial element in $G'$ or (after canceling $t_{r_j}$) the word \[x_bw_{r_j}x_ax_b^{-1}x_a^{-1}w_{r_j}^{-1} = [x_b, w_{r_j} x_a].\]
	This word is sufficiently generic by \cref{lem:sufficiently_generic}~\ref{it:suff3}.
	For $k=7$ there is the additional solution $x=y_{r_j,5}$, $x'=u_{b,t_{r_j}}$ which yields the same words as the previous solution.
	\item\label{it:case4} Suppose $x''=y_{r_{j},13}$. There is $x=y_{r_j,12}$, $x'=x_c$ and $x=y_{r_j,11}$, $x'=u_{c,t_{r_j}}$ which yields a trivial word in $G'$ and the word
	\[
		x_cw_{r_j}x_ax_bx_c^{-1}x_b^{-1}x_a^{-1}w_{r_j}^{-1}=\left[x_c,w_{r_j}x_ax_b\right].
	\]
	This word is sufficiently generic by \cref{lem:sufficiently_generic}~\ref{it:suff3}.
	Furthermore there are solutions $x=y_{r_j,14}$, $x'=t_a$ and $x=y_{r_j,15}$, $x'=u_{a,t_{r_j}}$. Both solutions yield in both orders trivial words in $G'$ as all $t$'s commute with all other letters.
	\item Suppose $x''=y_{r_{j},k}$ with $14\le k \le 34$. The prefixes $y_{r_j,k}$ all involve $f_{r_j}$ and at least three different bases elements of $B$ with $b_{r_j}$-coefficient at least two.
	Thus the only solutions involve the two previous or the two following prefixes and the appropriate letters that got added between these prefixes.
	Denoting by $l_{k}$ the $k$-th letter of the word~\eqref{eq:scrambling_word} we have the following solutions:
	\begin{enumerate}[(1),leftmargin=*]
		\item $x=y_{r_{j},k-2}$, $x'=u_{l_k^{-1},l_{k+1}^{-1}}$ where $u_{l_k^{-1},l_{k+1}^{-1}}$ is the commutator symbol of the letters $l_{k}^{-1}$ and $l_{k+1}^{-1}$,
		\item $x=y_{r_{j},k-1}$, $x'=l_{k+1}^{-1}$,
		\item $x=y_{r_{j},k+1}$, $x'=l_{k-1}^{-1}$,
		\item $x=y_{r_{j},k+2}$, $x'=u_{l_{k-1}^{-1},l_{k-2}^{-1}}$ where $u_{l_{k-1}^{-1},l_{k-2}^{-1}}$ is the commutator symbol of the letters $l_{k-2}^{-1}$ and $l_{k-1}^{-1}$.
	\end{enumerate}
	The case $k=14$ and $x=y_{r_{j},12}$, $x'=u_{c,t_{r_j}}$ reduces to the sufficiently generic word given in Case 4.
	 All other resulting words are trivial in $G'$ as they involve the ``noncommutative block'' $w_{r_j}x_ax_bx_c$ in the right order on both sides and the $t$'s cancel.
	\item Suppose $x''=y_{r_{j},35}$. There are also solutions of the same shape as in the previous case, namely $x=y_{r_{j},33}$, $x'=u_{t_b,t_{r_j}}$, $x=y_{r_{j},34}$, $x'=t_{b_b}$, and $x=y_{r_{j},36}$, $x'=b_{r_j}$ which all yield trivial words in $G'$.
	Furthermore, there is also the solution $x=u_{r_j,t_{r_j}}$, $x'=u_{t_c,t_{r_j}}$ which after reordering some $t$'s yields the word 
	\begin{equation}\label{eq:trivial_word}
		w_{r_j}t_{r_j}^2t_cy_{r_{j},35}^{-1}= w_{r_j}y_{r_{j},36}^{-1}u_{t_c,t_{r_j}}.
	\end{equation}
	This word is trivial in $G'$ as there is the relation $w_r^{-1} y_{r_j,36}u_{t_{c},t_{r}}^{-1}=e$ in this group.
	\item Suppose $x''=y_{r_{j},36}$. There are the solutions $x=y_{r_{j},34}$, $x'=u_{t_b,t_{r_j}}$ and $x=y_{r_{j},35}$, $x'=t_{r_j}$ which also yield trivial words as above.
	Furthermore, there are the solution $x=w_{r_j}$, $x'=u_{t_c,t_{r_j}}$, $x=u_{r_j,t_c}$, $x'=t_{r_j}$, and $x=u_{r_j,t_{r_j}}$, $x'=t_c$ which yield the same trivial word as in Equation~\eqref{eq:trivial_word} after reordering the $t$'s.
	\end{enumerate}
	The only remaining cases left to check are the ones where $x=y_{r_j,k}$ for some $1\le k\le 36$ and $x''$ is a generator of~\Cref{tab:scrambling}.
	As the $b_{r_j}$-coefficient of the generator $x$ can be at most one and the generators $x,x',x''$ satisfy Equation~\eqref{eq:piZ_equation} there are only the following cases to consider.
	\begin{enumerate}[label=\textbf{Case \arabic*:}, labelwidth=-8mm,  labelindent=3em,leftmargin =!,resume]
		\item Suppose $x=y_{r_j,1}$. The only solution is $x'=w_{r_j}^{-1}$ and $x''=x_a$. This yields a trivial word and the word $w_{r_j}x_aw_{r_j}^{-1}x_a^{-1}$. The latter is sufficiently generic by \cref{lem:sufficiently_generic}.
		\item Suppose $x=y_{r_j,2}$. The only solution is $x'=w_{r_j}^{-1}$ and $x''=u_{a,t_{r_j}}$. As the $t$'s commute and cancel we obtain the same words as in the previous case.
		\item Suppose $x=y_{r_j,36}$. The only solution is $x'=w_{r_j}^{-1}$ and $x''=u_{t_c,t_{r_j}}$. Both words are trivial due to the relation $w_r^{-1} y_{r_j,36}u_{t_{c},t_{r}}^{-1}=e$ in the group $G'$.
	\end{enumerate}
		Thus, condition \ref{it:sc5} holds for all generators $x,x',x''\in S'$  which completes the proof.
\end{proof}

\section{Entropic matroid representability is undecidable}\label{sec:entropic_undecidability}

We have now all necessary tools at our disposal to complete the proof that there is no algorithm that checks whether a matroid is entropic.
We prove this result by connecting the uniform word problem for finite groups with the entropic representations of the associated partial Dowling geometries.
The first part of this relation is described in the following theorem.

\begin{theorem}\label{thm:WP_entropic}
	Let $\langle S\mid R\rangle, s\in S$ be an instance of the uniform word problem for finite groups.
	Furthermore let $\langle S''\mid R''\rangle$ be the augmented presentation from~\Cref{con:augmentation}
	 and $\mathcal{M}$ the set of partial Dowling geometries subordinate to this presentation.
		If there exists a finite quotient of $G_{S,R}$ in which $s$ is nontrival then some matroid in $\mathcal{M}$ is entropic.
\end{theorem}
\begin{proof}
Assume there is a group homomorphism $\varphi:G_{S,R}\to G$ for some finite group G with $\varphi(s)\neq e$.
Set $n\coloneqq |G|$ and identify the elements of $G$ with $\{1,\ldots,n\}$.
Let $\rho:G_{S,R} \to \mathrm{GL}_n(\C)$ be the representation where each $\rho(g)$ is the permutation matrix corresponding to the action of $\varphi(g)$ by left-multiplication on $G$.

By assumption we have $\rho(s)\neq \rho(e)$.
	Therefore we can apply~\Cref{prop:augmentation_reps} and obtain a representation $\widetilde{\rho}:G_{S'',R''}\to \mathrm{GL}_{\widetilde{n}}(\C)$ for some $\widetilde{n}\in \N$ such that
	\begin{enumerate}
	\item $\widetilde{\rho}(s)-\widetilde{\rho}(s^{\prime})$ is invertible for any distinct $s,s^{\prime}\in S''$ and
	\item whenever $s,s^{\prime},s^{\prime\prime}\in S''$ (not necessarily
	distinct) satisfy 
	$\widetilde{\rho}(s^{\prime\prime} s^{\prime}s) \neq I_n$
	then the matrix
	$\widetilde{\rho}(s^{\prime\prime}s^{\prime}s) - I_n$ 
\end{enumerate}
Hence by~\Cref{thm:gdg_vector_space_reps} some of the partial Dowling geometries $\mathcal{M}$ subordinate to $\langle S''\mid R''\rangle$ is multilinear over $\F$.
	Thus by~\cite{Mat99} this matroid in $\mathcal{M}$ is also entropic.
\end{proof}

The next theorem describes the converse implication
\begin{theorem}\label{thm:entropic_WP}
	Let $\langle S\mid R\rangle, s\in S$ be an instance of the uniform word problem for finite groups and $M$ be the partial Dowling geometry of the presentation $\langle S''\mid R''\rangle$ obtained from \Cref{con:augmentation}.
	Assume that some matroid of the partial Dowling geometries $\mathcal{M}$ subordinate to $\langle S''\mid R''\rangle$ is entropic.
	Then there exists a group homomorphism $\varphi:G_{S,R}\to G$ to a finite group $G$ with $\varphi(s)\neq e$.
\end{theorem}
\begin{proof}
Suppose the matroid $M\in \mathcal{M}$ is entropic.
This is the partial Dowling geometry of a quotient of $G_{S'',R''}$.
Composing the quotient map with the group homomorphism stemming from \Cref{thm:entropic_Dowling_reps} applied to the entropic matroid $M$ we obtain an $n\in \N$ and a group homomorphism $\rho:G_{S'',R''}\rightarrow S_{n}$ with $\rho\left(x\right)\neq\rho\left(x^{\prime}\right)$
for distinct $x,x^{\prime}\in S''$.
Recall from \Cref{con:augmentation} that there is an isomorphism 
\[
	\nu: \left(G_{S,R}\ast F_R \ast \left\langle z_{1},\ldots,z_{4}\right\rangle \right)\times\mathbb{Z}^{N}\to G_{S'',R''}
\]
such that  $\nu(z_{1})=s_z$ for some generator $s_z\in S^{\prime\prime}$
and $\nu(sz_1s)=t$ with $t\in S''$.

As $\nu$ is an isomorphism the generators $s_z$ and $t$ must be distinct.
Hence we obtain $\rho(t)\neq \rho(s_z)$.
Composing these maps therefore yields $\rho \circ \nu(sz_1s)\neq \rho \circ \nu(z_1)$.
Thus $\rho \circ \nu(s)\neq \rho \circ \nu(e)$.
Restricting $\rho\circ \nu$ to $G_{S,R}\le\left(G_{S,R}\ast F_R\ast \left\langle z_{1},\ldots,z_{4}\right\rangle \right)\times\mathbb{Z}^{N}$ therefore yields the desired map from $G_{S,R}$ to the finite group $S_n$ with $\rho\circ \nu(s)\neq \rho\circ \nu(e)$.
\end{proof}

Combining the last two theorems with Slobodskoi's undecidability of the uniform word problem for finite groups immediately yields a proof of \cref{thm:entropic}, which we restate here:

\begin{theorem}\label{thm:undecidable_entropic}
	The entropic matroid representation problem is algorithmically undecidable.
	In other words, there is no algorithm that takes a matroid as input, always halts, and returns ``true'' if and only if the matroid is entropic.
\end{theorem}
\begin{proof}
	 \Cref{thm:WP_entropic,thm:entropic_WP} imply that solving an instance of the uniform word for finite groups is equivalent to checking whether at least one member in a finite set of matroids is entropic. The conclusion therefore follows from Slobodskoi's theorem that the uniform word problem for finite groups is undecidable (\Cref{thm:undecidable_grp_theory}).
\end{proof}

\section{The conditional independence implication problem}\label{sec:cii}

We fix some finite ground set $E$ for the entire section.

\begin{lemm}\label{lem:determination}
	A family of discrete random variables $\{X_e\}_{e\in E}$ realizes the CI statement $(i\perp i \mid J)$ with $i\in E$ and $J\subseteq E\setminus \{i\}$ if and only if $X_i$ is determined by $\{X_j\}_{j\in J}$.
\end{lemm}
\begin{proof}
	The random variables $\{X_e\}_{e\in E}$ realize a CI statement $(A\perp B\mid C)$ for $A,B,C\subseteq E$ if and only if
	\[
		H(X_A \mid X_C) + H(X_B \mid X_C) - H(X_{A\cup B} \mid X_C) =0, 
	\]
	where $H(X_S\mid X_T)$ is the entropy of $X_S$ conditioned on $X_T$ for subsets $S,T\subseteq E$.
	Applying this to the CI statement $(i\perp i \mid J)$ implies that $H(X_i\mid X_J)=0$ which is the case if and only if $X_i$ is determined by $\{X_j\}_{j\in J}$.
\end{proof}

We relate probability space representations of matroids to the following variant of the conditional independence implication (CII) problem.
\begin{prob}\label{prob:cir}
	The \emph{conditional independence realization (CIR) problem} asks:
	\begin{description}
		\item[Instance] A set $\mathcal C$ of CI statements on a finite ground set $E$.
		\item[Question] Does there exist a nontrivial family of discrete random variables $\{X_e\}_{e\in E}$ realizing all CI statements in $\mathcal{C}$?
		By nontrivial we mean that there is at least one random variable that is not constant (i.e., at least one random variable does not take a single value with probability $1$).
	\end{description}
\end{prob}

\begin{theorem}\label{thm:CI_Entropic}
	Let $M$ be a connected matroid on the ground set $E$.
	There exists a set of CI statements~$\mathcal{C}_M$ on the ground set $E$ such that $M$ has a discrete probability space representation if and only if $\mathcal{C}_M$ can be realized by a nontrivial family of discrete random variables.
\end{theorem}
\begin{proof}
	Given a connected matroid $M$ we construct a set of CI statements $\mathcal{C}_M$:
	\begin{enumerate}
		\item For every independent set $A\subseteq E$ in $M$ we add the CI statements $(i\perp A\setminus \{i\}\mid \emptyset)$ for all $i\in A$ to $\mathcal{C}_M$
		\item For every circuit $C\subseteq E$ in $M$ we add the CI statements $(i\perp i \mid C\setminus \{i\})$ for all $i\in C$ to $\mathcal{C}_M$.
	\end{enumerate}
	Let $\{X_e\}_{e\in E}$ be a set of discrete random variables.
	Suppose $A=\{a_1,\dots,a_k\}\subseteq E$ is an independent subset of $M$.
	The random variables $\{X_e\}_{e\in A}$ are independent if and only if they realize the CI statements  $(a_{i+1}\perp \{a_1,\dots,a_{i}\}\mid \emptyset)$ for all $1\le i \le k-1$.
	By construction, $\mathcal{C}_M$ contains all these CI statements, because a subset of an independent set of $M$ is independent. Therefore if $\{X_e\}_{e\in E}$ satisfy $\mathcal{C}_M$ then $\{X_a\}_{a\in A}$ are independent for every independent set $A$ of $M$.
	
	Conversely, it is clear that if the variables $\{X_e\}_{e\in E}$ give a probability space representation of $M$, they satisfy every CI statement in $\mathcal{C}_M$ constructed in (a).
	
	Let $C\subseteq E$ be a circuit of $M$.
	\Cref{lem:determination} yields that the random variables $\{X_e\}_{e\in C\setminus\{i\}}$ determine $X_i$ for all $i\in C$ if and only if the random variables realize the CI statements corresponding to this circuit.
	Thus, for the family $\{X_e\}_{e\in E}$ it is equivalent to realize all CI statements corresponding to circuits of the matroid and to fulfill all determination properties dictated by the matroid in~\Cref{def:prob_space_rep}.
	
	Finally, the probability space representation being nontrivial implies that the random variable corresponding to an element that is not a loop is not constant with probability $1$.
	Hence, if $\{X_e\}_{e\in E}$ are random variables corresponding to a probability space representation of $M$ then they are a nontrivial realization of $\mathcal{C}_M$.
	
	Conversely, assume that $\{X_e\}_{e\in E}$ is a nontrivial family of random variables realizing $\mathcal{C}_M$, so that there exists $e\in E$ such that $X_e$ is a nontrivial random variable.
	We show that this implies the nontriviality condition of a probability space representation in~\Cref{def:prob_space_rep}:
	Let $f\in E$ be any element that is not a loop in the matroid $M$.
	Since the matroid is connected, there exists a circuit $C$ of $M$ with $\{e,f\}\subseteq C$.
	By the above arguments we know that the family $\{X_e\}_{e\in E}$ satisfies the independence and determination assumptions.
	In particular, the subfamily $\{X_g\}_{g\in C\setminus\{e\}}$ is independent and determines $X_e$.
	Thus, the subfamily $\{X_g\}_{g\in C\setminus\{e,f\}}$ does not determine $X_e$ which implies that $X_f$ must be nontrivial too.
\end{proof}

\begin{coro}\label{cor:CIR}
	The conditional independence realization (CIR) problem is algorithmically undecidable.
\end{coro}
\begin{proof}
	This follows directly form the~\Cref{thm:undecidable_entropic,thm:CI_Entropic} since the partial Dowling geometries are connected matroids.
\end{proof}

Now we are finally ready to prove that the conditional independence implication problem is undecidable,

\begin{theorem}
	The conditional independence implication (CII) problem is algorithmically undecidable.
\end{theorem}
\begin{proof}
	Assume there is an oracle to decide the CII problem.
	We will show that using this oracle one can also decide the CIR problem.
	By~\Cref{cor:CIR} this then shows that the CII problem is algorithmically undecidable.
	
	Let $\mathcal C$ be an CIR problem instance and denote by $\A_E$ the set of all CI statements on the ground set $E$.
	We claim that $\mathcal C$ has a nontrivial realization, that is the associated CIR problem has a positive solution, if and only if at least one of the following finite set of CII problem instances has a negative answer:
	\begin{equation}\label{eq:cii}
		\left\{\bigwedge_{A\in \mathcal C} A \Rightarrow c \mid c\in \A_E\setminus \mathcal C\right\}.
	\end{equation}
	Suppose the family $\{X_e\}_{e\in E}$ is nontrivial and realizes $\mathcal C$.
	Since the family is nontrivial there is some $c_0\in \A_E$ that is not realized by this family: $(e \perp e | \emptyset)$ is not realized whenever $X_e$ is not constant.
	Hence, the CII problem instance $\bigwedge_{A\in \mathcal C} A \Rightarrow c_0$ that appears in~\eqref{eq:cii} has a negative answer.
	
	Conversely, assume that $\bigwedge_{A\in \mathcal C} A \Rightarrow c_0$ for some $c_0\in \A_E\setminus \mathcal C$ has a negative answer.
	Hence, there exists a family $\{X_e\}_{e\in E}$ of discrete random variables that realizes $\mathcal C$ but not $c_0$.
	Thus, they also realize $\mathcal C$.
	Since $\{X_e\}_{e\in E}$ does not realize $c$ the family is nontrivial and therefore the CIR problem has a positive answer.
\end{proof}

\section{Almost multilinear matroids}
\label[section]{sec:almost_mutilinear}
This section presents our results in the almost multilinear setting. \cref{sec:approx_vec_reps} and \cref{sec:approx_dowling} generalize \cref{sec:vec_reps} and \cref{sec:entropic_reps}. \cref{sec:almost_undecidable} puts everything together to prove undecidability results parallel to \cref{sec:entropic_undecidability}, but for almost multilinear rather than entropic matroids.

\subsection{Approximate vector space representations}
\label{sec:approx_vec_reps}
Here we adapt~\Cref{def:vec_rep,def:vect_space_rep} to the approximate setting.
We use the notation for collection of linear maps introduced in~\Cref{sec:map_notation}
\begin{defn}
	Let $V$ be a vector space, $c\in \N$ and $E$ be a finite set.
	Further, let $\{W_e\}_{e\in E}$ be a collection of vector spaces with $\dim W_e=c$ and let $\{T_e: V\to W_e\}_{e\in E}$ be a collection of surjective linear maps. 
	Fix some $\varepsilon > 0$.
	\begin{enumerate}
		\item The maps $\{T_e\}_{e\in E}$ are \emph{independent with error $\varepsilon$} if $\rk(T_E)\ge c(|E|-\varepsilon)$.
		\item Fix $x\in E$. The map $T_x$ is \emph{determined with error $\varepsilon$} by $\{T_e\}_{e\in E\setminus\{x\}}$ if there exists a linear map $S:W_{E\setminus\{x\}}\to W_x$ such that
		\[\rk (T_x - S\circ T_{E\setminus\{x\}})\le c\varepsilon.\]
		That is, the normalized rank distance of $T_x$ and $S\circ T_{E\setminus\{x\}}$ is at most $\varepsilon$.
		In this case, $S$ is called an \emph{$\varepsilon$-determination map}.
	\end{enumerate}
For the sake of brevity we sometimes write that a set of maps is $\varepsilon$-independent, or that some map is $\varepsilon$-determined by a given collection of maps.
\end{defn}

\begin{lemm}\label{lem:approximate_rk_inverse}
	Let $A\in M_c(\F)$ be a matrix over a field $\F$ and let $\delta\ge 0$ be a real number.
	Then $\rk(A)\ge c(1-\delta)$ if and only if there exists an invertible matrix $D\in M_c(\F)$ such that $\rk(I_c-DA)\le c \delta$.
\end{lemm}
\begin{proof}
	Suppose $\rk(A)\ge c(1-\delta)$.
	Then there exists an invertible matrix $A'$ such that $A'-A$ has at most $c\delta$ nonzero rows: To construct such an $A'$ from the given matrix $A$, iteratively find a row of the matrix which is in the span of the others, and replace it by a row which is not in the row span. After $c-\rk(A)$ row replacements we obtain an invertible matrix and the process ends.
	
	For $D=A'^{-1}$, we have
	\[
		\rk(I_c - DA)=\rk(DA'- DA)=\rk(A'-A)\le c\delta.
	\]
	Conversely, suppose there exists a matrix $D$ with $\rk(I_c-DA)\le c \delta$.
	By the triangle inequality $\rk(I_c) \le \rk(I_c-DA)+ \rk(DA)$, and hence
	$\rk(DA)\ge c(1-\delta)$, which implies the claim.
\end{proof}
The following corollary is obvious from the lemma.
\begin{coro}\label[corollary]{cor:approximate_rk_inverse}
	Let $T:V\to W$ be a linear transformation between vector spaces of the same (finite) dimension $c$ and let $\delta \ge 0$. Then $\rk(T)\ge c(1-\delta)$ if and only if there exists an invertible transformation $S:W\to V$ such that $\rk(\mathrm{id}_V - S\circ T) \le c\delta$.
\end{coro}

\begin{defn}
	Let $M$ be a matroid on $E$.
	An \emph{$\varepsilon$-approximate vector space representation} of $M$ consists of $c\in \N$, a vector space $V$ and a collection of surjective linear maps $\{T_e : V \to W_e\}_{e\in E}$ with $\dim W_e=c$ such that
	\begin{enumerate}
		\item If $A\subseteq E$ is independent, the maps $\{T_e\}_{e\in A}$ are independent with error $\varepsilon$.
	\item If $C\subseteq E$ is a circuit and $e\in C$, then $T_e$ is determined with error $\varepsilon$ by $\{T_f\}_{f\in C\setminus\{e\}}$.
	\end{enumerate}
\end{defn}

\begin{theorem}\label{thm:multilinear_and_vector_space_reps}
	A simple matroid $M$ is multilinear if and only if it has a vector space representation.
	It is almost-multilinear if and only if it has an $\varepsilon$-approximate
	vector space representation for every $\varepsilon>0$.
\end{theorem}
The proof consists of simple but slightly lengthy calculations. 

\begin{notation}
	For $a,b\in\mathbb{R}$, we write $a\approx_{\varepsilon}b$
	as shorthand for $\left|a-b\right|\le\varepsilon$.	
\end{notation}

\begin{lemm}
	Let $M=\left(E,r\right)$ be a simple matroid. A vector space $V$
	and a collection of linear maps $\left\{ T_{e}:V\rightarrow W_{e}\right\} _{e\in E}$
	define a vector space representation of $M$ if and only if there
	exists $c\in\mathbb{N}$ such that for all $S\subseteq E$
	\[
	r(S) = \frac{1}{c}\rk (T_{S}).
	\]
\end{lemm}
\begin{proof}
	Suppose $V$ and the maps $\left\{ T_{e}\right\} _{e\in E}$ define
	a vector space representation of $M$. Then $c\coloneqq\dim W_{e}=\rk (T_{e})$
	is independent of $e\in E$. Each $S\subseteq E$ contains a maximal
	independent subset $S^{\prime}\subseteq S$ with $r\left(S\right)=r\left(S^{\prime}\right)=\left|S^{\prime}\right|$,
	which then satisfies
	\[
	\rk(T_{S^{\prime}})=\sum_{e\in S^{\prime}}\rk(T_{e})=c\left|S^{\prime}\right|.
	\]
	If $e\in S\setminus S^{\prime}$ then $e$ is in the closure of $S^{\prime}$,
	so $T_{e}$ is determined by $\left\{ T_{f}\right\} _{f\in S^{\prime}}$.
	It follows that $\rk (T_{S})=\rk(T_{S^{\prime}})=c\left|S^{\prime}\right|=c\, r\left(S\right)$.
	
	Conversely, suppose a vector space $V$ and linear maps $\left\{ T_{e}:V\rightarrow W_{e}\right\} _{e\in E}$
	are given such that $\rk(T_{S})=c \, r\left(S\right)$
	for all $S\subseteq E$. If $S\subseteq E$ is independent then $r\left(S\right)=\left|S\right|$,
	so that $\rk(T_{S})=c\left|S\right|=\sum_{e\in S}\rk(T_{e})$,
	and the maps $\left\{ T_{e}\right\} _{e\in S}$ are independent. If
	$C=\left\{ e_{1},\ldots,e_{n}\right\} $ is a circuit of $M$ then
	$r\left(C\right)=\left|C\right|-1$ and $C\setminus\left\{ e_{1}\right\} $
	is independent, so
	\[
	\rk(T_{C\setminus\left\{ e_{1}\right\} })=c\left(\left|C\right|-1\right)=\rk(T_{C}).
	\]
	The map $\pi:\bigoplus_{e\in C}W_{e}\rightarrow\bigoplus_{e\in C\setminus\left\{ e_{1}\right\} }W_{e}$
	which drops the $W_{e_{1}}$ coordinate satisfies $T_{C\setminus\left\{ e_{1}\right\} }=\pi\circ T_{C}$,
	so it induces an isomorphism $\img(T_{C})\rightarrow\img(T_{C\setminus\left\{ e_{1}\right\} })$
	($\pi$ must be a surjection onto $\img(T_{C\setminus\left\{ e_{1}\right\} })$
	because $T_{C\setminus\left\{ e_{1}\right\} }$ is a surjection; the
	dimensions of the two spaces are equal, so it is injective as well).
	Let $\psi:\img(T_{C\setminus\left\{ e_{1}\right\} })\rightarrow\img(T_{C})$
	be its inverse and let $\pi_{e_{1}}:\bigoplus_{e\in C}W_{e}\rightarrow W_{e_{1}}$
	be the projection to the $W_{e_{1}}$ summand.
	Then 
	\[
	\left(\pi_{e_{1}}\circ\psi\right)\circ T_{C\setminus\left\{ e_{1}\right\} }=\pi_{e_{1}}\circ T_{C}=T_{e_{1}},
	\]
	and $T_{e_{1}}$ is determined by $\left\{ T_{e}\right\} _{e\in C\setminus\left\{ e_{1}\right\} }$
	as required. 
\end{proof}
The proof of the analogous statement for almost-multilinear matroids
is very similar. The following simple claim is useful:
\begin{lemm}
	\label[lemma]{lem:left_inverse}Let $T:W_{1}\rightarrow W_{2}$ be
	a surjection. Then there exists a map $S:W_{2}\rightarrow W_{1}$
	such that $\rk\left(S\circ T-\mathrm{id}_{W_{1}}\right)\le\dim W_{1}-\dim W_{2}$.
\end{lemm}
\begin{proof}
	Pick a basis $v_{1},\ldots,v_{n}$ of $W_{2}$ and choose $w_{1}\in T^{-1}\left(v_{1}\right),\ldots,w_{n}\in T^{-1}\left(v_{n}\right)$.
	Then $w_{1},\ldots,w_{n}$ are independent since they have an independent
	image, and they can be completed to a basis $w_{1},\ldots,w_{n},w_{n+1},\ldots,w_{n+r}$
	of $W_{1}$. Define $S:W_{2}\rightarrow W_{1}$ on $v_{1},\ldots,v_{n}$
	by $S\left(v_{i}\right)=w_{i}$ and extend linearly. Then the map
	$S\circ T-\mathrm{id}_{W_{1}}$ vanishes on $\mathrm{span}\left(w_{1},\ldots,w_{n}\right)$,
	so its image is equal to the image of its restriction to $\mathrm{span}\left(w_{n+1},\ldots,w_{n+r}\right)$,
	and therefore has dimension at most $r=\dim W_{1}-\dim W_{2}$.
\end{proof}
\begin{lemm}
	\label[lemma]{lem:almost_multilinear_spaces_and_maps}Let $M=\left(E,r\right)$
	be a simple matroid, let $V$ be a vector space and let $\left\{ T_{e}:V\rightarrow W_{e}\right\} _{e\in E}$
	be a collection of linear maps. If $\left\{ T_{e}\right\} _{e\in E}$
	defines an $\varepsilon$-approximate vector space representation of
	$M$ then there exists $c\in\mathbb{N}$ such that
	\[
	\rk(T_{S})\approx_{c\left|E\right|\varepsilon} c\cdot r\left(S\right)\quad\text{for all \ensuremath{S\subseteq E}.}
	\]
	Conversely, if there exists $c\in\mathbb{N}$ such that $\rk(T_{e})=c$
	for all $e\in E$ and 
	\[
	\rk(T_{S})\approx_{c\varepsilon} c \cdot r\left(S\right)\quad\text{for all \ensuremath{S\subseteq E}}
	\]
	then the maps $\left\{ T_{e}\right\} _{e\in E}$ define a $2\varepsilon$-approximate
	vector space representation of $M$.
\end{lemm}
\begin{proof}
	Suppose $V$ and the maps $\left\{ T_{e}\right\} _{e\in E}$ define
	an $\varepsilon$-approximate vector space representation of $M$.
	Then $c\coloneqq\dim W_{e}=\rk(T_{e})$ is independent of
	$e\in E$. Each nonempty $S\subseteq E$ contains a maximal independent
	subset $S^{\prime}\subseteq S$ with $r\left(S\right)=r\left(S^{\prime}\right)=\left|S^{\prime}\right|$,
	which then satisfies
	\[
	\rk(T_{S^{\prime}})\approx_{c\varepsilon}\sum_{e\in S^{\prime}}\rk(T_{e})=c\left|S^{\prime}\right|.
	\]
	If $e\in S\setminus S^{\prime}$ then $e$ is in the closure of $S^{\prime}$,
	so $T_{e}$ is determined by $\left\{ T_{f}\right\} _{f\in S^{\prime}}$
	with error $\varepsilon$. It follows that 
	\[
	\rk(T_{S})\approx_{c\left(\left|S\right|-\left|S^{\prime}\right|\right)\varepsilon}\rk(T_{S^{\prime}})\approx_{c\varepsilon}\left|S^{\prime}\right|c=c\cdot r\left(S\right),
	\]
	so 
	\[
	\rk(T_{S})\approx_{c\left(\left|S\right|-\left|S^{\prime}\right|+1\right)} c\cdot r\left(S\right),
	\]
	where $\left|S\right|-\left|S^{\prime}\right|+1\le\left|E\right|$
	because $S^{\prime}\neq\emptyset$ (or $S$ consists of loops, and
	$M$ is not simple). 
	
	Conversely, suppose a vector space $V$ and linear maps $\left\{ T_{e}:V\rightarrow W_{e}\right\} _{e\in E}$
	are given such that $\rk(T_{S})\approx_{c\varepsilon}c \cdot r\left(S\right)$
	for all $S\subseteq E$. If $S\subseteq E$ is independent then $r\left(S\right)=\left|S\right|$,
	so that $\rk(T_{S})\approx_{c\varepsilon}c\left|S\right|=\sum_{e\in S}\rk(T_{e})$,
	and the maps $\left\{ T_{e}\right\} _{e\in S}$ are independent with
	error $\varepsilon$. If $C=\left\{ e_{1},\ldots,e_{n}\right\} $
	is a circuit of $M$ then $r\left(C\right)=\left|C\right|-1$ and
	$C\setminus\left\{ e_{1}\right\} $ is independent, so
	\[
	\rk(T_{C\setminus\left\{ e_{1}\right\} })\approx_{c\varepsilon}c\left(\left|C\right|-1\right)\approx_{c\varepsilon}\rk(T_{C}),
	\]
	and $\rk(T_{C})-\rk(T_{C\setminus\left\{ e_{1}\right\} })\le2c\varepsilon$.
	The map $\pi:\bigoplus_{e\in C}W_{e}\rightarrow\bigoplus_{e\in C\setminus\left\{ e_{1}\right\} }W_{e}$
	which drops the $W_{e_{1}}$--coordinate satisfies $T_{C\setminus\left\{ e_{1}\right\} }=\pi\circ T_{C}$,
	so it induces a surjection $\img(T_{C})\rightarrow\img(T_{C\setminus\left\{ e_{1}\right\} })$
	($\pi$ is a surjection onto $\img(T_{C\setminus\left\{ e_{1}\right\} })$
	because $T_{C\setminus\left\{ e_{1}\right\} }$ is a surjection).
	\cref{lem:left_inverse} implies that there exists $\psi:\img(T_{C\setminus\left\{ e_{1}\right\} })\rightarrow\img(T_{C})$
	such that
	\[
	\rk\left(\psi\circ\pi-\mathrm{id}_{\img(T_{C})}\right)\le2c\varepsilon.
	\]
	Denote the projection to the $e_{1}$-summand $\bigoplus_{e\in C}W_{e}\rightarrow W_{e_{1}}$
	by $\pi_{e_{1}}$. Then $\pi_{e_{1}}\circ T_{C}=T_{e_{1}}$ by definition,
	and 
	\begin{align*}
		\pi_{e_{1}}\circ\left(\psi\circ\pi-\mathrm{id}_{\img({T_{C}})}\right)\circ T_{C}&=\pi_{e_{1}}\circ\psi\circ\left(\pi\circ T_{C}\right)-\pi_{e_{1}}\circ T_{C}\\
		&=\left(\pi_{e_{1}}\circ\psi\right)\circ T_{C\setminus\left\{ e_{1}\right\} }-T_{e_{1}}
	\end{align*}
	has rank at most $2c\varepsilon$ (since $\left(\psi\circ\pi-\mathrm{id}_{\img({T_{C}})}\right)$
	has rank at most $2c\varepsilon$). This shows that $T_{e_{1}}$ is
	determined by $\left\{ T_{e}\right\} _{e\in C\setminus\left\{ e_{1}\right\} }$
	with error at most $2\varepsilon$.
\end{proof}
\begin{lemm}
	\label[lemma]{lem:almost_multilinear_rounding}A matroid $M=\left(E,r\right)$
	is almost multilinear if and only if for every $\varepsilon>0$ there
	exists a linear polymatroid $\left(\widetilde{E},\widetilde{r}\right)$ and
	a $c\in\mathbb{N}$
	\[
	\left\Vert r-\frac{1}{c}\widetilde{r}\right\Vert _{\infty}<\varepsilon
	\]
	and in addition $\widetilde{r}(e)=c$ for all $e\in E$.
\end{lemm}
\begin{proof}
	One direction is trivial: if for every $\varepsilon>0$ there exists
	a polymatroid as in the statement then $M$ is almost multilinear.
	
	Conversely, suppose $M$ is almost multilinear and let $\varepsilon>0$.
	Denote $\varepsilon^{\prime}=\frac{1}{\left|E\right|+1}\varepsilon$.
	Take a linear polymatroid $\left(\widetilde{E},\widetilde{r}\right)$ and
	a $c\in\mathbb{N}$
	\[
	\lim_{n\rightarrow\infty}\left\Vert r-\frac{1}{c}\widetilde{r}\right\Vert _{\infty}<\varepsilon^{\prime}.
	\]
	Let $V$ be a vector space and let $\left\{ W_{e}\right\} _{e\in E}$
	be subspaces representing $\left(\widetilde{E},\widetilde{r}\right)$. Assume
	$\dim V\ge c$ (by enlarging $V$ if necessary). For each $e\in E$
	denote $d_{e}=\dim W_{e}$, and take a basis $b_{1}^{e},\ldots,b_{d_{e}}^{e}$
	for $W_{e}$. If $c>d_{e}$ add vectors to the basis such that $b_{1}^{e},\ldots,b_{d_{e}}^{e},\ldots,b_{c}^{e}$
	are linearly independent; if $c<d_{e}$ remove the last vectors from
	the list. Then define 
	\[
	W_{e}^{\prime}=\mathrm{span}\left\{ b_{1}^{e},\ldots,b_{c}^{e}\right\}.
	\]
	Consider the subspaces $\left\{ W_{e}^{\prime}\right\} _{e\in E}$.
	For any $S\subseteq E$ we have 
	\[
	\left|\dim\left(\sum_{e\in S}W_{e}^{\prime}\right)-\dim\left(\sum_{e\in S}W_{e}\right)\right|\le\sum_{e\in S}\left|c-d_{e}\right|\le\sum_{e\in E}\left|c-d_{e}\right|
	\]
	where $\left|c-d_{e}\right|=c\left|1-\frac{1}{c}\widetilde{r}(e)\right|\le c\left\Vert r-\frac{1}{c}\widetilde{r}\right\Vert _{\infty}$.
	
	In particular, if $r^{\prime}$ is the rank function of the polymatroid
	represented by $\left\{ W_{e}^{\prime}\right\} _{e\in E}$ then 
	\[
	\left\Vert r^{\prime}-\widetilde{r}\right\Vert _{\infty}\le\left|E\right|c\left\Vert r-\frac{1}{c}\widetilde{r}\right\Vert _{\infty},
	\]
	and therefore
	\begin{align*}
		\left\Vert r-\frac{1}{c}r^{\prime}\right\Vert _{\infty}&\le\left\Vert r-\frac{1}{c}\widetilde{r}\right\Vert _{\infty}+\left\Vert \frac{1}{c}r^{\prime}-\frac{1}{c}\widetilde{r}\right\Vert _{\infty}=\left\Vert r-\frac{1}{c}\widetilde{r}\right\Vert _{\infty}+\frac{1}{c}\left\Vert r^{\prime}-\widetilde{r}\right\Vert _{\infty}\\
		&\le\left\Vert r-\frac{1}{c}\widetilde{r}\right\Vert _{\infty}\left(\left|E\right|+1\right)<\varepsilon.\qedhere
	\end{align*}
\end{proof}
\begin{proof}
	[Proof of \Cref{thm:multilinear_and_vector_space_reps}]Let $V$ be
	a finite dimensional vector space and let $\left\{ W_{e}\right\} _{e\in E}$
	be a finite indexed collection of subspaces. For each $W\le V$ denote
	by $W^{0}\le V^{*}$ the annihilator of $W$ in the dual space, and
	recall $\dim W^{0}=\dim V-\dim W$. Define $T_{e}:V^{*}\rightarrow V^{*}/W_{e}^{0}$
	to be the quotient map. The indexed collection of maps $\left\{ T_{e}:V^{*}\rightarrow V^{*}/W_{e}^{0}\right\} _{e\in E}$
	satisfies
	\[
	\ker T_{S}=\bigcap_{e\in S}W_{e}^{0}=\left(\sum_{e\in S}W_{e}\right)^{0}
	\]
	for any $S\subseteq E$, where $T_{S}$ is the map
	\[
	T_{S}:V^{*}\rightarrow\bigoplus_{e\in S}V^{*}/W_{e}^{0}.
	\]
	Thus 
	\[
	\rk(T_{S})=\dim V^{*}-\dim\ker T_{S}=\dim V^{*}-\left(\dim V-\dim\left(\sum_{e\in S}W_{e}\right)\right)
	\]
	\[
	=\dim\left(\sum_{e\in S}W_{e}\right),
	\]
	and $\dim V^{*}/W_{e}^{0}=\dim W_{e}$ for all $e\in E$. 
	
	It follows that the subspaces $\left\{ W_{e}\right\} _{e\in E}$ define
	a multilinear representation of a matroid $\left(E,r\right)$ if and
	only if the maps $\left\{ T_{e}\right\} _{e\in E}$ define a vector
	space representation. Similarly, by \cref{lem:almost_multilinear_rounding}
	and \cref{lem:almost_multilinear_spaces_and_maps} the matroid $M=\left(E,r\right)$
	is almost multilinear if and only if for every $\varepsilon>0$ it
	has an $\varepsilon$-approximate vector space representation.
\end{proof}
\subsection{Almost multilinear Dowling geometries}
\label{sec:approx_dowling}

The next two theorems provide sufficient conditions for a partial Dowling geometry to be almost multilinear.
Moreover, we discuss  a group-theoretic consequence of a Dowling geometry being almost multilinear.

\begin{theorem}\label{thm:epsilon_reps}
	Let $G=\left\langle S\mid R\right\rangle $
	be a group with a given symmetric triangular presentation and fix $\varepsilon \ge 0$.
	Let $\rho:S\rightarrow\mathrm{GL}(W)$
	be an $\varepsilon/18$-approximate representation of $\langle S \mid R \rangle$, where $W$ is a finite dimensional vector space over a field $\F$.
	Suppose that
	\begin{enumerate}
		\item $d_{\rk}(\rho(s),\rho(s'))\ge 1-\varepsilon/18$ for all distinct $s,s'\in S$,
		\item For all triples $s,s',s'' \in S$ (not necessarily distinct) either
		\begin{align*}
		&d_{\rk}(\rho(s'')\rho(s')\rho(s),\mathrm{id}_W)\le \varepsilon/18 \text{ or}\\
		&d_{\rk}(\rho(s'')\rho(s')\rho(s),\mathrm{id}_W)\ge 1 - \varepsilon/18.\end{align*}
		\item If $s,s',s'' \in S$ (not necessarily distinct) satisfy $d_{\rk}(\rho(s'')\rho(s')\rho(s),\mathrm{id}_W)\le \varepsilon/18$ then $s'' s' s = e$ is a relation in $R$.
	\end{enumerate}
	Then the partial Dowling geometry of
	 the presentation $\left\langle S\mid R\right\rangle $ has an $\varepsilon$-approximate vector space representation.
	
	Moreover, if the approximate representation $\rho$ just satisfies the assumptions (a) and (b) then some matroid among the partial Dowling geometries $\mathcal{M}_{S,R}$ subordinate to $\langle S \mid R \rangle$ has an $\varepsilon$-approximate vector space representation.
\end{theorem}
\begin{proof}
	The second statement (``Moreover, ...'') follows from the first after adding the relations $s''s's=e$ to $R$ whenever $\rho(s'' s' s)=e$ holds.
	Note that by the definition of the subordinate partial Dowling geometries, the matroid of this new presentation is a member of $\mathcal{M}_{S,R}$.
	
	As in \Cref{def:GDG}, we denote the partial Dowling geometry
	of $\left\langle S\mid R\right\rangle $ by $M$, the ground set by
	$E$, and the special basis by $B=\left\{ b_{1},b_{2},b_{3}\right\} $.
	We construct an $\varepsilon$-approximate vector space representation of $M$.
	
	Set $c=\dim W$ and for each $e\in E$ set $W_{e}=W$. Let $V=W_{b_{1}}\oplus W_{b_{2}}\oplus W_{b_{3}}$,
	and let $T_{b_{i}}:V\rightarrow W_{b_{i}}$ be given by the projection.
	Let $i,j\in\{1,2,3\}$ be two distinct indices, and suppose $j$ is
	the element following $i$ in the cyclic ordering. Let $s\in S$ be
	any element. Define
	\begin{align*}
	T_{s_{i}}:V=W_{b_{1}}\oplus W_{b_{2}}\oplus W_{b_{3}}\rightarrow W_{s_{i}} \\
	T_{s_{i}}\left(v_{1},v_{2},v_{3}\right)=v_{j}-\rho(s)\left(v_{i}\right),
	\end{align*}
	or in other words $T_{s_{i}}=T_{b_{j}}-\rho(s)T_{b_{i}}$. 
	
	(One can come up with this guess for the maps by starting with the
	following determination map for $v_{j}$ given $v_{s_{i}}=T_{s_{i}}(v_{1},v_{2},v_{3})$
	and $v_{i}$: $S(v_{i},v_{s_{i}})=\rho(s_{i})v_{i}+v_{s_{i}}$. Such
	determination maps ``compose correctly'' in the sense of \Cref{thm:Dowling_prob_rep}, condition (b).
	Another way is to inspect the matrix representations of Dowling geometries.)
	In order to prove the required $\varepsilon$-independence and $\varepsilon$-determination conditions we first establish the following claims.
	\begin{description}
		\item[Claim 1] $d_{\rk}(\rho(s)^{-1},\rho(s^{-1}))\le \varepsilon/9$.
		\item[Claim 2] Fix $\varepsilon'\ge 0$. Let $S\subseteq E$ with $|S|=3$.
			If $T_{b_i}$ is determined by $\{T_e\}_{e\in S}$ with error $\varepsilon'/3$ for all $1\le i\le 3$ then  $\{T_e\}_{e\in S}$ is independent with error $\varepsilon'$.
		\item[Claim 3] Let $S\subseteq E$ with $|S|=3$. If $\{T_e\}_{e\in S}$ is independent with error $\varepsilon'/3$ then $T_{b_i}$ is determined by $\{T_e\}_{e\in S}$ with error $\varepsilon'$ for all $1\le i\le 3$.
	\end{description}
	\begin{proof}[Proof of Claim 1]
		Applying assumption (b) to the relation $s^{-1}se=e$, we obtain 
		\[d_{\rk}(\rho(s^{-1})\rho(s)\rho(e),\mathrm{id}_W)\le \varepsilon/18.\]
		Since $d_{\rk}(\rho(e),\mathrm{id}_W)\le\varepsilon/18$, we have
		\begin{align*}
		& d_{\rk}(\rho(s^{-1})\rho(s),\mathrm{id}_W) \\ 
		\le\,& d_{\rk}(\rho(s^{-1})\rho(s)\circ\mathrm{id}_W, \rho(s^{-1})\rho(s)\rho(e)) + d_{\rk}(\rho(s^{-1})\rho(s)\rho(e),\mathrm{id}_W) \le\varepsilon/9. 
		\end{align*}
	by \cref{rem:normalized_rank_metric} and the triangle inequality.
	\end{proof}
	\begin{proof}[Proof of Claims 2 and 3]
		Given a basis $S\subseteq E$ of $M$ (so that in particular $\left|S\right|=3)$
		consider the map 
		\[T_S:V = W_{b_1} \oplus W_{b_2} \oplus W_{b_3} \to W_S=\bigoplus_{e\in S} W_e.\]
		
		Suppose each $T_{b_i}$ is $(\varepsilon'/3)$-determined by $\{T_e\}_{e\in S}$ . Then there exist maps $\widetilde{T}_1$, $\widetilde{T}_2$, and $\widetilde{T}_3$ such that
		\[\rk(T_{b_i} - \widetilde{T}_i \circ T_S) \le c\varepsilon'/3\]
		for each $1\le i\le 3$. Define
		\begin{align*}
			\widetilde{T}:W_S &\to V=W_{b_1}\oplus W_{b_2} \oplus W_{b_3} \\
			w=(w_e)_{e\in S} & \mapsto (\widetilde{T}_1(w),\widetilde{T}_2(w),\widetilde{T}_3(w))
		\end{align*}
		and observe that $T_B$ differs from $\widetilde{T}\circ T_S$ on a subspace of dimension at most $3\cdot c\varepsilon'/3=c\varepsilon'$. In particular $T_S$ has rank at least $\rk(T_B)-c\varepsilon'$, and thus $\{T_e\}_{e\in S}$ are $\varepsilon'$-independent.
		
		Suppose $\{T_e\}_{e\in S}$ are independent with error $\varepsilon'/3$. Then by definition $T_S$ has rank at least $c(1-\varepsilon'/3)$. Thus by \cref{cor:approximate_rk_inverse} there exists a map $\widetilde{T}:W_S \to V$ such that 
		\[\rk(\mathrm{id}_V - \widetilde{T}\circ T_S) \le 3c(\varepsilon'/3)=c\varepsilon'.\]
		Composing with $T_{b_i}$ for $1\le i\le 3$, we find
		\[\rk(T_{b_i} - (T_{b_i}\circ\widetilde{T})\circ T_S) \le c\varepsilon',\]
		so that each $T_{b_i}$ is determined by $\{T_e\}_{e\in S}$ with error $\varepsilon'$. \qedhere
	\end{proof}
	
	We now verify that the correct independence and determination conditions
	hold with error at most $\varepsilon$ for the maps $\left\{ T_{e}:V\rightarrow W_{e}\right\} _{e\in E}$.
	
	For the independence conditions there are several cases. It suffices
	to check the condition for bases of $M$ (recall these are all of size $3=\rk(M)$).
	In each case we will show that $\{T_e\}_{e\in S}$ $\varepsilon/3$-determines $T_{b_i}$ for all $1\le i\le3$, which suffices by Claim~2.
	\begin{enumerate}
		\item For $\left\{ b_{1},b_{2},b_{3}\right\} $ the statement is clear:
		$T_{b_{i}}$ ($i=1,2,3$) are distinct projections onto summands of
		$V=W_{b_{1}}\oplus W_{b_{2}}\oplus W_{b_{3}}$.
		\item\label{it:ind_b} For subsets of the form $\left\{ b_{1},b_{2},s_{2}\right\} $, we
		have
		\[
		T_{s_{2}}+\rho(s)T_{b_{2}}=T_{b_{3}},
		\]
		so that $T_{b_{3}}$ is determined (with error $0$) by $\left\{ T_{b_{1}},T_{b_{2}},T_{s_{2}}\right\} $,
		and we reduce to the previous case. The same holds for subsets of
		the form $\left\{ b_{1},b_{2},s_{3}\right\} $, or similar subsets
		with cyclic shifts of the indices.
		\item Subsets of the form $\left\{ s_{1},s_{2}',b_{1}\right\} $
		(up to shifts of the indices, with $s=s'$ allowed) are similar:
		we first observe that $T_{b_{2}}$ is determined (with error $0$) by $\left\{ T_{b_{1}},T_{s_{1}}\right\} $
		and then reduce to~\eqref{it:ind_b}.
		The same idea works for subsets of the form
		$\left\{ s_{1},s_{2}',b_{2}\right\} $.
		\item\label{it:ind_d} For subsets of the form $\left\{ s_{1},s_{1}'\right\} $ we
		note that
		\[
		\rho(s)T_{b_{1}}+T_{s_{1}}=\rho(s')T_{b_{1}}+T_{s_{1}'}=T_{b_{2}}
		\]
		and therefore $\left(\rho(s)-\rho(s')\right)T_{b_{1}}=T_{s_{1}'}-T_{s_{1}}$.
		By assumption we know that $\rk(\rho(s)-\rho(s'))\ge c(1-\frac{\varepsilon}{18})$.
		Thus by~\Cref{cor:approximate_rk_inverse}, there is a $\widetilde{T}\in \mathrm{GL}(W)$ such that 
		\[\rk(\mathrm{id}_W-\widetilde{T}\circ(\rho(s)-\rho(s')))\le c\varepsilon/18.\]
		Precomposing with $T_{b_1}$ and using the identity $\left(\rho(s)-\rho(s')\right)T_{b_{1}}=T_{s_{1}'}-T_{s_{1}}$, we find
		\[\rk(T_{b_1} - \widetilde{T}\circ(T_{s'_1} - T_{s_1})) \le c\varepsilon/18.\]
		Thus $T_{b_{1}}$ is determined by $\{T_{s'_1},T_{s_1}\}$
		with error~$\varepsilon/18$.
		
		Using $\rho(s)T_{b_{1}}+T_{s_{1}}=T_{b_{2}}$ and composing the maps in the previous rank inequality with $\rho(s)$, we find
		\begin{align*}
		 &\rk(T_{b_2} - [\rho(s)\widetilde{T}\circ(T_{s'_1} - T_{s_1}) - T_{s_1}]) \\
		 =&\rk([\rho(s)T_{b_1} + T_{s_1}] - [\rho(s)\widetilde{T}\circ(T_{s'_1} - T_{s_1}) - T_{s_1}]) \\
		 =&\rk(\rho(s)\circ T_{b_1} - \rho(s)\widetilde{T}\circ(T_{s'_1} - T_{s_1})) \le c\varepsilon/18.
		\end{align*}
		Observe that $\rho(s)\widetilde{T}\circ(T_{s'_1} - T_{s_1}) - T_{s_1}$ is the composition of a map 
		\[W_{s'_1}\oplus W_{s_1} \to W=W_{b_2}\]
		on $T_{\{s'_1,s_1\}}$. Therefore $T_{b_{2}}$ is determined with error $\varepsilon/18$ by $T_{\{s'_1,s_{1}\}}$.
		
		By Claims 2 and 3, this computation yields the independence condition for subsets of the form $S=\left\{ s_{1},s_{1}',b_{3}\right\} $: by our computation, the maps $\{T_e\}_{e\in S}$ determine $T_{b_1}$ and $T_{b_2}$ with error $\varepsilon/18$ each, so that $T_B$ is determined with error at most $\varepsilon/9$ and the maps are $\varepsilon/3$-independent.
		
		It also yields the independence condition for subsets of the form $S=\left\{ s_{1},s_{1}',s_{2}''\right\} $ which are independent in $M$
		(up to shifts of the indices, with $s,s',s''$
		not necessarily distinct): the maps $\{T_e\}_{e\in S}$ determine each $T_{b_i}$ with error $\varepsilon/18$.
		\item Finally, for subsets of the form $\left\{ s_{1},s_{2}',s_{3}''\right\} $
		with $s''s's\neq e$, we have
		\begin{align*}
		&T_{s_{3}''}+\rho(s'')T_{s_{2}'}+\rho(s'')\rho(s')T_{s_{1}}\\
		=&\left[T_{b_{1}}-\rho(s'')T_{b_{3}}\right]+\rho(s'')\left[T_{b_{3}}-\rho(s')T_{b_{2}}\right]+\rho(s'')\rho(s')\left[T_{b_{2}}-\rho(s)T_{b_{1}}\right]\\
		=&T_{b_{1}}-\rho(s'')\rho(s')\rho(s)T_{b_{1}}=\left[\mathrm{id}_{W}-\rho(s'')\rho(s')\rho(s)\right]T_{b_{1}}.
		\end{align*}
		By assumption $\rk(\mathrm{id}_W - \rho(s''^{-1})\rho(s')\rho(s))\ge c(1-\varepsilon/18)$.
		By \cref{cor:approximate_rk_inverse} there exists a $\widetilde{T}\in \mathrm{GL}(W)$ such that
		\[\rk(\mathrm{id}(W) - \widetilde{T}\circ \rho(s'')\rho(s')\rho(s))\le c\varepsilon/18.\]
		As in case \eqref{it:ind_d}, this implies that $T_{b_1}$ is determined by $\{T_e\}_{e\in S}$ with error $\varepsilon/18$ for each $1\le i\le 3$.
		By permuting the indices $\left(1,2,3\right)$ and generators $\left(s,s',s''\right)$
		cyclically, we find similar expressions for $T_{b_{2}}$ and $T_{b_{3}}$.
		This shows each $T_{b_{i}}$is determined by $\left(T_{s_{1}},T_{s_{2}'},T_{s_{3}''}\right)$ with error $\varepsilon/18$, which by Claim 2 
		implies the claimed independence.
	\end{enumerate}
	We now consider the circuits and show that the determination conditions
	are satisfied.
	\begin{enumerate}
		\item If $C$ is a circuit of size $4$, let $x\in C$. The subset $C\setminus\left\{ x\right\} $
		is a basis of $M$ (since this subset is independent and $M$ has
		rank $3$), so $\left\{ T_{e}\right\} _{e\in C\setminus\{ x\} }$
		determine $T_{b_{1}}$, $T_{b_{2}}$ and $T_{b_{3}}$ with error $\varepsilon/3$ by the above arguments.
		
		By construction we can express $T_x$ by $T_x=\sum_{i=1}^3A_iT_{b_i}$ for some maps $A_i\in \mathrm{GL}(W)$.
		The above argument also implies that $\left\{ T_{e}\right\} _{e\in C\setminus\{ x\} }$ determines $A_iT_{b_i}$ with error $\varepsilon/3$ for all $1\le i\le 3$.
		Therefore $\left\{ T_{e}\right\} _{e\in C\setminus\{ x\} }$ determines $T_x$ with error $\varepsilon$.
		\item If $C$ consists of $3$ elements of the flat spanned by $\left\{ b_{1},b_{2}\right\} $
		then any subset consisting of two elements is of the form $S=\{b_1,b_2\}$, $S=\left\{ b_{i},s_{1}\right\} $
		($i\in\{1,2\}$), or $S=\left\{ s_{1},s_{1}'\right\} $
		(where $s\neq s'$ in $S$). In the first case it is clear that $\{T_e\}_{e\in S}$ determines $T_x$ (with error $0$) for $x$ the unique element of $C\setminus S$.
		
		For the latter two cases, note that in the cases~\eqref{it:ind_b} and~\eqref{it:ind_d} of the independence conditions
		it is shown that $\left\{ T_{e}\right\} _{e\in S}$
		determines $T_{b_{1}}$ and $T_{b_{2}}$ in either case with error $\varepsilon/18$.
		Therefore, $\left\{ T_{e}\right\} _{e\in S}$ determines $T_{e'}$ with error $\varepsilon/18$ for
		any $e'$ in the flat spanned by $\left\{ b_{1},b_{2}\right\} $ by an analogous argument as in the previous case, and
		in particular for $x$ the unique element of $C\setminus S$.
		\item Suppose $C=\left\{ s_{1},s_{2}',s_{3}''\right\} $
		where $s''s's=e$, or equivalently $s''=\left(s's\right)^{-1}=s^{-1}s^{'-1}$.
		We show that $T_{s_3''}$ is determined by $\left\{ T_{s_{1}},T_{s_{2}'}\right\}$ with error $\varepsilon$.
		To this end, we compute
		\begin{align*}
			&-\rho(s)^{-1}T_{s_1}-\rho(s)^{-1}\rho(s')^{-1}T_{s_2'}\\
			=&-\rho(s)^{-1}\left[T_{b_{2}}-\rho\left(s\right)T_{b_{1}}\right]-\rho\left(s\right)^{-1}\rho\left(s'\right)^{-1}\left[T_{b_{3}}-\rho\left(s'\right)T_{b_{2}}\right]\\
			=& T_{b_{1}} - \rho\left(s\right)^{-1}\rho\left(s'\right)^{-1}T_{b_{3}}
		\end{align*}
	By assumption we have $\rk(\rho(s'')\rho(s')\rho(s)-\mathrm{id}_W)\le \varepsilon/18$. Composing the transformation with $\rho(s)^{-1}\rho(s')^{-1}$ from the right, we obtain
	\[\rk(\rho(s'')-\rho(s)^{-1}\rho(s)^{-1})\le \varepsilon/18.\]
	Therefore
	\[\rk(T_{s_3''} - [T_{b_{1}} - \rho\left(s\right)^{-1}\rho\left(s'\right)^{-1}T_{b_{3}}]) \le \varepsilon/18,\]
	which implies
	\[\rk([-\rho(s)^{-1}T_{s_1}-\rho(s)^{-1}\rho(s')^{-1}T_{s_2'}] - T_{s_3''})\le \varepsilon/18,\]
	so the map $T_{s_3''}$ is determined by $\{T_{s_1},T_{s_2'}\}$ as required. \qedhere
	\end{enumerate}
\end{proof}

\begin{theorem}\label[theorem]{thm:almost_multilinear_to_group}
	Suppose the partial Dowling geometry $M=\left(E,r\right)$ associated
	to a finitely presented group $G=\left\langle S\mid R\right\rangle $
	is almost multilinear. Then $s\neq s^{\prime}$ in $G$ for all distinct
	$s,s^{\prime}\in S$.
\end{theorem}

The next lemma is helpful in part of the computation.
\begin{lemm}\label[lemma]{lem:invertible_determination}
	Let $M=\left(E,r\right)$ be a matroid and let $E^{\prime}=\left\{ e_{1},e_{2},e_{3}\right\} \subseteq E$
	be a subset such that $r\left(\left\{ e_{1},e_{2},e_{3}\right\} \right)=2$
	and $r\left(\left\{ e_{i},e_{j}\right\} \right)=2$ for all distinct
	$1\le i,j\le3$. Let an $\varepsilon$-approximate multilinear representation
	of $M$ be given by the vector space $V$ and the maps $\left\{ T_{e}:V\rightarrow W_{e}\right\} _{e\in E}$,
	and denote by $c$ the dimension of each of the vector spaces $\left\{ W_{e}\right\} _{e\in E}$
	(recall this dimension is constant by assumption). Then there are
	$6\varepsilon$-determination functions $f:W_{e_{1}}\oplus W_{e_{2}}\rightarrow W_{e_{3}}$
	and $g:W_{e_{1}}\oplus W_{e_{3}}\rightarrow W_{e_{2}}$ such that
	the following holds. Pick bases for $W_{e_{i}}$ ($1\le i\le3$) and
	identify the spaces with $\mathbb{F}^{c}$. Then there are matrices
	$A_{1},A_{2}\in M_{c}\left(\mathbb{F}\right)$ such that $A_{2}$
	is invertible, and $f,g$ satisfy 
	\[
	f\left(v_{1},v_{2}\right)=A_{1}v_{1}+A_{2}v_{2},\quad g\left(v_{1},v_{3}\right)=-A_{2}^{-1}A_{1}v_{1}+A_{2}^{-1}v_{3}.
	\]
\end{lemm}
\begin{proof}
	Take an $\varepsilon$-determination map $\widetilde{f}:W_{e_{1}}\oplus W_{e_{2}}\rightarrow W_{e_{3}}$.
	Since $\widetilde{f}$ is linear, it is of the form $\widetilde{f}\left(v_{1},v_{2}\right)=A_{1}v_{1}+\widetilde{A}_{2}v_{2}$
	for some $A_{1},\widetilde{A}_{2}\in M_{c}\left(\mathbb{F}\right)$. Define
	\[
	\widetilde{T}_{e_{3}}:V\rightarrow W_{e_{3}}
	\]
	\[
	\widetilde{T}_{e_{3}}\left(v\right)=\widetilde{f}\left(T_{e_{1}}(v),T_{e_{2}}(v)\right).
	\]
	Further define 
	\[
	\widetilde{T}_{(e_{1},e_{3})}:V\rightarrow W_{e_{1}}\oplus W_{e_{3}}
	\]
	
	\[
	\widetilde{T}_{(e_{1},e_{3})}(v)=\left(T_{e_{1}}(v),\widetilde{T}_{e_{3}}(v)\right)
	\]
	by analogy with $T_{(e_{1},e_{3})}$. 
	
	By definition $d_{\mathrm{rk}}\left(T_{e_{3}},\widetilde{T}_{e_{3}}\right)\le\varepsilon$,
	and hence also $d_{\mathrm{rk}}\left(T_{(e_{1},e_{3})},\widetilde{T}_{(e_{1},e_{3})}\right)\le\varepsilon$.
	In particular, $\mathrm{rk}\left(\widetilde{T}_{(e_{1},e_{3})}\right)\ge\mathrm{rk}\left(T_{(e_{1},e_{3})}\right)-c\varepsilon\ge2c-2c\varepsilon$.
	Observe that 
	\[
	\widetilde{T}_{(e_{1},e_{3})}(v)=\left(T_{e_{1}}(v),\widetilde{T}_{e_{3}}(v)\right)=\left(T_{e_{1}}(v),\widetilde{f}\left(T_{e_{1}}(v),T_{e_{2}}(v)\right)\right).
	\]
	Hence, for 
	\[
	F:W_{e_{1}}\oplus W_{e_{2}}\rightarrow W_{e_{1}}\oplus W_{e_{3}}
	\]
	\[
	F\left(v_{1},v_{2}\right)=\left(v_{1},f\left(v_{1},v_{2}\right)\right)=\left(v_{1},A_{1}v_{1}+\widetilde{A}_{2}v_{2}\right),
	\]
	we have $\widetilde{T}_{(e_{1},e_{3})}=F\circ T_{(e_{1},e_{2})}$. In
	particular, $F$ has rank at least $2c-2c\varepsilon$. Identifying
	$W_{e_{1}}\oplus W_{e_{2}}$ and $W_{e_{1}}\oplus W_{e_{3}}$ with
	$\mathbb{F}^{2c}$ via the chosen bases, we represent $F$ by the
	block matrix
	\[
	\left[\begin{matrix}I & 0\\
		A_{1} & \widetilde{A}_{2}
	\end{matrix}\right]
	\]
	(note that indeed $\left[\begin{smallmatrix}I & 0\\
		A_{1} & \widetilde{A}_{2}
	\end{smallmatrix}\right]\left[\begin{smallmatrix}v_{1}\\
		v_{2}
	\end{smallmatrix}\right]=\left[\begin{smallmatrix}v_{1}\\
		A_{1}v_{1}+\widetilde{A}_{2}v_{2}
	\end{smallmatrix}\right]$). Since such a block matrix has rank $c+\mathrm{rk}\left(\widetilde{A}_{2}\right)$,
	we obtain $\mathrm{rk}\left(\widetilde{A}_{2}\right)\ge c-2c\varepsilon$.
	By \cref{lem:approximate_rk_inverse} there is an invertible matrix $A_{2}$ such that
	$d_{\mathrm{rk}}\left(A_{2},\widetilde{A}_{2}\right)\le 2\varepsilon$.
	Define
	\[
	f:W_{e_{1}}\oplus W_{e_{2}}\rightarrow W_{e_{3}}
	\]
	\[
	f\left(v_{1},v_{2}\right)=A_{1}v_{1}+A_{2}v_{2}
	\]
	as well as
	\[
	g:W_{e_{1}}\oplus W_{e_{3}}\rightarrow W_{e_{2}}
	\]
	\[
	g\left(v_{1},v_{3}\right)=-A_{2}^{-1}A_{1}v_{1}+A_{2}^{-1}v_{3}.
	\]
	It remains to show that these are $6\varepsilon$-determination functions.
	First observe that $\left(\widetilde{f}-f\right)\left(v_{1},v_{2}\right)=\left(A_{2}-\widetilde{A}_{2}\right)v_{2}$,
	and therefore 
	\[
	d_{\mathrm{rk}}\left(\widetilde{f},f\right)=d_{\mathrm{rk}}\left(\widetilde{A}_{2},A_{2}\right)=\frac{1}{c}\mathrm{rk}\left(\widetilde{A}_{2}-A_{2}\right)\le2\varepsilon.
	\]
	Hence also 
	\[
	d_{\mathrm{rk}}\left(f\circ T_{(e_{1},e_{2})},T_{e_{3}}\right)\le d_{\mathrm{rk}}\left(f\circ T_{(e_{1},e_{2})},\widetilde{f}\circ T_{(e_{1},e_{2})}\right)+d_{\mathrm{rk}}\left(\widetilde{f}\circ T_{(e_{1},e_{2})},T_{e_{3}}\right) \le3\varepsilon,
	\]
	so $f$ is a $3\varepsilon$-determination function. For $g$, observe
	that for all $v\in V$, denoting $v_{1}=T_{e_{1}}(v)$ and $v_{2}=T_{e_{2}}(v)$
	we have $g\left(v_{1},f\left(v_{1},v_{2}\right)\right)=v_{2}$ by
	construction. So it suffices to bound
	\[
	d_{\mathrm{rk}}\left(g\circ T_{(e_{1},e_{3})},\quad\left[v\mapsto g\left(T_{e_{1}}(v),f\left(T_{e_{1}}(v),T_{e_{2}}(v)\right)\right)\right]\right),
	\]
	since the map on the right equals $T_{e_{2}}$. By \cref{rem:normalized_rank_metric}
	(using the fact that both maps are constructed by composing a map
	with $g$) this is at most
	\[
	2d_{\mathrm{rk}}\left(T_{(e_{1},e_{3})},\quad\left[v\mapsto\left(T_{e_{1}}(v),f\circ T_{(e_{1},e_{2})}(v)\right)\right]\right)
	\]
	\[
	\le2d_{\mathrm{rk}}\left(T_{e_{3}},f\circ T_{(e_{1},e_{2})}\right)\le6\varepsilon.\qedhere
	\]
\end{proof}

\begin{proof}[Proof of \cref{thm:almost_multilinear_to_group}]
	We start by constructing approximate representations
	of the Dow\-ling groupoid $\mathcal{G}$ associated to $G$. Recall
	that this is a finitely presented category with objects $\left\{ b_{1},b_{2},b_{3}\right\} $
	and with generating morphisms
	\[
	\left\{ g_{s,i,j}:b_{i}\rightarrow b_{j}\mid s\in S,\text{ and }i,j\in\{1,2,3\}\text{ with }i\neq j\right\} .
	\]
	These objects and generating morphisms define a directed graph $H$. Using the notation of \cref{sec:fp_cats,sec:approx_groupoid_reps}, we denote by $\mathcal{C}(H)$ the free category on $H$, so that $\mathcal{G} = \mathcal{C}(H)/\mathord{\sim}$, with $\sim$ the congruence generated by the relations described in \cref{def:Dowling_groupoid}.
	
	Let an $\varepsilon$-approximate vector
	space representation of $M$ be given by the vector space $V$
	and the linear maps $\left\{ T_e:V\rightarrow W_e\right\} _{e\in E}$.
	Suppose the spaces $\left\{ W_e\right\} _{e\in E}$ all have dimension
	$c$, and define $\mathcal{D}$ to be the category
	of vector spaces over the underlying field of
	$V$.
	
	We define a graph homomorphism $f:H\rightarrow\mathrm{Graph}\left(\mathcal{D}\right)$
	as follows. On objects define $f$ by $f\left(b_{i}\right)=W_{b_{i}}$
	for all $1\le i\le3$. To define $f$ on the morphisms, suppose
	$1\le i,j\le3$ and $i$ precedes $j$ in the cyclic ordering. Given
	$s\in S$, choose $6\varepsilon$-approximate determination maps
	\begin{align*}
	\varphi_{s,i,j}:&W_{b_{i}}\oplus W_{s_{i}}\rightarrow W_{b_{j}}\mbox{ and}\\
	\varphi_{s,j,i}:&W_{b_{j}}\oplus W_{s_{i}}\rightarrow W_{b_{i}}
	\end{align*}
	for $T_{b_{j}}$ given $\left\{ T_{b_{i}},T_{s_{i}}\right\} $
	and for $T_{b_{i}}$ given $\left\{ T_{b_{j}},T_{s_{i}}\right\} $,
	respectively. (We take $6\varepsilon$ instead of $\varepsilon$ because we later apply this construction with determination maps produced by \cref{lem:invertible_determination} and obtain additional consequences. For the first part of the proof this choice doesn't matter.) Then define $f\left(g_{s,i,j}\right):W_{b_{i}}\rightarrow W_{b_{j}}$
	by
	\[
	\left(f\left(g_{s,i,j}\right)\right)\left(w\right)=\varphi_{s,i,j}\left(w,0\right)
	\]
	and similarly define
	\[
	\left(f\left(g_{s,j,i}\right)\right)\left(w\right)=\varphi_{s,j,i}\left(w,0\right).
	\]
	
	We now show that for any relation $\varphi_{1}=\varphi_{2}$
	in the presentation of $\mathcal{G}$ we have \[d_\mathrm{rk} \left(f\left(\varphi_{1}\right),f\left(\varphi_{2}\right)\right)\le 20\varepsilon.\]
	
\begin{enumerate}[label=\textbf{Case \arabic*:}, labelwidth=-8mm,  labelindent=3em,leftmargin =!]
		\item For $i$, $j$, and $s$ as above consider the relation $g_{s,j,i}\circ g_{s,i,j}=\mathrm{id}_{b_{i}}$.
		Since $T_{(b_{i},s_{i})}$ is $\varepsilon$-independent, its image intersects
		the subspace $W_{b_{i}}\oplus\left\{ 0\right\} \subset W_{b_{i}}\oplus W_{s_{i}}$
		in a subspace of dimension at least $c\left(1-\varepsilon\right)$. In the
		same way, the image of $T_{(b_{j},s_{i})}$ intersects $W_{b_{j}}\oplus\left\{ 0\right\} $
		in a subspace of dimension at least $c\left(1-\varepsilon \right)$. Define
		\[
		V^{\prime}=T_{s_{i}}^{-1}\left(0\right).
		\]
		The previous considerations imply precisely that $T_{(b_{i},s_{i})}\left(V^{\prime}\right)\simeq T_{b_{i}}\left(V^{\prime}\right)$
		and $T_{(b_{j},s_{i})}\left(V^{\prime}\right)\simeq T_{b_{j}}\left(V^{\prime}\right)$
		have dimension at least $c\left(1-\varepsilon\right)$. It is clear that
		\[
		\mathrm{rk}\left(T_{b_{j}}\restriction_{V^{\prime}}-\varphi_{s,i,j}\circ T_{(b_{i},s_{i})}\restriction_{V^{\prime}}\right)\le\mathrm{rk}\left(T_{b_{j}}-\varphi_{s,i,j}\circ T_{(b_{i},s_{i})}\right)\le 6c\varepsilon
		\]
		(the inequality on the right is from the definition of $\varphi_{s,i,j}$
		as a determination map). In the same way,
		\[
		\mathrm{rk}\left(T_{b_{i}}\restriction_{V^{\prime}}-\varphi_{s,j,i}\circ T_{(b_{j},s_{i})}\restriction_{V^{\prime}}\right)\le 6c\varepsilon.
		\]
		Therefore 
		\[
		V^{\prime\prime}=\left[\ker\left(T_{b_{i}}\restriction_{V^{\prime}}-\varphi_{s,j,i}\circ T_{(b_{j},s_{i})}\restriction_{V^{\prime}}\right)\cap\ker\left(T_{b_{j}}\restriction_{V^{\prime}}-\varphi_{s,i,j}\circ T_{(b_{i},s_{i})}\restriction_{V^{\prime}}\right)\right]
		\]
		is a subspace of $V^{\prime}$ of dimension at least $\dim V^{\prime}-12c\varepsilon$.
		Its image under $T_{b_{i}}$ therefore has codimension at most $12c\varepsilon$
		within the image of $V^{\prime}$, and similarly for $T_{b_{j}}$.
		That is,
		\[
		\dim T_{b_{i}}\left(V^{\prime\prime}\right)\ge c\left(1-13\varepsilon\right)\quad\text{and}\quad\dim T_{b_{j}}\left(V^{\prime\prime}\right)\ge c\left(1-13\varepsilon\right).
		\]
		By definition, if $w\in T_{b_{i}}\left(V^{\prime\prime}\right)$ then
		$w=T_{b_{i}}\left(v\right)$ for some $v\in V^{\prime\prime}$ and
		\[
		\left(f_{n}\left(g_{s,i,j}\right)\right)\left(w\right)=\varphi_{s,i,j}\left(w,0\right)=\varphi_{s,i,j}\left(T_{(b_{i},s_{i})}\left(v\right)\right)=T_{b_{j}}\left(v\right)
		\]
		where the rightmost equality is because $v\in V^{\prime\prime}$ is
		contained in the kernel of $T_{b_{j}}-\varphi_{s,i,j}\circ T_{(b_{i},s_{i})}$.
		In the same way, if $w\in T_{b_{j}}\left(V^{\prime\prime}\right)$
		then $w=T_{b_{j}}\left(v\right)$ for some $v\in V^{\prime\prime}$
		and
		\[
		\left(f_{n}\left(g_{s,j,i}\right)\right)\left(w\right)=T_{b_{i}}\left(v\right).
		\]
		It follows that 
		\[
		f_{n}\left(g_{s,j,i}\right)\circ f_{n}\left(g_{s,i,j}\right)\restriction_{T_{b_{i}}\left(V^{\prime\prime}\right)}=\mathrm{id}_{T_{b_{i}}\left(V^{\prime\prime}\right)},
		\]
		and the normalized rank distance between $f_{n}\left(g_{s,j,i}\right)\circ f_{n}\left(g_{s,i,j}\right)$
		and $\mathrm{id}_{W_{b_{i}}}$ is at most 
		\[
		\frac{1}{c}\left(\dim W_{b_{i}}-\dim T_{b_{i}}\left(V^{\prime\prime}\right)\right)\le 13\varepsilon.
		\]
		\item Suppose $\left(i,j,k\right)$ is an even permutation of $\left(1,2,3\right)$
		(so that $i<j<k<i$ in the cyclic ordering) and let $s,s^{\prime},s^{\prime\prime}\in S$
		such that $s^{\prime\prime}s^{\prime}s=e$ is a relation in $R$.
		We verify that
		\[
		f_{n}\left(g_{s^{\prime\prime},k,i}\right)\circ f_{n}\left(g_{s^{\prime},j,k}\right)\circ f_{n}\left(g_{s,i,j}\right)
		\]
		has small normalized rank distance from $\mathrm{id}_{W_{b_{i}}}$. 
		Define 
		\[
		V^{\prime}=T_{s_{i}}^{-1}\left(0\right)\cap T_{s_{j}^{\prime}}^{-1}\left(0\right)\cap T_{s_{k}^{\prime\prime}}^{-1}\left(0\right)=T_{(s_{i},s_{j}^{\prime},s_{k}^{\prime\prime})}^{-1}\left(0\right).
		\]
		By $\varepsilon$-determination of $T_{s_{k}^{\prime\prime}}$ by $\left\{ T_{s_{i}},T_{s_{j}^{\prime}}\right\} $,
		the map $T_{(s_{i},s_{j}^{\prime},s_{k}^{\prime\prime})}$ has rank
		at most $c\left(2+\varepsilon\right)$. Therefore $V^{\prime}$ is a subspace
		of $V$ with dimension at least 
		\[
		\dim V-\mathrm{rk}\left(T_{(s_{i},s_{j}^{\prime},s_{k}^{\prime\prime})}\right)\ge\dim V-c\left(2+\varepsilon\right).
		\]
		Since $\varphi_{s,i,j}$ is a $6\varepsilon$ determination map we have
		\[
		\mathrm{rk}\left(T_{b_{j}}\restriction_{V^{\prime}}-\varphi_{s,i,j}\circ T_{(b_{i},s_{i})}\restriction_{V^{\prime}}\right)\le\mathrm{rk}\left(T_{b_{j}}-\varphi_{s,i,j}\circ T_{(b_{i},s_{i})}\right)\le 6c\varepsilon,
		\]
		and in the same way also
		\begin{align*}
		\mathrm{rk}\left(T_{b_{k}}\restriction_{V^{\prime}}-\varphi_{s^{\prime},j,k}\circ T_{(b_{j},s_{j}^{\prime})}\restriction_{V^{\prime}}\right)\le &6c\varepsilon\mbox{ and}\\ \mathrm{rk}\left(T_{b_{i}}\restriction_{V^{\prime}}-\varphi_{s^{\prime\prime},k,i}\circ T_{(b_{k},s_{k}^{\prime\prime})}\restriction_{V^{\prime}}\right)\le &6c\varepsilon.
		\end{align*}
		Define
		\[
		\begin{aligned}V^{\prime\prime}= & \ker\left(T_{b_{j}}\restriction_{V^{\prime}}-\varphi_{s,i,j}\circ T_{(b_{i},s_{i})}\restriction_{V^{\prime}}\right)\cap\\
			& \ker\left(T_{b_{k}}\restriction_{V^{\prime}}-\varphi_{s^{\prime},j,k}\circ T_{(b_{j},s_{j}^{\prime})}\restriction_{V^{\prime}}\right)\cap\\
			& \ker\left(T_{b_{i}}\restriction_{V^{\prime}}-\varphi_{s^{\prime\prime},k,i}\circ T_{(b_{k},s_{k}^{\prime\prime})}\restriction_{V^{\prime}}\right).
		\end{aligned}
		\]
		Then $V^{\prime\prime}$ has codimension at most $18c\varepsilon$ within $V^{\prime}$.
		
		We now compute $\dim T_{b_{i}}\left(V^{\prime\prime}\right)$: observe
		that $V^{\prime\prime}\subseteq V^{\prime}\subseteq T_{s_{i}}^{-1}\left(0\right)\cap T_{s_{j}^{\prime}}^{-1}\left(0\right)$.
		Since $\left\{ T_{s_{i}},T_{s_{j}^{\prime}}\right\} $ determine $T_{s_{k}^{\prime\prime}}$
		with error at most $\varepsilon$, the codimension of $V^{\prime}$ within
		$T_{s_{i}}^{-1}\left(0\right)\cap T_{s_{j}^{\prime}}^{-1}\left(0\right)$
		is at most $c\varepsilon$. Thus the codimension of $V^{\prime\prime}$ within
		$T_{s_{i}}^{-1}\left(0\right)\cap T_{s_{j}^{\prime}}^{-1}\left(0\right)$
		is at most $c\left(1+18\right)\varepsilon=19c\varepsilon$.
		Since $\left\{ T_{b_{i}},T_{s_{i}},T_{s_{j}^{\prime}}\right\} $ are
		$\varepsilon$-independent, the image of $T_{(b_{i},s_{i},s_{j}^{\prime})}$
		intersects
		\[
		W_{b_{i}}\oplus\left\{ 0\right\} \oplus\left\{ 0\right\} \subseteq W_{b_{i}}\oplus W_{s_{i}}\oplus W_{s_{j}^{\prime}}
		\]
		in a subspace of codimension at most $c\varepsilon$. This intersection is
		isomorphic to \[T_{b_{i}}\left(T_{s_{i}}^{-1}\left(0\right)\cap T_{s_{j}^{\prime}}^{-1}\left(0\right)\right),\]
		which thus has codimension at most $c\varepsilon$ in $W_{b_{i}}$. It follows
		that $T_{b_{i}}\left(V^{\prime\prime}\right)$ has codimension at
		most $19c\varepsilon+c\varepsilon=20c\varepsilon$ within $W_{b_{i}}$.
		
		Fix $v\in V^{\prime\prime}$. Then $v\in\ker\left(T_{b_{j}}-\varphi_{s,i,j}\circ T_{(b_{i},s_{i})}\right)$,
		and
		\[
		\left(f_{n}\left(g_{s,i,j}\right)\right)\left(T_{b_{i}}\left(v\right)\right)=\varphi_{s,i,j}\left(T_{b_{i}}\left(v\right),0\right)=\varphi_{s,i,j}\left(T_{(b_{i},s_{i})}\left(v\right)\right)=T_{b_{j}}\left(v\right).
		\]
		In the same way we have 
		\[
		\left(f_{n}\left(g_{s^{\prime},j,k}\right)\right)\left(T_{b_{j}}\left(v\right)\right)=T_{b_{k}}\left(v\right)\quad\text{and}\quad\left(f_{n}\left(g_{s^{\prime\prime},k,i}\right)\right)\left(T_{b_{k}}\left(v\right)\right)=T_{b_{i}}\left(v\right).
		\]
		It follows that 
		\[
		f_{n}\left(g_{s^{\prime\prime},k,i}\right)\circ f_{n}\left(g_{s^{\prime},j,k}\right)\circ f_{n}\left(g_{s,i,j}\right)\restriction_{T_{b_{i}}\left(V^{\prime\prime}\right)}=\mathrm{id}_{T_{b_{i}}\left(V^{\prime\prime}\right)},
		\]
		so the normalized rank distance between $f_{n}\left(g_{s^{\prime\prime},k,i}\right)\circ f_{n}\left(g_{s^{\prime},j,k}\right)\circ f_{n}\left(g_{s,i,j}\right)$
		and $\mathrm{id}_{W_{b_{i}}}$ is at most
		\[
		\frac{1}{c}\left[c-\dim T_{b_{i}}\left(V^{\prime\prime}\right)\right]\le 20\varepsilon.
		\]
	\end{enumerate}
	
	This shows that $f$ is a $20\varepsilon$-approximate representation of $\mathcal{G}$.
	
	Let $s,s^{\prime}\in S$
	be distinct generators. By \cref{cor:approx_different_groupoid_elements}, it suffices to find a positive constant lower bound on $d_\mathrm{rk}(f(\varphi_{s,1,2}), f(\varphi_{s',1,2}))$ that holds for all small enough $\varepsilon$: this implies that $\varphi_{s,1,2} \neq \varphi_{s',1,2}$ in $\mathcal{G}$ and hence by the results of \cref{sec:dowling_representations} that $s,s'$ map to distinct elements of $G$ as required.
	
	We apply \cref{lem:invertible_determination} to $W_{b_1}, W_{b_2}, W_{s_1}$ and obtain $6\varepsilon$-determination maps $\psi:W_{b_1} \oplus W_{b_2} \to W_{s_1}$ and $\varphi_{s,1,2}:W_{b_1} \oplus W_{s_1} \to W_{b_2}$ such that with respect to bases for the three vector spaces, $\psi(v_1,v_2) = A_1 v_1 + A_2 v_2$, the matrix $A_2$ is invertible, and $\varphi_{s,1,2}(v_1,v_3) = -A_2^{-1} A_1 v_1 + A_2^{-1}v_3$. Applying the lemma to $W_{b_1}, W_{b_2}, W_{s'_1}$ we obtain similar maps $\psi':W_{b_1} \oplus W_{b_2} \to W_{s'_1}$ and $\varphi_{s',1,2}:W_{b_1} \oplus W_{s'_1} \to W_{b_2}$ with matrices $A'_1, A'_2$. We may assume that the chosen bases for $W_{b_1}$ and $W_{b_2}$ are the same in both applications of the lemma. Note that with respect to our chosen bases,
	\[f(\varphi_{s,1,2}) = -{A_2}^{-1}A_1 \quad \text{and} \quad f(\varphi_{s',1,2}) = -{A'_2}^{-1}A'_1,\]
	and it suffices to find a constant positive lower bound for the normalized rank distance between these two matrices that holds for all small enough $\varepsilon$.
	
	Observe that $\left\{ T_{s_{1}},T_{s_{1}^{\prime}}\right\}$
	are $\varepsilon$-independent, so $\mathrm{rk}(T_{(s_1,s'_1)}) \ge c(2 - \varepsilon)$. Define
	\[F:V \to W_{s_1} \oplus W_{s'_1}\]
	\[F(v) = (\psi\circ T_{(b_1,b_2)}(v), \psi'\circ T_{(b_1,b_2)}(v))\]
	and observe that 
	\[\mathrm{rk}(F - T_{(s_1,s'_1)}) \le \mathrm{rk}(\psi\circ T_{(b_1,b_2)} - T_{s_1}) + \mathrm{rk}(\psi'\circ T_{(b_1,b_2)} - T_{s'_1}) \le 12c\varepsilon.\]
	Therefore $\mathrm{rk}(F) \ge \mathrm{rk}(T_{(s_1,s'_1)}) - 12c\varepsilon \ge c(2 - 13\varepsilon)$. Representing $F$ with respect to the bases chosen in our applications of \cref{lem:invertible_determination}, we obtain the block matrix
	\[\left[\begin{matrix}
		A_1 & A_2 \\
		A'_1 & A'_2
	\end{matrix}\right].\]
	By applying block row operations (multiplying the first row by ${A_2}^{-1}$, the second by ${A'_2}^{-1}$, and then subtracting the second row from the first) we find it has rank equal to the rank of
	\[\left[\begin{matrix}
		{A_2}^{-1}A_1 - {A'_2}^{-1}A'_1 & 0 \\
		{A'_2}^{-1}A'_1 & I
	\end{matrix}\right],\]
	which has rank $c + \mathrm{rk}({A_2}^{-1}A_1 - {A'_2}^{-1}A'_1)$. It follows that $c+\mathrm{rk}({A_2}^{-1}A_1 - {A'_2}^{-1}A'_1) \ge c(2-13\varepsilon)$, or in other words that
	\[d_\mathrm{rk}(-{A_2}^{-1}A_1, -{A'_2}^{-1}A'_1) \ge 1-13\varepsilon,\]
	and if $\varepsilon < \frac{1}{26}$ this is at least $\frac{1}{2}$ as required.
\end{proof}

\subsection{Almost-multilinear matroid representability is undecidable}\label{sec:almost_undecidable}
We put together our tools and show that it is undecidable whether a matroid is almost multilinear: this section is roughly parallel to \cref{sec:entropic_undecidability}. 

Unlike in that section, using the collection of all matroids subordinate to a given partial Dowling geometry is not sufficient here. The (mild) issue is that, unlike for presentations resulting from the scrambling construction, it is possible for some pairs of generators of a group presentation $\langle S \mid R \rangle$ to map to the same element of the group. To handle this possibility we use the following simple lemma.
\begin{lemm}
	Let $\langle S \mid R \rangle$ be a finite presentation of a group and let $\sim$ be an equivalence relation on $S$. Denote by $\langle S/\mathord{\sim} \mid R/\mathord{\sim} \rangle$ the group presentation which is obtained from $\langle S \mid R \rangle$ by replacing $S$ with $S/\mathord{\sim}$, and replacing each letter in each relation in $R$ by its equivalence class in $S/\mathord{\sim}$.
	
	If $\langle S \mid R \rangle$ is symmetric triangular then so is $\langle S/\mathord{\sim} \mid R/\mathord{\sim} \rangle$. There is a group homomorphism
	\[\langle S \mid R \rangle \to \langle S/\sim \mid R/\mathord{\sim} \rangle\]
	that maps each $s\in S$ to its equivalence class $[s]_\sim$.
\end{lemm}

\begin{defn}\label[definition]{def:extended_subordinate}
	Let $\langle S \mid R \rangle$ be a finite group presentation and let $s \in S$. Let $\{\mathord{\sim}_i\}_{i=1}^N$ be the set of all equivalence relations on $S$ satisfying that $s$ is not identified with $e$. The \emph{extended subordinate set} to the pair $(\langle S \mid R \rangle,s)$ is the set of all partial Dowling geometries subordinate to the presentations in~$\langle S/\mathord{\sim}_i \mid R/\mathord{\sim}_i \rangle$.
\end{defn}

\begin{theorem}\label{thm:WP_almost_multilinear}
	Let $G = \langle S\mid R\rangle, s\in S$ be an instance of the word problem. Assume $\langle S \mid R \rangle$ is symmetric triangular, and that $G$ is sofic and torsion-free.
	Let $\mathcal{M}$ be the extended subordinate set of $(\langle S \mid R \rangle,s)$.
	If $s$ is nontrival in $\langle S \mid R \rangle$ then some matroid in $\mathcal{M}$ is almost multilinear.
\end{theorem}
\begin{proof}
	Assume $s$ is nontrivial in $G$. Let $\sim$ be the equivalence relation on $S$ that identifies elements whenever they map to the same element of $G$. This relation is one of the relations $\mathord{\sim}_i$ considered in \cref{def:extended_subordinate}, because it does not identify $s$ with $e$.

	Let $\langle T \mid R_T \rangle$ be a presentation of $G$ such that $S/\mathord{\sim}\subseteq T$, $R/\mathord{\sim}\subseteq R_T$, and $T$ contains an element mapping onto $x\cdot x'$ in $G$ for any $x,x' \in S/\mathord{\sim}$. By \cref{lem:linear_sofic_amplification} for any $\varepsilon > 0$ there is an $n\in\mathbb{N}$ and an $\varepsilon$-approximate representation $\rho: T \to \mathrm{GL}_n(\C)$ of $\langle T \mid R_T \rangle$ such that $d_\mathrm{rk}(x,x') \ge 1-\varepsilon$ whenever $x,x' \in T$ map to distinct elements of $G$.
	
	This implies that for any $\varepsilon > 0$ there exists an $n \in \mathbb{N}$ and an $\varepsilon$-approximate representation $\rho: S/\mathord{\sim} \to \mathrm{GL}_n(\C)$ of $\langle S/\mathord{\sim} \mid R/\mathord{\sim} \rangle$ such that $d_\mathrm{rk}(\rho(x), \rho(x')) \ge 1 - \varepsilon$ for all distinct generators $x,x' \in S/\mathord{\sim}$ and such that if $x,x',x'' \in S/\mathord{\sim}$ is any triple of generators (not necessarily distinct) then $d_\mathrm{rk}(\rho(x'')^{-1}, \rho(x' x))$ is either at most $\varepsilon$ or at least $1 - \varepsilon$. (The statement on pairs follows from the same statement for $\langle T \mid R_T \rangle$. The statement on triples follows by taking the pair $y = x''$, $y' = x' x$ in $T$ for each triple $x,x',x'' \in S$.)
	
	Hence by \cref{thm:epsilon_reps} at least one of the partial Dowling geometries $M$ subordinate to $\langle S/\mathord{\sim}\mid R/\mathord{\sim}\rangle$ is almost multilinear over $\C$. This geometry is in the extended subordinate set of $\langle S \mid R \rangle$.
\end{proof}

\begin{theorem}\label{thm:almost_multilinear_WP}
	Let $G = \langle S\mid R\rangle, s\in S$ be an instance of the word problem and $M$ be its partial Dowling geometry.
	Assume that some matroid of the partial Dowling geometries $\mathcal{M}$ in the extended subordinate set to $(\langle S\mid R\rangle, s)$ is almost multilinear.
	Then $s \neq e$ in $G$.
\end{theorem}
\begin{proof}
	Say the matroid $M\in \mathcal{M}$ is almost multilinear.
	This is a partial Dowling geometry subordinate to a presentation $\langle S/\mathord{\sim} \mid R/\mathord{\sim}\rangle$, where $\sim$ does not identify $s$ and $e$.
	By \cref{thm:almost_multilinear_to_group}, we have $s \neq e$ in $G$ as desired.
\end{proof}

\begin{coro}
	Almost-multilinearity of matroids is undecidable.
\end{coro}
\begin{proof} The preceding two theorems reduce the word problem in a finitely presented sofic group to a finite (computable) sequence of almost-multilinearity problems for matroids. But the former problem is undecidable by \cref{thm:undecidable_sofic}.
\end{proof}

\begin{rmrk}
	One can formulate a problem parallel to conditional independence implication (as in \cref{sec:cii}) in the almost-multilinear setting. We leave the details to the interested reader. Undecidability of conditional rank inequalities in a linear setting already follows from \cite{KY19}; the approximate version would correspond to considering ``stable'' implications, which continue to hold in an approximate sense even when the assumptions hold only in an approximate sense.
\end{rmrk}

\bibliographystyle{myalpha}
\bibliography{matroid.bib}

\end{document}